\documentclass[12pt]{amsart}
\usepackage[margin=1.5in]{geometry}

\usepackage{amssymb, amsmath}
\usepackage{amsthm}
\usepackage{subfigure}
\usepackage{cite}
\usepackage{enumerate}
\usepackage{cases}
\usepackage[font={it}, margin=1cm]{caption}
\usepackage{tikz}
\usetikzlibrary{decorations.markings, decorations.pathreplacing, positioning,arrows, matrix, decorations.pathmorphing, shapes.geometric, calc}
\usetikzlibrary{backgrounds, decorations, calc}
\newtheorem{theorem}{Theorem}[section]
\newtheorem{lemma}[theorem]{Lemma} %%Delete [thm] to re-start numbering
\newtheorem{proposition}[theorem]{Proposition} %%Delete [thm] to re-start
\newtheorem{thmletter}{Theorem}

\newtheorem{corletter}[thmletter]{Corollary}

\newtheorem{corollary}[theorem]{Corollary} %%Delete [thm] to re-start

\newtheorem{question}{Question}

\newtheoremstyle{example}{3pt}{3pt}{}{10pt}{\itshape}{:}{.5em}{}
\theoremstyle{example}

\newcommand{\p}[1]{\noindent {\newline\bf #1.}}

\newcommand{\out}{\operatorname{Out}}
\newcommand{\aut}{\operatorname{Aut}}
\newcommand{\inn}{\operatorname{Inn}}

\renewcommand{\>}{\rangle}
\newcounter{enumerateCounter}

\def\centerarc[#1](#2)(#3:#4:#5)% Syntax: [draw options] (center) (initial angle:final angle:radius)
    { \draw[#1] ($(#2)+({#5*cos(#3)},{#5*sin(#3)})$) arc (#3:#4:#5); }

\usepackage{verbatim}

\title[HNN-extensions of triangle groups]{Every group is the outer automorphism group of an HNN-extension of a fixed triangle group}
\author{Alan D. Logan}
\begin{document}

\maketitle

\begin{abstract}
%We prove a technical theorem which, in a certain sense, writes down a specific subgroup of the outer automorphism group of a particular kind of HNN-extension. We apply this to prove the following result.
Fix an equilateral triangle group $T_i=\<a, b; a^i, b^i, (ab)^i\>$ with $i\geq6$ arbitrary. Our main result is: for every presentation $\mathcal{P}$ of every countable group $Q$ there exists an HNN-extension $T_{\mathcal{P}}$ of $T_i$ such that $\out(T_{\mathcal{P}})\cong Q$. We construct the HNN-extensions explicitly, and examples are given. The class of groups constructed have nice categorical and residual properties. In order to prove our main result we give a method for recognising malnormal subgroups of small cancellation groups, and we introduce the concept of ``malcharacteristic'' subgroups.
%Moreover, the HNN-extension $G_Q$ is an automorphism-induced HNN-extension; such HNN-extensions form a particularly tractable class of HNN-extensions.
%This demonstrates a universal property of triangle groups.
%utilises fibre products of maps to graphs and a small cancellation theory in this setting, which was developed by Wise. This theory of Wise (and its generalisation) is a cornerstone of his recent work which lead to the resolution of the virtually fibering and virtually Haken conjectures
%
%We prove that every countable group $Q$ can be realised as the outer automorphism group of an HNN-extension $G_Q$ of some fixed triangle group, $\out(G_Q)\cong Q$.
\end{abstract}

%%
%Kleene Star stuff:
%Note that I have replaced the Kleene star $S^{\star}$ with the Kleene plus $S^{{+}}$. To undo this, search and replace for {+}\mapsto\star (note the curly brackets).
%The relevant footnote has been commented out.
%%

%\marginpar{\today}
\section{Introduction}
\label{sec:introduction}
%In order to prove the theorem in the title this paper has two main innovations.
%
%
Every group can be realised as the outer automorphism group of some group \cite{matumoto1989any}. One can ask what restrictions can be placed on the groups involved. Several authors have achieved results in this vein \cite{kojima1988isometry} \cite{gobel2000outer} \cite{droste2001all} \cite{braun2003outer} \cite{frigerio2005countable} \cite{logan2015outer} \cite{LoganNonRecursive} \cite{logan2015Bass}.
Notably, Bumagin--Wise proved that every countable group $Q$ can be realised as the outer automorphism group of a finitely generated group $G_Q$ \cite{BumaginWise2005} (here $G_Q$ is the kernel of a short exact sequence inspired by Rips' construction \cite{rips1982subgroups}), and Minasyan proved that $G_Q$ can additionally be taken to have $2$-conjugacy classes and Kazdhan's property T \cite{minasyan2009groups} (here $G_Q$ is one of Osin's monster groups \cite{Osin2010small}).

We therefore say that a class of finitely generated groups $\mathcal{C}$ \emph{possesses the Bumagin--Wise property} if for every countable group $Q$ there exists a group $G_Q\in \mathcal{C}$ such that $\out(G_Q)\cong Q$. Hence,
%Bumagin--Wise proved that the class of finitely generated groups possess the Bumagin--Wise property; in fact, they
Bumagin--Wise prove that a certain class of groups related to Rips' construction possesses the Bumagin--Wise property. Minasyan proved that the class of Osin's monster groups possesses the Bumagin--Wise property. These two classes of groups are well-known to possess pathological properties. Therefore, the results of Bumagin--Wise and Minasyan suggest that the Bumagin--Wise property is a pathological property.

%We define the class of ``automorphism-induced'' HNN-extensions below; it is a class of particularly tractable HNN-extensions (by ``tractable'' we mean that they are an easy class of groups to work with and possess nicer properties than general HNN-extensions). The main result of this paper is Theorem \ref{thm:intro1}, below, which proves that automorphism-induced HNN-extensions of certain triangle groups possess the Bumagin--Wise property. In view of the tractability of these HNN-extensions, we record a quasi-conjecture: the Bumagin--Wise property is not a pathological property. Note that Bumagin--Wise asked if finitely generated, residually finite groups possess the Bumagin--Wise property \cite[Problem 1]{BumaginWise2005}; giving a positive answer to Bumagin--Wise's question could be viewed as a positive answer to our quasi-conjecture.
%, in the sense that we construct the group $G_Q$ as an HNN-extension obtained from any countable presentation of $Q$.

%We define the class of ``automorphism-induced'' HNN-extensions below; it is a class of particularly tractable HNN-extensions
The main result of this paper is Theorem \ref{thm:intro1}, below, which proves that certain classes of ``nice'' HNN-extensions of ``nice'' groups possess the Bumagin--Wise property. Theorem \ref{thm:intro2} then says that this construction is functorial. These results are in stark contrast to the results of Bumagin--Wise and Minasyan, and suggest that the Bumagin--Wise property is not pathological. By a ``nice'' HNN-extension we mean an automorphism-induced HNN-extension; we define these HNN-extensions in Section \ref{sec:AutIndHNN}. By a ``nice'' group we mean a triangle group, defined below; such groups are extremely well studied \cite{Magnus1974Noneuclidean}.

The proof of Theorem \ref{thm:intro1} has two main innovations.
Firstly, we give a method for recognising malnormal subgroups of small cancellation groups
(Theorem \ref{thm:FreeMalnormMetric} says that if the presentation $\langle \mathbf{x}; \mathbf{r}, \mathbf{s}\rangle$ and the set $\mathbf{s}$ possess certain properties then $\langle\mathbf{s}\rangle$ is a malnormal subgroup of $G=\langle \mathbf{x}; \mathbf{r}\rangle$).
%(Theorem \ref{thm:FreeMalnormMetric} proves that the set $\mathbf{s}\subset F(\mathbf{x})$ generates a malnormal subgroup of the group defined by $\langle \mathbf{x}; \mathbf{r}\rangle$ if the presentation $\langle \mathbf{x}; \mathbf{r}, \mathbf{s}\rangle$ is small cancellation).
Secondly, we introduce and study the notion of ``malcharacteristic'' subgroups (which generalise malnormal subgroups). The key idea of this paper is that the existence of malcharacteristic subgroups in certain triangle groups implies Theorem \ref{thm:intro1}.

\p{Main theorems}
A \emph{triangle group} is a group with a presentation of the following form.
\[
T_{i, j, k}:=\langle a, b; a^i, b^j, (ab)^k\rangle
\]
If $i=j=k$ we shall write $T_i:=T_{i, i, i}$ for the corresponding \emph{equilateral triangle group}. Theorem \ref{thm:intro1} reveals a certain universal property possessed by equilateral triangle groups $T_{i}$ with $i\geq6$.

%By a \emph{countable set} we mean a set whose cardinality is either finite or equal to the cardinality of the natural numbers.
%By a \emph{countable group presentation} we mean a presentation $\mathcal{P}=\langle \mathbf{x}; \mathbf{r}\rangle$ where $|\mathbf{x}|\leq|\mathbb{N}|$ (and hence the set $\mathbf{r}$ also has cardinality at most $|\mathbb{N}|$). All our presentations are countable, so we shall often omit the adjective ``countable''. We also assume that each presentation $\mathcal{P}$ contains at least one generator (that is, $|\mathbf{x}|\geq1$); this excludes the empty presentation of the trivial group. For a group presentation $\mathcal{P}$ we write $\pi_1(\mathcal{P})$ for the group defined by $\mathcal{P}$. We often abuse notation and write $Q=\langle \mathbf{x}; \mathbf{r}\rangle$ for $Q$ a group (as we did above with $T_{i, j, k}$). Note that a group $Q$ is countable if and only if there exists a countable presentation $\mathcal{P}$ with $Q\cong\pi_1(\mathcal{P})$.
By a \emph{countable group presentation} we mean a presentation $\mathcal{P}=\langle \mathbf{x}; \mathbf{r}\rangle$ where $|\mathbf{x}|\leq|\mathbb{N}|$ (and hence $|\mathbf{r}|\leq|\mathbb{N}|$). All our presentations are countable so we shall often omit the adjective ``countable''.
%We also assume that each presentation $\mathcal{P}$ contains at least one generator (that is, $|\mathbf{x}|\geq1$); this excludes the empty presentation of the trivial group.
For a group presentation $\mathcal{P}$ we write $\pi_1(\mathcal{P})$ for the group defined by $\mathcal{P}$. We often abuse notation and write $Q=\langle \mathbf{x}; \mathbf{r}\rangle$ for $Q$ a group (as we did above with $T_{i, j, k}$). Note that a group $Q$ is countable if and only if there exists a countable presentation $\mathcal{P}$ with $Q\cong\pi_1(\mathcal{P})$.

\begin{thmletter}
\label{thm:intro1}
Fix an equilateral triangle group $T_{i}$ with $i\geq6$.
%\begin{enumerate}[i.]
%\item
For every countable group presentation $\mathcal{P}$ there exists an automorphism-induced HNN-extension $T_{\mathcal{P}}$ of $T_i$ such that $\out(T_{\mathcal{P}})\cong Q$ and $\aut(T_{\mathcal{P}})\cong T_{\mathcal{P}}\rtimes Q$, where $Q:=\pi_1(\mathcal{P})$.
%\item\label{intro1:functorial}
%If $P$ and $Q$ are countable groups with a surjection $P\twoheadrightarrow Q$ then the group $G_P$ can be chosen such that there exists a surjection $G_P\twoheadrightarrow G_Q$.
%\end{enumerate}
\end{thmletter}

We construct the groups $T_{\mathcal{P}}$ from Theorem \ref{thm:intro1} in Section \ref{sec:Construction}.
We discuss why we require $i\geq6$ in Theorem \ref{thm:intro1} after the proof of Lemma \ref{lem:malcharTriange}.

In Section \ref{sec:Construction} we study properties of the construction underlying Theorem \ref{thm:intro1}. This analysis leads to Theorems \ref{thm:intro2}, \ref{thm:intro3} and \ref{thm:intro4}.
Firstly, in Section \ref{sec:Construction} we note that countable group presentations form a category $\operatorname{Pres}$, with morphisms $\mathcal{P}_1\rightarrow \mathcal{P}_2$ corresponding to certain surjective homomorphisms of the groups $\pi_1(\mathcal{P}_1)$ and $\pi_1(\mathcal{P}_2)$.
Theorem \ref{thm:intro2} then says that the construction of Theorem \ref{thm:intro1} is functorial.%; the map $\mathcal{Q}\mapsto T_{\mathcal{Q}}$ is a functor from the category of group presentations to $\operatorname{Grp}$, the category of groups.
%Theorem \ref{thm:intro2} is written in more general terms, using the category of quotient groups, so we first define this category.
%: take $P=F(\mathbf{x})$ in the following definition to obtain the category of $|\mathbf{x}|$-generator group presentations.
%It is not necessarily clear what we mean by the category of $n$-generator group presentations. The theorem is
%Note that if $P$ is taken to be a free group $F(\mathbf{x})$ then this category can be viewed as the category of $|\mathbf{x}|$-generator presentations $F(\mathbf{x})/N_{\mathbf{r}}=\langle \mathbf{x}; \mathbf{r}\rangle$.

\begin{thmletter}
\label{thm:intro2}
The map defined by $\mathcal{P}\mapsto T_{\mathcal{P}}$ is a functor from the category of countable group presentations $\operatorname{Pres}$ to the category of groups $\operatorname{Grp}$.
%Fix a presentation $\mathcal{P}$ of a group $P$.
%Denote by $\mathbf{H_{\mathcal{P}}}$ the class consisting of the groups $H_{\mathcal{P}_i}$ constructed by Theorem \ref{thm:intro1}, where
%Let $\mathcal{P}_k$ be a quotient presentation of $\mathcal{P}$ and let $N_k$ denote the kernel of the induced quotient map $\pi_1(\mathcal{P})\twoheadrightarrow\pi_1(\mathcal{P}_k)$. Then the map $\mathbf{P}\rightarrow\operatorname{Grp}$ given by $P/N_k\mapsto T_{\mathcal{P}_k}$ is a functor.
\end{thmletter}

Theorem \ref{thm:intro2} follows from Theorem \ref{thm:FunctorialProperties}. We note after the proof of Theorem \ref{thm:FunctorialProperties} that if an appropriate subcategory of $\operatorname{Pres}$ is chosen then this functor may be extended to a functor from a subcategory of $\operatorname{Grp}$ to $\operatorname{Grp}$.

%\p{Residual properties}
The construction of Theorem \ref{thm:intro1} possesses the following residual property.

\begin{thmletter}
\label{thm:intro3}
Let $\mathbb{T}_{\operatorname{Fin}}$ be the class of groups $T_{\mathcal{Q}}$ where $\pi_1(\mathcal{Q})$ is finite. If $\pi_1(\mathcal{P})$ is residually finite then $T_{\mathcal{P}}$ is residually-$\mathbb{T}_{\operatorname{Fin}}$.
\end{thmletter}

Theorem \ref{thm:intro3} follows from Theorem \ref{thm:ResidualProperties}. In Section \ref{sec:RF} we prove the following corollary of Theorem \ref{thm:intro3}. Bumagin--Wise asked if every countable group $Q$ can be realised as the outer automorphism group of a finitely generated, residually finite group $G_Q$ \cite[Problem 1]{BumaginWise2005}. Corollary \ref{corol:intro3} gives a positive answer to this question of Bumagin--Wise for all finitely generated, residually finite groups $Q$ by taking $Q:=\pi_1(\mathcal{P})$ and $G_Q:=T_{\mathcal{P}}$.

\begin{corletter}
\label{corol:intro3}
If the presentation $\mathcal{P}$ has a finite generating set and the group $\pi_1(\mathcal{P})$ is residually finite then the group $T_{\mathcal{P}}$ is residually finite.
\end{corletter}

The construction of Theorem \ref{thm:intro1} has a certain flexibility. It allows a choice of certain subgroups $M_n$, and this choice may be made in such a way that we obtain easy examples or additional properties of the groups $T_{\mathcal{P}}$. The above theorems do not require these subgroups to be chosen in any particular way. However, the subgroups $M_n$ may be chosen in such a way that the following result holds, where for presentations $\mathcal{P}=\langle \mathbf{x}; \mathbf{r}\rangle$ and $\mathcal{Q}=\langle \mathbf{y}; \mathbf{s}\rangle$ we write $\mathcal{P}\ast\mathcal{Q}$ for the presentation $\langle\mathbf{x}, \mathbf{y}; \mathbf{r}, \mathbf{s}\rangle$ (so $\pi_1(\mathcal{P}\ast\mathcal{Q})\cong\pi_1(\mathcal{P})\ast\pi_1(\mathcal{Q})$).

\begin{thmletter}
\label{thm:intro4}
The subgroups $M_n$, $n\in\mathbb{N}$, in the proof of Theorem \ref{thm:intro1} may be chosen in such a way that for every presentation $\mathcal{P}$ with finite generating set and for every countable group presentation $\mathcal{Q}$ there exists a surjection $T_{\mathcal{P}}\twoheadrightarrow T_{\mathcal{P}\ast\mathcal{Q}}$.
\end{thmletter}

Theorem \ref{thm:intro4} follows from Theorem \ref{thm:FreeProduct}. The chosen subgroups $M_n$ are specified in Theorem \ref{thm:FreeProduct}.

\p{Examples}
We now give some simple examples of our construction. These examples illustrate the simplicity of the construction itself (examples of Bumagin--Wise's groups and of Minasyan's groups are much harder to construct) as well as the construction's functorial properties. Let $\mathcal{P}_{\infty}=\langle z; -\rangle$ and $\mathcal{P}_k=\langle z; z^k\rangle$ for $k\in\mathbb{N}$ (we will always assume $0\not\in\mathbb{N}$). Theorem \ref{thm:intro1} then outputs the following groups: Fix $i\geq6$. By $\phi$ we mean the automorphism $\phi: a\mapsto b, b\mapsto (ab)^{-1}$ of $T_i=\langle a, b; a^i, b^i, (ab)^i\rangle$ which has order three. By $x$ and $y$ we mean specific words over $a$ and $b$ given at the end of Section \ref{sec:Malchar} (see also Lemma \ref{lem:malcharTriange}).
\begin{align*}
T_{\mathcal{P}_{\infty}}&=\langle T_i, t; ty^{-j}xy^jt^{-1}=\phi(y^{-j}xy^j), \forall~j\in\mathbb{Z}\rangle\\
T_{\mathcal{P}_k}&=\langle T_i, t; ty^kt^{-1}=\phi(y^k), ty^{-j}xy^jt^{-1}=\phi(y^{-j}xy^j), \forall~0\leq j<k\rangle
\end{align*}
Note that the words $y^{-j}xy^j$, $j\in\mathbb{Z}$, generate the kernel of the map $F(x, y)\rightarrow\pi_1(\mathcal{P}_{\infty})$, $x\mapsto 1, y\mapsto z$, while the words $y^k$ and $y^{-j}xy^j$, $0\leq j<k$, generate the kernel of the analogous map $F(x, y)\rightarrow\pi_1(\mathcal{P}_k)$.
It is clear that the map $\mathbb{Z}/k\mathbb{Z}\mapsto T_{\mathcal{P}_k}$, $k\in\mathbb{N}\cup\{\infty\}$, is a functor from a subcategory of $\operatorname{Grp}$ to $\operatorname{Grp}$.

%\p{Generalising Theorem \ref{thm:intro1}}
\p{Triangle groups}
In Theorems \ref{thm:GeneralTriangle1}--\ref{thm:GeneralTriangle4} we obtain results analogous to Theorems \ref{thm:intro1}, \ref{thm:intro2}, \ref{thm:intro3} and \ref{thm:intro4} where the base group $H$ is any triangle group $T_{i, j, k}$, $i, j, k\geq6$. The price we pay in allowing for more base groups is that the ``$\out(T_{\mathcal{P}})\cong Q$'' condition of Theorem \ref{thm:intro1} is weakened to ``$Q$ embeds with index one or two in $\out(T_{\mathcal{P}})$''.

The constructions of Theorem \ref{thm:intro1} and Theorem \ref{thm:GeneralTriangle1} are specific cases of Theorem \ref{thm:mainconstruction}. This more general theorem potentially allows for other base groups $H$ to be taken in Theorem \ref{thm:intro1}, and not just triangle groups.
However, it is a highly non-trivial task to apply Theorem \ref{thm:mainconstruction}, and so our proofs of Theorem \ref{thm:intro1} and Theorem \ref{thm:GeneralTriangle1} use properties of triangle groups throughout. We focus on triangle groups because they are well-studied groups with many nice properties. It would be extremely surprising if the universal properties of Theorem \ref{thm:intro1} and Theorem \ref{thm:GeneralTriangle1} were to disappear if we replaced the groups $H$ with less well-behaved groups (for example Thompson's groups $F$, $T$ or $V$). We therefore record a quasi-conjecture: If a group $H$ has a suitably rich subgroup structure (e.g. contains a non-abelian free subgroup) then the class $\mathcal{C}_H$ of HNN-extensions of $H$ possesses the Bumagin--Wise property.

%the standard examples of badly-behaved HNN-extensions are Baumslag--Solitar groups, that is, HNN-extensions of the infinite cyclic group. However, if a Baumslag--Solitar group is {automorphism-induced} then it is ``nice'': it is residually finite and has virtually cyclic outer automorphism group. The constructions in Sections \ref{section5} are surprising because they show that automorphism-induced HNN-extensions can still be ``wild'', even when the base group is well-behaved. In particular, we prove that an automorphism-induced HNN-extension is not necessarily residually finite, and can have arbitrary outer automorphism group.

%We show that an underlying reason for these wild properties is that if $H$ is a hyperbolic triangle group then for every countable group $Q$ there exists some subgroup $K$ of $H$ such that $Q\cong N_H(K)/K$, where $N_H(K)$ is the normaliser of $K$ in $H$. This is proved in Section \ref{sec:mainconstruction}. In an automorphism-induced HNN-extension $G=\langle H, t; K^t=K^{\prime}\rangle$, the quotient $N_H(K)/K$ embeds into $\aut(G)$ and, under certain conditions, into $\out(G)$. Thus the properties of $N_H(K)/K$ are in a certain sense bestowed upon $G$. However, if $H$ is cyclic (and hence $G$ is a Baumslag--Solitar group) then $N_H(K)/K$ is necessarily cyclic and so $G$ inherits no pathological properties from this subgroup quotient.

%\cite{logan2015Bass}, for a further discussion on why pathological properties are surprising for such constructions

\p{Malcharacteristic subgroups}
This paper introduces malcharacteristic subgroups, and most of the paper is devoted to their study.
Malnormal subgroups are central objects in geometric group theory, and malcharacteristic subgroups generalise malnormal subgroups in the same way that characteristic subgroups generalise normal subgroups.

We formally define malcharacteristic subgroups in Section \ref{sec:Malchar} (the definition of malnormal subgroups can be found there too). Our focus on malcharacteristic subgroups is because of Theorem \ref{thm:mainconstruction}, which can be interpreted as saying ``Theorem \ref{thm:intro1} holds if each group $T_i$, $i\geq6$, contains a malcharacteristic subgroup which is free of rank two''.

In Proposition \ref{lem:malcharFreeArbRank} we provide concrete examples of malcharacteristic subgroups in the free group $F(a, b)$. We use these to provide, in Proposition \ref{lem:malcharTriangleArbRank}, concrete examples of malcharacteristic subgroups in the triangle groups $T_{i, j, k}$ with $i, j, k\geq6$. Theorem \ref{thm:intro1} then follows from Theorem \ref{thm:mainconstruction}.

%Our proof of the existence of malcharacteristic subgroups is constructive: we first construct a malcharacteristic subgroup $M$ of $T_{i, j, k}$ which is freely generated by two words $x,y\in F(a, b)$, and the other subgroups are constructed using $M$. If any of $i$, $j$ or $k$ are less than or equal to $8$ then it is not clear that the subgroup $M$ is malcharacteristic, and so the construction fails. This failure is due to our choice of words $x$ and $y$ rather than to some theory we are applying. In particular, the underlying arguments reply on the $C^{\prime}(1/4)-T(4)$ small cancellation condition (and not the stronger $C^{\prime}(1/8)$ condition suggested by the requirement of $i, j, k>8$). Lemma \ref{lem:transversal} suggests we probably require $i, j, k\geq6$. Therefore, choosing different words $x$ and $y$ may allow the proof to be extended to the groups $T_{i,j ,k}$ with $i, j, k\geq6$.

\p{Malnormal subgroups of small cancellation groups}
In order to obtain concrete examples of malcharacteristic subgroups we first need concrete examples of malnormal subgroups.
%To find malnormal subgroups of the groups $T_{i, j, k}$, $i, j, k\geq6$,
To find these we apply Theorem \ref{thm:FreeMalnormMetric}, stated below. Theorem \ref{thm:FreeMalnormMetric} gives a new method for recognising malnormal subgroups of small cancellation groups.
%We refer the reader to Lyndon--Schupp's text for the definitions of the small cancellation conditions \cite[Section V]{L-S}.
We refer the reader to Lyndon--Schupp's classic text for the definitions of the metric $C^{\prime}(1/6)$ and $C^{\prime}(1/4)-T(4)$ small cancellation conditions \cite[Section V]{L-S}.
%; these conditions imply that the underlying group is hyperbolic.
%Note that if $i, j, k>6$ then the group $T_{i, j, k}$ has the $C^{\prime}(1/6)$ condition, and if $i,j ,k>4$ then $T_{i, j, k}$ has the $C^{\prime}(1/4)-T(4)$ condition.
%
%Note that in Theorem \ref{thm:FreeMalnormMetric}, the group $\langle \mathbf{x}; \mathbf{r}\rangle$ has the relevant small cancellation condition. Therefore, the theorem states that ``relatively'' small cancellation subgroups of small cancellation groups are free and malnormal.
%
A word $W$ in $F(\mathbf{x})$ is a \emph{proper power} if there exists another word $U\in F(\mathbf{x})$ and an integer $n>1$ such that $W=U^n$. For a group $G=\langle\mathbf{x}; \mathbf{r}\rangle$, Theorem \ref{thm:FreeMalnormMetric} links the malnormality of a subgroup $M$ generated by a set of words $\mathbf{s}$, so $M=\langle \mathbf{s}\rangle$, with properties of the presentation $\langle \mathbf{x}; \mathbf{r}, \mathbf{s}\rangle$. Note that Theorem \ref{thm:FreeMalnormMetric} allows for the explicit construction of malnormal subgroups of small cancellation groups.

\begin{thmletter}
\label{thm:FreeMalnormMetric}
Suppose that $\langle \mathbf{x};\mathbf{r}, \mathbf{s}\rangle$ is a $C^{\prime}(1/6)$ or $C^{\prime}(1/4)-T(4)$ small cancellation presentation, and that no element of the set $\mathbf{s}$ is a proper power. Then the subgroup $\langle \mathbf{s}\rangle$ of the small cancellation group $G=\langle \mathbf{x}; \mathbf{r}\rangle$ is malnormal and free with basis $\mathbf{s}$.
\end{thmletter}

The assumption that no element of the set $\mathbf{s}$ is a proper power is necessary, as if $U^n\in\mathbf{s}$ with $n>1$ then $U\not\in\langle\mathbf{s}\rangle$ (by small cancellation) but $\langle U^n\rangle\leq\left(U^{-1}\langle\mathbf{s}\rangle U\cap\langle\mathbf{s}\rangle\right)$.

Theorem \ref{thm:FreeMalnormMetric} follows from two more general theorems, Theorems \ref{thm:FreeMalnormMetric6} and \ref{thm:FreeMalnormMetric4T4}. These two theorems give a new method for recognising, for $A$ and $B$ subgroups of a small cancellation group $G$, if $A^g\cap B= 1$ for all $g\in G$.
%if there does not exists any $g\in G$ such that $L^g\cap M\neq 1$ for $L$ and $M$ subgroups of a small cancellation group $G$.
Also related is Lemma \ref{lem:FreeSbgps}, which gives a new method for recognising free subgroups of small cancellation groups.
Note that fibre products of Stalling's graphs \cite[Theorem 5.5]{stallings1983topology} solve for free groups the underlying decision problems of Theorem \ref{thm:FreeMalnormMetric} and of Theorems \ref{thm:FreeMalnormMetric6} and \ref{thm:FreeMalnormMetric4T4}.

\p{Outline of the paper}
%Theorem \ref{thm:mainconstruction} (this gives the construction underlying Theorem \ref{thm:intro1})
In Section \ref{sec:AutIndHNN} we prove Lemma \ref{lem:corolofmaintheorem}, a result on automorphism-induced HNN-extensions which underpins the construction of Theorem \ref{thm:mainconstruction}.
In Section \ref{sec:Malchar} we define, and present basic results about, malcharacteristic subgroups.
In Section \ref{sec:Construction} we prove Theorem \ref{thm:mainconstruction}; this gives the construction underlying Theorem \ref{thm:intro1}. This section also contains the proofs of various properties of this construction relating to Theorems \ref{thm:intro2}, \ref{thm:intro3} and \ref{thm:intro4}.
In Section \ref{sec:MalcharF2} we give concrete examples of malcharacteristic subgroups of the free group $F(a, b)$.
In Section \ref{sec:MalnormAspher} we give concrete examples of malnormal subgroups of triangle groups $T_{i, j, k}$, $i, j, k\geq6$, and in particular we prove Theorem \ref{thm:FreeMalnormMetric}.
In Section \ref{sec:MalcharTriangle} we combine the results of the previous sections to give concrete examples of malcharacteristic subgroups of triangle groups $T_{i, j, k}$, $i, j, k\geq6$, and in particular we prove Theorems \ref{thm:intro1}, \ref{thm:intro2}, \ref{thm:intro3} and \ref{thm:intro4}.
In Section \ref{sec:RF} we apply Theorem \ref{thm:intro3} to the question of Bumagin--Wise regarding residually finite groups, and in particular we prove Corollary \ref{corol:intro3}.
In Section \ref{sec:Qns} we pose certain questions which arose in the writing of this paper.

\p{Acknowledgements} I would like to thank Stephen J. Pride and Tara Brendle for many helpful discussions about this paper. I would also like to thank Jim Howie for observing when the groups in Theorem \ref{thm:intro1} are finitely and infinitely presentable (see the discussion following the proof of Theorem \ref{thm:mainconstruction}), Peter Kropholler for ideas which led to Theorem \ref{thm:intro3}, and Henry Wilton for pointing out that triangle groups are LERF (see the proof of Corollary \ref{corol:intro3}). This research was supported by the EPSRC standard grant of Laura Ciobanu, EP/R035814/1.

%%%%%%%%%%%%%%%%%%%%%%%%%%%%%%%%%%%%%%%
%--------------------------------Results from Bass-Jiang paper--------------------------------
%%%%%%%%%%%%%%%%%%%%%%%%%%%%%%%%%%%%%%%

\section{Automorphism-induced HNN-extensions}
\label{sec:AutIndHNN}

The main result of this section is Lemma \ref{lem:corolofmaintheorem}, which gives certain isomorphisms underpinning Theorem \ref{thm:intro1}.
We also define automorphism-induced HNN-extensions and explain why they are ``nice''.
%Previous results of the author allow for a description of the outer automorphism groups of automorphism-induced HNN-extensions \cite{logan2015Bass}; these results are summarised in Theorem \ref{thm:description}. Lemma \ref{lem:corolofmaintheorem} follows from Theorem \ref{thm:description}.

\p{Automorphism-induced HNN-extensions}
%The HNN-extensions $H_{\mathcal{Q}}$ of $T_i$ and of $T_{i, j, k}$ in Theorems \ref{thm:intro1} and \ref{thm:GeneralTriangle1} are built using a framework for the construction of tractable groups with pathological outer automorphism groups, introduced in \cite{logan2015Bass}.
Let $H$ be a group, $K\lneq H$ a proper subgroup and $\phi\in\aut(H)$ an automorphism.\footnote{If $H=K$ then $G=H\rtimes_{\phi}\mathbb{Z}$ is the mapping torus of $\phi$. Here the outer automorphism group is very different because $\phi$ extends to an inner automorphism.}
Then the \emph{automorphism-induced HNN-extension of $H$ with associated subgroup $K$ and with associated automorphism $\phi$} is the group with relative presentation
\[
H\ast_{(K, \phi)}=\langle H, t; tkt^{-1}=\phi(k), k\in K\rangle.
\]
The groups $T_{\mathcal{P}}$ in Theorem \ref{thm:intro1} are automorphism-induced HNN-extensions. The fact that they are automorphism-induced is crucial to our paper. For example, to prove Theorem \ref{thm:intro2} we apply the following fact particular to automorphism-induced HNN-extensions: if $K<K_0< H$ then $H\ast_{(K, \phi)}\twoheadrightarrow H\ast_{(K_0, \phi)}$.

Automorphism-induced HNN-extensions are ``nice'', in the sense that they are an easy class of groups to work with and possess nicer properties than general HNN-extensions.
%In general, an HNN-extension may exhibit a variety of pathological properties.
The simplest examples of HNN-extensions are Baumslag--Solitar groups (HNN-extensions of the infinite cyclic group), and these also provide the standard examples of badly-behaved HNN-extensions. For example, Baumslag--Solitar groups can have non-finitely generated outer automorphism group \cite{collins1983automorphisms}. However, if a Baumslag--Solitar group is {automorphism-induced} then it has virtually cyclic outer automorphism group \cite{GHMR}.
%The construction underlying Theorem \ref{thm:intro1} is surprising because it shows that automorphism-induced HNN-extensions can still be ``wild'', even when the base group is well-behaved.

As another example of the tractability of automorphism-induced HNN-extensions, we note that the notation $H\ast_{(K, \phi)}$ emphasises that the second associated subgroup $\phi(K)$ can, in practice, be ignored when studying these groups. This in demonstrated in Theorem \ref{thm:description} and Lemma \ref{lem:corolofmaintheorem}, where the descriptions of $\out(H\ast_{(K, \phi)})$ mention $K$ but not $\phi(K)$.
%For further discussion, including examples, on the nice properties of automorphism-induced HNN-extensions, see \cite{logan2015Bass}.

\p{Describing \boldmath{$\out(G)$}}
%Recall the definition of an automorphism-induced HNN-extension $G=H\ast_{(K, \phi)}$ from the introduction.
For $G=H\ast_{(K, \phi)}$, Theorem \ref{thm:description} now describes $\out(G)$ subject to certain conditions.
One of the conditions required by Theorem \ref{thm:description} is that the group $H$ has \emph{Serre's property FA}, that is, every action of $H$ on any tree has a global fixed point.
%For finitely generated groups, this is equivalent to not mapping onto the infinite cyclic group and not splitting non-trivially as an HNN-extension or free product with amalgamation \cite[Theorem 15]{trees}.
This is an extremely natural property (see, for example, \cite[Theorem 15]{trees}).
For $\delta\in\aut(H)$ we write $\widehat{\delta}$ to mean the outer automorphism $\delta\inn(H)\in\out(H)$. For $g\in H$ we write $\gamma_g$ for the inner automorphism of $H$ acting as conjugation by $g$, so $\gamma_g(h)=g^{-1}hg$ for all $h\in H$.
%
%The results we apply relate to the ``Bass-Jiang group'' of an automorphism-induced HNN-extension $G=H\ast_{(K, \phi)}$. This is a subgroup of $\out(G)$ which arises naturally in the context of Bass-Serre theory. As opposed to giving the definition here, we refer the reader to our previous paper and give two conditions which imply that the Bass-Jiang group is the full outer automorphism group, which are conditions (\ref{item:thmdescription1}) and (\ref{item:thmdescription2}) in our statement of Theorem \ref{thm:decription}.
%One of the conditions required by Theorem \ref{thm:description} is that the group $H$ has \emph{Serre's property FA}, that is, every action of $H$ on any tree has a global fixed point. For finitely generated groups, this is equivalent to not mapping onto the infinite cyclic group and not splitting non-trivially as an HNN-extension or free product with amalgamation \cite{trees}.
%Note that triangle groups $T_{i, j, k}$ have Serre's property FA \cite[Example 6.3.5]{trees}.
We define the subgroup $A_{(K, \phi)}$ of $\aut(H)$ as follows:
\[
A_{(K, \phi)}:=\{\delta\in\aut(H): \delta(K)=K, \:\exists\: a\in H \textnormal{ s.t. } \delta\phi(k)=\phi\delta\gamma_a(k) \:\forall\: k\in K\}
\]
Theorem \ref{thm:description} follows immediately from two results from our previous paper \cite[Theorems 4.7 \& 6.2]{logan2015Bass}.

\begin{theorem}\label{thm:description}
Let $G=H\ast_{(K, \phi)}$ be an automorphism-induced $HNN$-extension. Assume that:
\begin{enumerate}
\item\label{item:thmdescription1} $H$ has Serre's property FA;
\item\label{item:thmdescription2} there does not exist any $b\in H$ such that $\phi(K)\lneq \gamma_b(K)$ and there does not exist any $c\in H$ such that $\gamma_c(K)\lneq \phi(K)$; and
\item\label{item:thmdescription3} $Z(H)=1$.
\end{enumerate}
Then there exists a short exact sequence:
\[
1\rightarrow C_H(K)\rtimes\frac{N_H(K)}{K} \rightarrow \out^0(G)\rightarrow \frac{A_{(K, \phi)}\inn(H)}{\inn(H)}\rightarrow 1
\]
where either $\out^0(G)=\out(G)$ or there exists some $\delta\in\aut(H)$ and some $a\in H$ such that $\delta(K)=\phi(K)$, $\delta^2\gamma_a(K)=K$ and $\delta\phi^{-1}(k)=\phi\delta\gamma_a(k)$ for all $k\in K$, whence $\out^0(G)$ has index two in $\out(G)$.
\end{theorem}

The assumption that $Z(H)=1$ means that the short exact sequence here is easily-digestible. A more extensive description of this sequence in our previous paper requires no such assumption \cite[Theorem 5.6]{logan2015Bass}.

%\p{Refining the description of \boldmath{$\out(G)$}}
Lemma \ref{lem:corolofmaintheorem} now refines the description of $\out(G)$ given by Theorem \ref{thm:description}.
%These conditions imply that $\out(G)\cong N_H(K)/K$ and that $\aut(G)= \inn(G)\rtimes\out(G)$. The fact that the automorphism group splits in this way also follows from our previous paper.
We shall often use $\out^0(G)$ to mean an index-one or-two subgroup of $\out(G)$. Then by $\aut^0(G)$ we shall mean the full pre-image of $\out^0(G)$ in $\aut(G)$ under the natural map (see Lemmas \ref{lem:corolofmaintheorem} and \ref{lem:underlying}, and Theorem \ref{thm:mainconstruction}).

\begin{lemma}
\label{lem:corolofmaintheorem}
Suppose $H$, $K$ and $\phi$ are such that the following hold:
\begin{enumerate}
\item\label{corolcondition1} $H$ has Serre's property FA;
\item\label{corolcondition2} $C_H(K)$ is trivial;
\item\label{corolcondition3} $\delta(K)\cap K=1$ for all automorphisms $\delta\not\in \operatorname{Inn}(H)$; and
\item\label{corolcondition4} $\phi\not\in\operatorname{Inn}(H)$.
\setcounter{enumerateCounter}{\value{enumi}}
\end{enumerate}
Let $G=H\ast_{(K, \phi)}$. Then there exists an index-one or-two subgroup $\out^0(G)$ of $\out(G)$ such that $\out^0(G)\cong N_H(K)/K$.
Moreover, $\aut^0(G)= \inn(G)\rtimes\out^0(G)$, and so $\aut^0(G)\cong G\rtimes N_H(K)/K$. Suppose, in addition, that the following holds:
\begin{enumerate}
\setcounter{enumi}{\value{enumerateCounter}}
\item\label{corolcondition5} $\widehat{\phi}$ has odd or infinite order in $\out(H)$.
\end{enumerate}
Then $\out^0(G)=\out(G)$ and $\aut^0(G)=\aut(G)$.
\end{lemma}

%\marginpar{Ask Tara about enumerate lists and referencing them. (1) or (2.1.1) and (2.2.1)?}
\begin{proof}
%The condition Theorem \ref{thm:description}.\ref{item:thmdescription1} holds by Condition~(\ref{corolcondition1}), while Theorem \ref{thm:description}.\ref{item:thmdescription3} holds by Condition~(\ref{corolcondition2}).
%Conditions~(\ref{corolcondition3}) and (\ref{corolcondition4}) imply that $\phi(K)\cap \gamma_h(K)=1$ for all $h\in H$; hence, Theorem \ref{thm:description}.\ref{item:thmdescription2} holds.
The three conditions of Theorem \ref{thm:description} hold here, with the second condition following from Conditions (\ref{corolcondition3}) and (\ref{corolcondition4}) here as Conditions (\ref{corolcondition3}) and (\ref{corolcondition4}) imply that $\phi(K)\cap \gamma_h(K)=1$ for all $h\in H$.
Hence, the description of $\out^0(G)$ given by Theorem \ref{thm:description} holds here. It is therefore sufficient to prove that $A_{(K, \phi)}\leq\inn(H)$, that $\aut^0(G)=\inn(G)\rtimes\out^0(G)$, and that if Condition (\ref{corolcondition5}) holds then $\out^0(G)=\out(G)$.

Condition (\ref{corolcondition3}) implies that $\delta(K)\neq K$ for all $\delta\not\in\inn(H)$, and so $A_{(K, \phi)}\leq\operatorname{Inn}(G)$ as required.

That $\aut^0(G)= \inn(G)\rtimes\out^0(G)$ follows from the fact that $A_{(K, \phi)}\leq\operatorname{Inn}(H)$ \cite[Lemma 6.4]{logan2015Bass}.

%To do this it is sufficient to prove that there does not exist a pair $(\delta, g)$ with $\delta\in\operatorname{Aut}(H)$ and $g\in H$ such that $\delta(K)=\phi(K)$, $\delta^2\gamma_g(K)=K$ and $\delta\phi^{-1}(k)=\phi\delta\gamma_g(k)$ for all $k\in K$.
Suppose that Condition (\ref{corolcondition5}) holds and that $\out^0(G)\neq\out(G)$. Then there exists $\delta\in\operatorname{Aut}(H)$ and $a\in H$ such that $\delta\phi^{-1}(K)=K=\delta^2\gamma_a(K)$.
%Now, $\delta^2\gamma_g(K)=K$ implies that $\delta^2\in\inn(H)$, by (\ref{corolcondition3}), and so either $\delta\in\inn(H)$ or $\widehat{\delta}$ has order two in $\out(H)$. On the other hand, $\delta\phi^{-1}(K)=K$ implies that $\widehat{\delta}=\widehat{\phi}$, again by (\ref{corolcondition3}), and so $\widehat{\delta}$ has odd order in $\out(H)$ by (\ref{corolcondition5}). Therefore, $\delta\in\inn(H)$, and so as $\widehat{\delta}=\widehat{\phi}$ we have that $\phi\in\inn(H)$, which contradicts (\ref{corolcondition4}). Hence, such a pair $(\delta, g)$ cannot exist, and so $\out^0(G)=\out(G)$ as required.
Now, $\delta\phi^{-1}(K)=K$ implies that $\widehat{\delta}=\widehat{\phi}$, by (\ref{corolcondition3}). Therefore, $\delta^2\not\in\inn(G)$, by (\ref{corolcondition4}) and (\ref{corolcondition5}).
On the other hand, $\delta^2\gamma_a(K)=K$ implies that $\delta^2\in\inn(H)$, by (\ref{corolcondition3}), a contradiction. Hence, $\out^0(G)=\out(G)$ as required.
\end{proof}

%%%%%%%%%%%%%%%%%%%%%%%%%%%%%%%%%%%%%%%
%--------------------------------The construction--------------------------------
%%%%%%%%%%%%%%%%%%%%%%%%%%%%%%%%%%%%%%%

\section{Malcharacteristic subgroups}
\label{sec:Malchar}
In this section we define malcharacteristic subgroups, and present basic results about them. We also state a subgroup $M$ of $T_{i, j, k}$ which is malcharacteristic and free of rank two.
Theorem \ref{thm:mainconstruction} reduces the proof of Theorem \ref{thm:intro1} to proving that this subgroup $M$ is indeed malcharacteristic and free of rank two; the purpose of Sections \ref{sec:MalcharF2}--\ref{sec:MalcharTriangle} is to prove these two properties of $M$.

\p{Definition of malcharacteristic subgroups}
A malnormal subgroup intersects its conjugates trivially apart from in the obvious place in the obvious way. More formally, a subgroup $M\leq H$ is \emph{malnormal} in $H$ if the following implication holds, where $h^g:=g^{-1}hg$ and $M^g:=\{h^g; h\in M\}$.
\[
M^g\cap M\neq1\Rightarrow g\in M
\]
Similarly, a {malcharacteristic subgroup} intersects its automorphic orbit trivially apart from in the obvious place in the obvious way.
More formally, a subgroup $M\leq H$ is \emph{malcharacteristic} in $H$ if for all $\delta\in\aut(H)$ the following implication holds, where $\inn(M):=\{\gamma_h; h\in M\}$.
\[
\delta(M)\cap M\neq1\Rightarrow\delta\in\operatorname{Inn}(M)
\]
Malcharacteristic subgroups may also be characterised as follows. It is characterisation (\ref{list:thirdCharacterisation}) of the following lemma which we apply in this paper (see Lemmas \ref{lem:malnormal} and \ref{lem:malcharTriange}, or the proof of Lemma \ref{lem:malcharFree}).

\begin{lemma}
\label{lem:MalcharCharacterisation}
Let $M$ be a subgroup of $H$. The following are equivalent.
\begin{enumerate}
\item
$M$ is malcharacteristic in $H$.
\item
$\inn(M)$ is malnormal in $\inn(H)$.
\item\label{list:thirdCharacterisation}
$M$ is malnormal in $H$ and the following implication holds:
\[\delta(M)\cap M\neq1\Rightarrow\delta\in\operatorname{Inn}(H).\]
\end{enumerate}
\end{lemma}
The proof of Lemma \ref{lem:MalcharCharacterisation} is routine and so omitted.

\p{Basic results}
We now make two observations which are applied in, respectively, Lemma \ref{lem:malcharArbRank} and Section \ref{sec:Construction}. The proofs of these observations are routine and so omitted.
Lemma \ref{lem:malnormLEQmalchar} is dual to the well-known result that if $N$ is normal in $H$ and $C$ is characteristic in $N$ then $C$ is normal in $H$.

\begin{lemma}
\label{lem:malnormLEQmalchar}
If $M$ is malcharacteristic in $H$ and $A$ is malnormal in $M$ then $A$ is malcharacteristic in $H$.
\end{lemma}

Note that we may replace the word ``malnormal'' in Lemma \ref{lem:malcharNormaliser} with the word ``malcharacteristic''.

\begin{lemma}
\label{lem:malcharNormaliser}
If $M$ is malnormal in $H$ and $K\neq1$ is a normal subgroup of $M$ then $N_H(K)=M$.
\end{lemma}

\p{Free, malcharacteristic subgroups}
Lemma \ref{lem:malcharArbRank} constructs free, malcharacteristic subgroups from a single ``seed'' subgroup $M$. This lemma is central to Theorem \ref{thm:mainconstruction}. We use $\infty$ to denote the first infinite cardinal $|\mathbb{N}|$.

\begin{lemma}
\label{lem:malcharArbRank}
If a group $H$ contains a malcharacteristic subgroup $M$ which is free of rank $m>1$ (possibly $m=\infty$) then for each $n\in\mathbb{N}\cup\{\infty\}$ the group $H$ contains a malcharacteristic subgroup $M_n$ which is free of rank $n$.
\end{lemma}

\begin{proof}
Suppose $M=\langle \mathbf{x}\rangle$ is a malcharacteristic subgroup of $H$ which is free on the set $\mathbf{x}$, where $|\mathbf{x}|=m$. Let $x_1, x_2\in \mathbf{x}$ and consider $A=\langle x_1, x_2\rangle$ (clearly $A$ is free of rank two). Now, $A$ is malnormal in $M$ and so $A$ is malcharacteristic in $H$, by Lemma \ref{lem:malnormLEQmalchar}. Therefore, again by Lemma \ref{lem:malnormLEQmalchar}, it is sufficient to prove that $A\cong F(x_1, x_2)$, the free group of rank two, contains malnormal subgroups of arbitrary rank (possibly countably infinite).

To see that $F(x_1, x_2)$ contains malcharacteristic subgroups of arbitrary rank $n\in\mathbb{N}\cup\{\infty\}$, first note that for each such $n$ there exists a $C^{\prime}(1/6)$ small cancellation set $\mathbf{s}_n\subset F(x_1, x_2)$ which has cardinality $n$ and where no element of $\mathbf{s}_n$ is a proper power \cite[Lemma 6]{BumaginWise2005}. If $n<\infty$ then $M_n:=\langle\mathbf{s}_n\rangle$ is a malnormal subgroup of $F(x_1, x_2)$ which is free of rank $n$ \cite[Theorems~2.11~\&~2.14]{wise2001residual}.
%\marginpar{\cite[Theorem 2.14]{wise2001residual} proves malnormal and $\pi_1$-injective?}
If $n=\infty$ then denote by $y_i$ for $i\in\mathbb{N}$ the elements of $\mathbf{s}_{\infty}$, so $\mathbf{s}_{\infty}=\{y_1, y_2, \ldots\}$, write $M_i=\langle y_1, \ldots, y_i\rangle$, and write $M_{\infty}:=\langle\mathbf{s}_{\infty}\rangle$. As $M_{\infty}$ is the union of the chain of subgroups $M_1<M_2<\ldots$,
any identity of the form
%$g_0=1$ or $g_1^w=g_2$, with $g_i\in M_{\infty}$ and $w\in L$,
$U_0(\mathbf{s}_{\infty})=1$ or $W(x_1, x_2)^{-1}U_1(\mathbf{s}_{\infty})W(x_1, x_2)=U_2(\mathbf{s}_{\infty})$
which holds in $M_{\infty}$ also holds in $M_i$ for some $i\in\mathbb{N}$.
%only contain finitely many elements of $\mathbf{s}_{\infty}^{\pm1}$ (as they have finite length). Hence, for each such identity there exists some $n\in\mathbb{N}$ such that $U_0\in\langle \mathbf{s}_n\rangle$ or $U_1, U_2\in\langle \mathbf{s}_n\rangle$.
Therefore, as each $M_i$ with $i\in\mathbb{N}$ is free and malnormal in $F(x_1, x_2)$ it follows that $M_{\infty}$ is a malnormal subgroup of $F(x_1, x_2)$ which is free of countably-infinite rank.
\end{proof}

\p{A malcharacteristic subgroup}
Let $i, j, k\geq6$. We record here that the subgroup $M=\langle x, y\rangle$, $x$ and $y$ as below with $\rho\gg\max(i,j,k)$, is a malcharacteristic subgroup of $T_{i, j, k}=\langle a, b; a^i, b^j, (ab)^k\rangle$ which is free of rank two.
\begin{align*}
x&:=(ab^{-1})^{3}(a^2b^{-1})^{3}(ab^{-1})^3(a^2b^{-1})^4\ldots (ab^{-1})^3(a^2b^{-1})^{\rho+2}\\
y&:=(ab^{-1})^3(a^2b^{-1})^{\rho+3}(ab^{-1})^3(a^2b^{-1})^{\rho+4}\ldots (ab^{-1})^{3}(a^2b^{-1})^{2\rho+2}
\end{align*}
This fact is proven in Lemma \ref{lem:malcharTriange}.

\section{The construction underlying Theorem \ref{thm:intro1}}
\label{sec:Construction}
In this section we prove Theorem \ref{thm:mainconstruction}, which proves Theorem \ref{thm:intro1} modulo the existence of a free, malcharacteristic subgroup $M$ of rank two in the group $T_{i}$ for each $i\geq6$.
We also prove Theorems \ref{thm:FunctorialProperties}--\ref{thm:FreeProduct}, which are the results underlying Theorems \ref{thm:intro2}, \ref{thm:intro3} and \ref{thm:intro4}.
%(In Section \ref{sec:Malchar} we stated such a subgroup $M$ which we claimed satisfies the required properties.) In particular, this section constructs the groups $H_{\mathcal{Q}}$ from Theorems \ref{thm:intro1} and \ref{thm:GeneralTriangle1}.
%The remainder of this paper proves that the stated subgroup $M$ possesses the required properties.

%Group presentations are central to this section. Therefore, for the sake of concreteness we fix for this section an alphabet $\mathbf{y}_{\infty}=\{y_1, y_2, \ldots\}$ and we shall assume that a presentation has the form either $\langle y_1, y_2, \ldots, y_n; \mathbf{r}\rangle$ or $\langle \mathbf{y}_{\infty}; \mathbf{r}\rangle$.

\p{The construction}
In the construction of Theorem \ref{thm:mainconstruction} we simply specify the subgroups $M_n$ and $K$ in Lemma \ref{lem:underlying}.
%Note that choosing different subgroups $M$ may yield similar constructions with the same universal property of Theorem \ref{thm:mainconstruction} but with different.
Recall that by $\aut^0(G)$ we mean the full pre-image of $\out^0(G)$ in $\aut(G)$ under the natural map.

\begin{lemma}
\label{lem:underlying}
Suppose that $H$ and $\phi\in\aut(H)$ are such that Conditions (\ref{corolcondition1}) and (\ref{corolcondition4}) of Lemma \ref{lem:corolofmaintheorem} hold.
%\begin{enumerate}
%\item the group $H$ has Serre's property FA;
%\item $\phi\not\in\out(H)$; and
%\item $M_n$ is a malcharacteristic subgroup of $H$ which is free of rank $n>1$.
%\end{itemize}
Suppose that $M_n$ is a malcharacteristic subgroup of $H$ which is free of rank $n>1$, and $K\neq1$ is a normal subgroup of $M_n$.

Let $G=H\ast_{(K, \phi)}$. Then there exists an index-one or-two subgroup $\out^0(G)$ of $\out(G)$ such that $\out^0(G)\cong M_n/K$. Moreover, $\aut^0(G)= \inn(G)\rtimes\out^0(G)$, and so $\aut^0(G)\cong G\rtimes M_n/K$.

Suppose, in addition, that $\phi$ is such that Condition (\ref{corolcondition5}) of Lemma \ref{lem:corolofmaintheorem} holds. Then $\out^0(G)=\out(G)$ and $\aut^0(G)=\aut(G)$.
\end{lemma}

\begin{proof}
By Lemma \ref{lem:malcharNormaliser}, $N_H(K)=M_{n}$. Therefore, to prove the result it is sufficient to prove that $H$, $K$ and $\phi$ satisfy Conditions (\ref{corolcondition2}) and (\ref{corolcondition3}) of Lemma \ref{lem:corolofmaintheorem}.

%Conditions (\ref{corolcondition1}) and (\ref{corolcondition4}) are satisfied by assumption.
The subgroup $M_n$ is malcharacteristic in $H$, so (\ref{corolcondition3}) holds.
We now prove that $C_H(K)$ is trivial, so (\ref{corolcondition2}) holds. Suppose that $g\in H$ is such that $[k, g]=1$ for all $k\in K$. Then $g\in M_n$, by malnormality of $M_n$, and so $C_H(K)\leq M_n$. As $M_n$ is free we have that $K$ is cyclic. Now, non-trivial normal subgroups of non-cyclic free groups can never be cyclic. Thus, as $M_n$ is free of rank $n>1$ we conclude that $C_H(K)$ is trivial, as required.%, and so $C_H(K)$ is trivial.
%
%The final, additional condition that $\widehat{\phi}$ has odd or infinite order in $\out(H)$ is precisely Condition (\ref{corolcondition5}) of Lemma \ref{lem:corolofmaintheorem}. Hence, $\out^0(H\ast_{(K, \phi)})=\out(H\ast_{(K, \phi)})$ and $\aut^0(H\ast_{(K, \phi)})=\aut(H\ast_{(K, \phi)})$ as required.
\end{proof}

If $\mathbf{r}\subseteq F(\mathbf{x})$ then we write $\langle\langle\mathbf{r}\rangle\rangle$ for the normal closure of the set $\mathbf{r}$ in $F(\mathbf{x})$, so $F(\mathbf{x})/\langle\langle\mathbf{r}\rangle\rangle=\pi_1(\langle\mathbf{x}; \mathbf{r}\rangle)$.
We say that $\langle\mathbf{y}; \mathbf{s}\rangle$ is a \emph{quotient presentation} of $\langle\mathbf{x}; \mathbf{r}\rangle$ if $\mathbf{x}=\mathbf{y}$ and $\langle\langle\mathbf{r}\rangle\rangle\lneq \langle\langle\mathbf{s}\rangle\rangle$.
We now have the main result of this section.
%Recall that if $\mathcal{P}=\langle\mathbf{x}; \mathbf{r}\rangle$ and $\mathcal{Q}=\langle\mathbf{y}; \mathbf{s}\rangle$ are countable group presentations then $\mathcal{P}\sim\mathcal{Q}$ if there exists an isomorphism $\delta: F(\mathbf{x})\rightarrow F(\mathbf{y})$ such that $\langle\langle\delta(\mathbf{r})\rangle\rangle=\langle\langle\mathbf{s}\rangle\rangle$.
%\p{The construction underlying Theorems \ref{thm:intro1} and \ref{thm:GeneralTriangle1}}

\begin{theorem}
\label{thm:mainconstruction}
Suppose that the group $H$ has
\begin{enumerate}
\item\label{item:mainconstruction1} Serre's property FA;
\item\label{item:mainconstruction2} non-trivial outer automorphism group; and
\item\label{item:mainconstruction3} a malcharacteristic subgroup which is free of rank two.
\setcounter{enumerateCounter}{\value{enumi}}
\end{enumerate}
Then:
\begin{enumerate}[i.]
\item\label{point:mainconstruction1}
For every countable group presentation $\mathcal{P}$ there exists an automorphism-induced HNN-extension $H_{\mathcal{P}}$ of $H$ such that $\out^0(H_{\mathcal{P}})\cong \pi_1(\mathcal{P})$, where $\out^0(H_{\mathcal{P}})$ is an index-one or-two subgroup of $\out(H_{\mathcal{P}})$. Moreover, $\aut^0(H_{\mathcal{P}})= \inn(H_{\mathcal{P}})\rtimes\out^0(H_{\mathcal{P}})$, and so $\aut^0(H_{\mathcal{P}})\cong H_{\mathcal{P}}\rtimes \pi_1(\mathcal{P})$.
%, where $\aut^0(H_{\mathcal{P}})$ is the full pre-image of $\out^0(H_{\mathcal{P}})$ from the natural map.
%\setcounter{enumerateCounter2}{\value{enumi}}
%\end{enumerate}
%Let $\mathcal{P}$, $\mathcal{Q}$ be countable group presentations with at lease one generator. Then:
%\begin{enumerate}[i.]
%\setcounter{enumi}{\value{enumerateCounter2}}
%\item\label{point:mainconstruction2}
%If $\mathcal{P}\sim\mathcal{Q}$ then $T_{\mathcal{P}}=T_{\mathcal{Q}}$.
%$H_{\mathcal{P}}\cong H_{\mathcal{Q}}$ if and only if $\mathcal{P}\sim\mathcal{Q}$.
\item\label{point:mainconstruction3}
For $\mathcal{P}_1$ and $\mathcal{P}_2$ countable group presentations, if $\mathcal{P}_2$ is a quotient presentation of $\mathcal{P}_1$ then there is a surjection $H_{\mathcal{P}_1}\twoheadrightarrow H_{\mathcal{P}_2}$.
\end{enumerate}
Suppose, in addition, that the group $H$ has
\begin{enumerate}
\setcounter{enumi}{\value{enumerateCounter}}
\item\label{item:mainconstruction4} a non-inner automorphism $\phi$ such that $\widehat{\phi}\in\out(H)$ has odd or infinite order.
\end{enumerate}
Then $\out^0(H_{\mathcal{P}})=\out(H_{\mathcal{P}})$ and $\aut^0(H_{\mathcal{P}})=\aut(H_{\mathcal{P}})$.
\end{theorem}

\begin{proof}
By Lemma \ref{lem:malcharArbRank}, and as $H$ contains a malcharacteristic subgroup which is free of rank two, the group $H$ contains malcharacteristic subgroups $M_n$ which are free of rank $n$ for any given $n\in\mathbb{N}\cup\{\infty\}$. Fix for each $n\in\mathbb{N}\cup\{\infty\}$ such a subgroup $M_n$.

Let $\mathcal{P}=\langle \mathbf{x}; \mathbf{r}\rangle$ be a given presentation. Let $H_{\mathcal{P}}$ be the automorphism-induced HNN-extension $H\ast_{(K_{\mathcal{P}}, \phi)}$ of $H$ where $\phi\in\aut(H)$ is a non-inner automorphism (with $\widehat{\phi}$ of odd or infinite order if such an automorphism exists) and where $K_{\mathcal{P}}$ is as follows:
Consider the presentation $\widehat{\mathcal{P}}=\langle {\mathbf{x}}, p, q; {\mathbf{r}}, p, q\rangle$, $p, q\not\in \mathbf{x}^{\pm1}$, and we shall write $\widehat{\mathbf{x}}=\mathbf{x}\sqcup\{p, q\}$ and $\widehat{\mathbf{r}}=\mathbf{r}\sqcup\{p, q\}$ (so $\widehat{\mathcal{P}}=\langle\widehat{\mathbf{x}}; \widehat{\mathbf{r}}\rangle$). Note that $\pi_1(\mathcal{P})\cong\pi_1(\widehat{\mathcal{P}})$, that $|\widehat{\mathbf{x}}|>1$, and that $\widehat{\mathbf{r}}$ is non-empty. Choose $M_{|\widehat{\mathbf{x}}|}$ to be malcharacteristic in $H$ of rank $|\widehat{\mathbf{x}}|$ and take $K_{\mathcal{P}}$ to be the normal subgroup of $M_{|\widehat{\mathbf{x}}|}$ associated with $\langle\langle\widehat{\mathbf{r}}\rangle\rangle$, the normal closure in $F(\widehat{\mathbf{x}})$ of $\widehat{\mathbf{r}}$.

By Lemma \ref{lem:underlying}, this construction proves Point (\ref{point:mainconstruction1}) of the theorem. Also by Lemma \ref{lem:underlying}, if there exists $\widehat{\phi}\in\out(H)$ which has odd or infinite order then $\out^0(H_{\mathcal{P}})=\out(H_{\mathcal{P}})$ and $\aut^0(H_{\mathcal{P}})=\aut(H_{\mathcal{P}})$.

%To see that this construction proves Point~(\ref{point:mainconstruction2}) of the theorem, note that the subgroup $M_n$ is dependent only on the cardinality of the generating set $\mathbf{x}$.

%To see that this construction proves Point~(\ref{point:mainconstruction2}) of the theorem, first note that if $\mathcal{P}\sim\mathcal{Q}$ then $K_{\mathcal{P}}=K_{\mathcal{Q}}$ and hence $H\ast_{(K_{\mathcal{P}}, \phi)}\cong H\ast_{(K_{\mathcal{Q}}, \phi)}$. For the opposite direction, XXX.

To see that this construction proves Point (\ref{point:mainconstruction3}) of the theorem, note that as $\mathcal{P}_2=\langle\mathbf{x}; \mathbf{s}\rangle$ is a quotient presentation of $\mathcal{P}_2=\langle \mathbf{x}; \mathbf{r}\rangle$ we have that $\langle\langle \mathbf{r}, p, q\rangle\rangle\lneq\langle\langle \mathbf{r}, \mathbf{s}, p, q\rangle\rangle\leq F(\mathbf{x}, p, q)$. Therefore,
%Then there are presentations $\mathcal{P}=\langle \mathbf{x}; \mathbf{r}\rangle$, $|\mathbf{x}|>1$ and $\mathbf{r}$ non-empty, and $\mathcal{P}=\langle \mathbf{x}; \mathbf{r}, \mathbf{s}\rangle$ which define, respectively, the groups $P$ and $Q$. Construct $H_{\mathcal{P}}$ and $H_{\mathcal{Q}}$ using these presentations. Then there is a surjection $H_{\mathcal{P}}\twoheadrightarrow H_{\mathcal{Q}}$ which corresponds to
adding the relators $tUt^{-1}\phi(U)^{-1}$ to $H_{\mathcal{Q}}$, where $U\in M_{|\widehat{\mathbf{x}}|}$ corresponds to an element of $\langle\langle\mathbf{r}, \mathbf{s}, p, q\rangle\rangle\setminus\langle\langle\mathbf{r}, p, q\rangle\rangle$, induces the required group homomorphism $H_{\mathcal{P}_1}\twoheadrightarrow H_{\mathcal{P}_2}$.
\end{proof}

%end complete proof

Note that because the subgroup $K_{\mathcal{P}}$ in the proof of Theorem \ref{thm:mainconstruction} is free, the presentation of the HNN-extension $H_{\mathcal{P}}=H\ast_{(K_{\mathcal{P}}, \phi)}$ is aspherical, and so minimal \cite[Theorem 3.1]{chiswell1981aspherical}. Thus, for $\mathcal{P}=\langle \mathbf{x}; \mathbf{r}\rangle$, the group $H_{\mathcal{P}}$ in the construction is finitely presented if and only if the set $\mathbf{x}$ and the group $\pi_1(\mathcal{P})\cong\out(H_{\mathcal{P}})$ are both finite.

\p{Proving Theorem \ref{thm:intro1}}
To obtain the groups $T_{\mathcal{P}}$ from Theorem \ref{thm:intro1}, take in Theorem \ref{thm:mainconstruction} the group $H$ to be $T_i$, the automorphism $\phi$ to be $a\mapsto b$, $b\mapsto b^{-1}a^{-1}$, and the malcharacteristic subgroup in Condition (\ref{item:mainconstruction3}) to be the subgroup $M:=\langle x, y\rangle$ stated in Section \ref{sec:Malchar}, above.
%Theorem \ref{thm:GeneralTriangle1} is a related construction; to obtain the groups $T_{\mathcal{Q}}^{i, j, k}$ from Theorem \ref{thm:GeneralTriangle1}, take in the proof of Theorem \ref{thm:mainconstruction} the group $H$ to be $T_{i, j, k}$, the automorphism $\phi$ to be $a\mapsto a^{-1}$, $b\mapsto b^{-1}$, and the malcharacteristic subgroup to be the subgroup $M:=\langle x, y\rangle$.

Theorem \ref{thm:GeneralTriangle1} is a related construction, and also follows from Theorem \ref{thm:mainconstruction}. Indeed, Theorem \ref{thm:mainconstruction} proves Theorem \ref{thm:intro1} and Theorem \ref{thm:GeneralTriangle1} modulo proving that the subgroup $M=\langle x, y\rangle$ of $T_{i, j, k}$, $i, j, k\geq6$, is malcharacteristic and free of rank two; that is, modulo proving that Condition (\ref{item:mainconstruction3}) of Theorem \ref{thm:mainconstruction} holds for these triangle groups. This is because the groups $T_{i, j, k}$ are well-known to possess Conditions (\ref{item:mainconstruction1}) and (\ref{item:mainconstruction2}), while the groups $T_i$ additionally possess Condition (\ref{item:mainconstruction4}) (the map $\phi: a\mapsto b$, $b\mapsto b^{-1}a^{-1}$ is non-inner of order three). %Therefore, Theorem \ref{thm:mainconstruction} can be interpreted as saying that if these triangle groups $T_i$ each possess a malcharacteristic subgroup which is free of rank two then Theorem \ref{thm:intro1} holds.
%Note that the only triangle groups $T_{i, j, k}$ with $i, j, k>8$ which have an outer automorphism of odd order are the equilateral triangle groups $T_i$ (and $\out(T_{i, j, k})$ is always finite).
%Equilateral triangle groups are required because
Note that non-equilateral triangle groups have outer automorphism groups of order two or four (see Lemma \ref{lem:transversal}), so Condition (\ref{item:mainconstruction4}) does not hold for the triangle groups of Theorem \ref{thm:GeneralTriangle1}.
%so Theorem \ref{thm:GeneralTriangle1} is the best which we can do by applying Theorem \ref{thm:mainconstruction}.
%stated Section \ref{sec:Malchar}, above.
Theorem \ref{thm:intro1} and Theorem \ref{thm:GeneralTriangle1} are formally proven in Section \ref{sec:MalcharTriangle}.

%Note that a non-uniform version of this construction holds if $\mathcal{P}$ has an infinite generating set. That is, fixing $\mathcal{P}$ and given a presentations $\mathcal{Q_1}$ we can construct the group $H_{\mathcal{P}\ast\mathcal{Q}_1}$ which can be used to obtain a second group $H^1_{\mathcal{P}}$ such that $H^1_{\mathcal{P}}\twoheadrightarrow H^1_{\mathcal{P}\ast\mathcal{Q}_1}$, and such that Point \ref{point:mainconstruction1} of Theorem \ref{thm:mainconstruction} holds for $H_{\mathcal{P}}$. However, given another presentation $\mathcal{Q_2}$ the corresponding group $H^2_{\mathcal{P}}$ with $H^2_{\mathcal{P}}\twoheadrightarrow H_{\mathcal{P}\ast\mathcal{Q}_2}$
%
%and $\mathcal{Q}_2$

\p{Functorial properties}
Countable group presentations form a category in the following sense:
%If $\mathbf{r}\subseteq F(\mathbf{x})$ then we write $\langle\langle\mathbf{r}\rangle\rangle$ for the normal closure of the set $\mathbf{r}$ in $F(\mathbf{x})$, so $F(\mathbf{x})/\langle\langle\mathbf{r}\rangle\rangle=\pi_1(\langle\mathbf{x}; \mathbf{r}\rangle)$.
We define the equivalence relation $\sim$ on presentations as $\langle \mathbf{x}; \mathbf{r}\rangle\sim\langle \mathbf{y}; \mathbf{s}\rangle$ if $\mathbf{x}=\mathbf{y}$ and $\langle\langle\mathbf{r}\rangle\rangle=\langle\langle\mathbf{s}\rangle\rangle$.
%Form the category $\operatorname{Pres}^{\prime}$ whose objects are countable group presentations, with morphisms $\mathcal{P}\rightarrow \mathcal{Q}$ if $\mathcal{Q}$ is a quotient presentation of $\mathcal{P}$.
%The \emph{category of countable group presentations $\operatorname{Pres}$} is the category $\operatorname{Pres}^{\prime}/\sim$.
Form the category $\operatorname{Pres}_{\mathbf{x}}$ whose objects are group presentations with generating set $\mathbf{x}$, under the equivalence relation $\sim$ and with morphisms $\mathcal{P}_j\rightarrow \mathcal{P}_k$ if $\mathcal{P}_k$ is a quotient presentation of $\mathcal{P}_j$. Note that the normal subgroups of $F(\mathbf{x})$ form a category, with morphisms obtained from the usual subgroup partial order, and this category is naturally isomorphic to the category $\operatorname{Pres}_{\mathbf{x}}$. Note also that the categories $\operatorname{Pres}_{\mathbf{x}}$ and $\operatorname{Pres}_{\mathbf{y}}$ are isomorphic if $|\mathbf{x}|=|\mathbf{y}|$. Therefore, there exists a well-defined \emph{category of $n$-generator presentations $\operatorname{Pres}_{n}$}, $n\in\mathbb{N}\cup\{\infty\}$. The \emph{category of countable group presentations $\operatorname{Pres}$} is the union of the categories $\operatorname{Pres}_{n}$ for $n\in\mathbb{N}\cup\{\infty\}$.
%However, for the sake of concreteness we consider the categories $\operatorname{Pres}_{\{y_1, \ldots, y_n\}}$, $n\in\mathbb{N}$, and $\operatorname{Pres}_{\mathbf{y}_{\infty}}$ rather than the abstract categories $\operatorname{Pres}_n$ and $\operatorname{Pres}_{\infty}$.
%
Theorem \ref{thm:FunctorialProperties} now says that the construction of Theorem \ref{thm:mainconstruction} is functorial.%; the map $\mathcal{Q}\mapsto T_{\mathcal{Q}}$ is a functor from the category of group presentations to $\operatorname{Grp}$, the category of groups.
%Theorem \ref{thm:intro2} is written in more general terms, using the category of quotient groups, so we first define this category.
%: take $P=F(\mathbf{x})$ in the following definition to obtain the category of $|\mathbf{x}|$-generator group presentations.
%It is not necessarily clear what we mean by the category of $n$-generator group presentations. The theorem is
%Note that if $P$ is taken to be a free group $F(\mathbf{x})$ then this category can be viewed as the category of $|\mathbf{x}|$-generator presentations $F(\mathbf{x})/N_{\mathbf{r}}=\langle \mathbf{x}; \mathbf{r}\rangle$.

\begin{theorem}
\label{thm:FunctorialProperties}
The map defined by $\mathcal{P}\mapsto H_{\mathcal{P}}$ is a functor from the category of countable group presentations $\operatorname{Pres}$ to the category of groups $\operatorname{Grp}$.
%Fix a presentation $\mathcal{P}$ of a group $P$.
%Denote by $\mathbf{H_{\mathcal{P}}}$ the class consisting of the groups $H_{\mathcal{P}_i}$ constructed by Theorem \ref{thm:intro1}, where
%Let $\mathcal{P}_k$ be a quotient presentation of $\mathcal{P}$ and let $N_k$ denote the kernel of the induced quotient map $\pi_1(\mathcal{P})\twoheadrightarrow\pi_1(\mathcal{P}_k)$. Then the map $\mathbf{P}\rightarrow\operatorname{Grp}$ given by $P/N_k\mapsto T_{\mathcal{P}_k}$ is a functor.
\end{theorem}

\begin{proof}
If there exists a morphism $\mathcal{P}_j\rightarrow\mathcal{P}_k$ then $\mathcal{P}_k$ is a quotient presentation of $\mathcal{P}_j$. Therefore, $H_{\mathcal{P}_j}\twoheadrightarrow H_{\mathcal{P}_k}$ by Theorem \ref{thm:mainconstruction}.\ref{point:mainconstruction3}.
\end{proof}

Let $Q$ be a fixed group and let $\mathbf{Q}$ be the category whose objects are quotient groups $Q/N_k$ and whose morphisms correspond to inclusion of kernels, so $Q/N_j\rightarrow Q/N_k$ if $N_j\lneq N_k$. Fixing a presentation of $\mathcal{P}$ of $Q$, the category $\mathbf{Q}$ is isomorphic to the subcategory of $\operatorname{Pres}$ with initial object $\mathcal{P}$. Hence, the functor of Theorem \ref{thm:intro2} also encodes information about quotient groups.

Write $\mathcal{P}_k$ for the image of $Q/N_k$ under the aforementioned isomorphism of categories. Suppose that the group $Q$ is such that $Q/N_j\cong Q/N_k$ implies $N_j=N_k$. Then $\mathbf{Q}$ is isomorphic to a subcategory of $\operatorname{Grp}$ and hence the map $Q/N_k\mapsto T_{\mathcal{P}_k}$ is a functor from a subcategory of $\operatorname{Grp}$ to $\operatorname{Grp}$.

%$\mapsto T_{\mathcal{P}_k}$ is functor from a subcategory of $\operatorname{Grp}$ to $\operatorname{Grp}$.
On the other hand, it is impossible to choose for each group $Q$ a presentation $\mathcal{P}$ with $\pi_1(\mathcal{P})=Q$ such that the map $Q\mapsto T_{\mathcal{P}}$, factoring as $Q\mapsto\mathcal{P}\mapsto T_{\mathcal{P}}$, is functorial; this is because there exist non-isomorphic groups $Q_1$ and $Q_2$ such that $Q_1\twoheadrightarrow Q_2\twoheadrightarrow Q_1$. This is why Theorem \ref{thm:FunctorialProperties} is about a functor from $\operatorname{Pres}$ to $\operatorname{Grp}$ rather than from $\operatorname{Grp}$ to $\operatorname{Grp}$. Indeed, $Q_1$ and $Q_2$ can be taken to be $2$-generator, 1-relator groups \cite{Borshchev2006isomorphism}, so if we fix $n>1$ and restrict to all finitely presentable $n$-generator groups then the map $Q\mapsto T_{\mathcal{P}}$ can still not be made functorial. The examples from the introduction show that if $n=1$ then this map can be made functorial.

%Theorem \ref{thm:intro2} is about group presentations rather than groups themselves because there is no way of choosing group presentations $\mathcal{Q}$ such that the map $Q\mapsto \mathcal{Q}$ is a functor because there exist non-isomorphic groups $P$ and $Q$ such that $P\twoheadrightarrow Q\twoheadrightarrow P$ \cite{Borshchev2006isomorphism}. Hence, the map $Q\mapsto H_{\mathcal{Q}}$ cannot be made functorial.

\p{Residual Properties}
%Note that $\mathbb{H}_{\operatorname{Fin}}$ is the class of \emph{countable} presentations of finite groups.
Let $\mathbb{H}_{\operatorname{Fin}}$ be the class of groups $H_{\mathcal{P}}$ where $\pi_1(\mathcal{P})$ is finite.
Note that these presentations $\mathcal{P}$ are countable presentations of finite groups. In particular, if the input presentation $\mathcal{P}=\langle \mathbf{x}; \mathbf{r}\rangle$ in Theorem \ref{thm:ResidualProperties} has infinite generating set, so $|\mathbf{x}|=\infty$, then the presentations $\mathcal{P}_g=\langle \mathbf{x}; \mathbf{r}, \mathbf{s}\rangle$ in the following proof are presentations of finite groups but with infinite generating set $\mathbf{x}$. It follows that if $|\mathbf{x}|=\infty$ then the associated subgroups $K_{\mathcal{P}_g}$ cannot be finitely generated.
If $U$ and $V$ are words then we write $U\leq V$ to mean that $U$ is a subword of $V$.

\begin{theorem}
\label{thm:ResidualProperties}
%Let $\mathbb{H}_{\operatorname{Fin}}$ be the class of groups $H_{\mathcal{P}}$ where $\pi_1(\mathcal{P})$ is finite.
If $\pi_1(\mathcal{P})$ is residually finite then $H_{\mathcal{P}}$ is residually-$\mathbb{H}_{\operatorname{Fin}}$.
\end{theorem}

\begin{proof}
%Suppose that $Q$ is a countable residually finite group, and let $\mathcal{Q}=\langle \mathbf{x};\mathbf{r}\rangle$ be a countable presentation of $Q$. We shall prove that the group $T_{\mathcal{Q}}=H\ast_{(K_{\mathcal{Q}}, \phi)}$ from Theorem \ref{thm:intro1} is residually finite; the result follows by taking $G_Q:=T_{\mathcal{Q}}$. To prove that $T_{\mathcal{Q}}$ is residually finite we prove that for every $g\in T_{\mathcal{Q}}$ there exists a finite group $P_g$ with presentation $\mathcal{P}_g$ and a homomorphism $\delta_g: T_{\mathcal{Q}}\rightarrow T_{\mathcal{P}_g}$ such that $\phi_g(g)\neq1$. This is sufficient as the groups $T_{\mathcal{P}_g}$ are residually finite by assumption.
%
Recall that $H_{\mathcal{P}}$ is an HNN-extension $H\ast_{(K_{\mathcal{P}}, \phi)}$. Let $g\in H\ast_{(K_{\mathcal{P}}, \phi)}\setminus \{1\}$ be arbitrary. By Britton's lemma, $g$ is represented by a word $W=h_0t^{\epsilon_1}h_1\cdots t^{\epsilon_m}h_m$ which does not contain any subwords of the form $tkt^{-1}$ or $t^{-1}\phi(k)t$ for any $k\in K_{\mathcal{P}}$.
For those $h_i$ in $N_H(K_{\mathcal{P}})$ such that $th_it^{-1}\leq W$, write $\overline{h}_i:=h_iK_{\mathcal{P}}\in M_n/K_{\mathcal{P}}$, and for those $h_i$ in $N_H(\phi(K_{\mathcal{P}}))$ such that $t^{-1}h_it\leq W$, write $\overline{h}_i:=\phi^{-1}(h_i)K_{\mathcal{P}}\in M_n/K_{\mathcal{P}}$.
%(Ignore the $h_i\not\in N_H(K_{\mathcal{Q}})\cup N_H(\phi(K_{\mathcal{Q}}))$.)
As $M_n/K_{\mathcal{P}}\cong \pi_1(\mathcal{P})$ is residually finite there exists a finite group $P_g$ and a homomorphism $\overline{\sigma}_g: M_n/K_{\mathcal{P}}\rightarrow P_g$ such that $\overline{\sigma}_g(\overline{h}_i)\neq1$ for each $\overline{h}_i$ (if there are no $\overline{h}_i$, so for example if every $h_i\not\in N_H(K_{\mathcal{P}})\cup N_H(\phi(K_{\mathcal{P}}))$, then we may take $P_g$ to be trivial).
Now, there exists a quotient presentation $\mathcal{P}_g=\langle\mathbf{x}; \mathbf{r}, \mathbf{s}\rangle$ of $\mathcal{P}=\langle \mathbf{x}; \mathbf{r}\rangle$ which corresponds to the map $\overline{\sigma}_g$. Consider the group $T_{\mathcal{P}_g}=H\ast_{(K_{\mathcal{P}_g}, \phi)}$ from Theorem \ref{thm:intro1}. Then, the map $\sigma_g: T_{\mathcal{P}}\twoheadrightarrow T_{\mathcal{P}_g}$ from Theorem \ref{thm:mainconstruction}.\ref{point:mainconstruction3} is obtained by adding the relators $ tUt^{-1}\phi(U)^{-1}$, where $U$ corresponds to an element of $\langle\langle\mathbf{r}, \mathbf{s}, p, q\rangle\rangle\setminus\langle\langle\mathbf{r}, p, q\rangle\rangle\subseteq F(\mathbf{x}, p, q)$. In particular, this map naturally corresponds to the map $\overline{\sigma}_g$, and so if $h_i\in N_H(K_{\mathcal{P}})$ and $th_it^{-1}\leq W$ then $h_i\not\in K_{\mathcal{P}_g}$, and similarly if $h_i\in N_H(\phi(K_{\mathcal{P}}))$ and $t^{-1}h_it\leq W$ then $h_i\not\in \phi(K_{\mathcal{P}_g})$.
Note also that $N_H(K_{\mathcal{P}})=N_H(K_{\mathcal{P}_g})$.

We now prove that $\sigma_g(g)\neq1$, which proves the result.
%the word $W$ does not contain any subwords of the form $tkt^{-1}$ or $t^{-1}\phi(k)t$ for any $k\in K_{\mathcal{P}_g}$; that $\sigma_g(g)\neq1$ then follows from Britton's lemma.
Suppose $th_it^{-1}\leq W$. If $h_i\in N_H(K_{\mathcal{P}})$ then $h_i\not\in K_{\mathcal{P}_g}$ by the above. On the other hand, if $h_i\not\in N_H(K_{\mathcal{P}})$ then $h_i\not\in K_{\mathcal{P}_g}\leq N_H(K_{\mathcal{P}})$. Hence, $W$ does not contain any subword of the form $tkt^{-1}$ for any $k\in K_{\mathcal{P}_g}$. Similarly, $W$ does not contain any subword of the form $t^{-1}\phi(k)t$ for any $k\in K_{\mathcal{P}_g}$. That $\sigma_g(g)\neq1$ now follows from Britton's Lemma.
%By the above, it is sufficient to prove that if $th_it^{-1}\leq W$ with $h_i\not\in N_H(K)$ then
%. In particular, $N_H(K_{\mathcal{Q}})=N_H(K_{\mathcal{P}_g})$.
%Note that $N_H(K_{\mathcal{Q}})=N_H(K_{\mathcal{P}_g})$, as $\mathcal{P}_g$ is a quotient presentation of $\mathcal{Q}$.
%
%Currently omitted
%Let $Q$ be a residually finite group. Then $Q$ is the inverse limit of a set of finite subgroups $P_i$, $i\in I$ some index set, with each $G_{P_i}$ residually finite. As surjections $P_j\twoheadrightarrow P_i$ induce surjections $H_{\mathcal{P}_j}\twoheadrightarrow H_{\mathcal{P}_i}$, the group $H_{\mathcal{Q}}$ is the inverse limit of a set of residually finite groups. Hence, $H_{\mathcal{Q}}$ as residually finite as required.
%
\end{proof}

\p{Choosing the subgroups \boldmath{$M_n$}}
Let $M_0=\langle x_0, y_0\rangle$ denote the malcharacteristic subgroup of $H$ which is free of rank two from Theorem \ref{thm:mainconstruction}. Now, the subgroup $M_n$ in the proof of Theorem \ref{thm:mainconstruction} may be taken to be any subgroup of $M_0$ which is generated by a small cancellation subset $\mathbf{s}_n$ of $F(x_0, y_0)$, such that $|\mathbf{s}_n|=n$ and $\mathbf{s}_n$ contains no proper powers. This flexibility allows for concrete examples of this construction to be written down (for example, the groups $T_{\mathcal{P}_{\infty}}$ and $T_{\mathcal{P}_{k}}$ from the introduction). Theorem  \ref{thm:FreeProduct} now demonstrates that if we are more strict in our choice of the subgroups $M_n$ then we can obtain more properties of the groups $H_{\mathcal{P}}$ from Theorem \ref{thm:mainconstruction}.%, but concrete examples may be harder to write down.

%For presentations $\mathcal{P}=\langle \mathbf{x}; \mathbf{r}\rangle$ and $\mathcal{Q}=\langle \mathbf{y}; \mathbf{s}\rangle$ we write $\mathcal{P}\ast\mathcal{Q}$ for the presentation $\langle\mathbf{x}, \mathbf{y}; \mathbf{r}, \mathbf{s}\rangle$ (so $\pi_1(\mathcal{P}\ast\mathcal{Q})\cong\pi_1(\mathcal{P})\ast\pi_1(\mathcal{Q})$).

Suppose $\mathbf{x}\subset\mathbf{y}$, and let $\mathbf{r}\subset F(\mathbf{x})$. Then we write $\langle\langle\mathbf{r}\rangle\rangle_{F(\mathbf{x})}$ for the normal closure of $\mathbf{r}$ in the free group $F(\mathbf{x})$, and $\langle\langle\mathbf{r}\rangle\rangle_{F(\mathbf{y})}$ for the normal closure of $\mathbf{r}$ in the free group $F(\mathbf{y})$. Note that if we view $F(\mathbf{x})$ as a subgroup of $F(\mathbf{y})$ in the natural way then $\langle\langle\mathbf{r}\rangle\rangle_{F(\mathbf{x})}\leq \langle\langle\mathbf{r}\rangle\rangle_{F(\mathbf{y})}$.

\begin{theorem}
\label{thm:FreeProduct}
Let $H$ be as in Theorem \ref{thm:mainconstruction}.
%Write $\mathbf{z}_n:=\{z_1, z_2, \ldots, z_n\}$ for $n\in\mathbb{N}$ and $\mathbf{z}_{\infty}:=\{z_1, z_2, \ldots\}$.

Then $H$ contains a malcharacteristic subgroup $\langle p, q, z_1, z_2, \ldots\rangle$ which is free on the infinite set $\{p, q, z_1, z_2, \ldots\}$.
Suppose that the subgroups $M_n$, $n\in\mathbb{N}$, in the proof of Theorem \ref{thm:mainconstruction} are taken to be the subgroups $M_n:=\langle p, q, z_1, \ldots, z_n\rangle$, and that the subgroup $M_{\infty}$ is taken to be $\langle p, q, z_1, z_2, \ldots\rangle$. Then for every presentation $\mathcal{P}$ with finite generating set and
for every countable group presentation $\mathcal{Q}$ there exists a surjection $H_{\mathcal{P}}\twoheadrightarrow H_{\mathcal{P}\ast\mathcal{Q}}$.
\end{theorem}

\begin{proof}
Write $\mathbf{z}_n:=\{z_1, z_2, \ldots, z_n\}$ for $n\in\mathbb{N}$ and $\mathbf{z}_{\infty}:=\{z_1, z_2, \ldots\}$.
Suppose $\mathcal{P}=\langle \mathbf{x}; \mathbf{r(x)}\rangle$ and $\mathcal{Q}=\langle \mathbf{y}; -\rangle$.
Note that $H_{\mathcal{P}}$ has associated subgroup $K_{\mathcal{P}}=\langle\langle\mathbf{r(z_{|x|})}, p, q\rangle\rangle_{F(p, q, \mathbf{z_{|x|}})}$ (we identify the symbols $p, q$ here with the symbols $p, q$ in Theorem \ref{thm:mainconstruction}).
Then $H_{\mathcal{P}\ast\mathcal{Q}}$ is obtained from $H_{\mathcal{P}}$ by adding the relators $tUt^{-1}\phi(U)^{-1}$ to $H_{\mathcal{P}}$, for all $U\in \langle\langle p, q, \mathbf{r(z_{|x|})}\rangle\rangle_{F(p, q, \mathbf{z_{|x\sqcup y|}})}\setminus\langle\langle p, q, \mathbf{r(z_{|x|})}\rangle\rangle_{F(p, q, \mathbf{z_{|x|}})}$. This induces the required group homomorphism $H_{\mathcal{P}}\twoheadrightarrow H_{\mathcal{P}\ast\mathcal{Q}}$. The result then follows for all $\mathcal{Q}$ from Theorem \ref{thm:mainconstruction}.\ref{point:mainconstruction3}. 
\end{proof}

The free product is the coproduct in the category of groups, and so Theorem \ref{thm:FreeProduct} hints at the notion of a coproduct in our setting. However,
%the proof of Theorem \ref{thm:FreeProduct} treats the ``left'' factor $\mathcal{P}$ and the ``right'' factor $\mathcal{Q}$ of $\mathcal{P}\ast\mathcal{Q}$ differently, and so
in general $H_{\mathcal{P}\ast\mathcal{Q}}\neq H_{\mathcal{Q}\ast\mathcal{P}}$ and so it is not clear whether or not $H_{\mathcal{Q}}$ always surjects onto $H_{\mathcal{P}\ast\mathcal{Q}}$, even if $\mathcal{Q}$ has a finite generating set.
%Write $\operatorname{Pres}_{<\infty}$ for the subcategory of $\operatorname{Pres}$ obtained by taking the union of the categories $\operatorname{Pres}_n$, $n\in\mathbb{N}$.
Consider the subcategory $\operatorname{Pres}_{<\infty}:=\cup_{n\in\mathbb{N}} \operatorname{Pres}_n$ of $\operatorname{Pres}$.
If the surjection $H_{\mathcal{Q}}\twoheadrightarrow H_{\mathcal{P}\ast\mathcal{Q}}$ always exists for $\mathcal{Q}\in\operatorname{Pres}_{<\infty}$ then adding the morphisms $\mathcal{P}\rightarrow\mathcal{P}\ast\mathcal{Q}$ and $\mathcal{Q}\rightarrow\mathcal{P}\ast\mathcal{Q}$ to the category $\operatorname{Pres}_{<\infty}$ produces a new category $\overline{\operatorname{Pres}}_{<\infty}$ (possibly with a coproduct), and then the map $\overline{\operatorname{Pres}}_{<\infty}\rightarrow\operatorname{Grp}$ given by $\mathcal{P}\mapsto H_{\mathcal{P}}$ is still functorial.

%%%%%%%%%%%%%%%%%%%%%%%%%%%%%%%%%%%%%%%
%--------------------------------Malcharacteristic subgroups of free groups--------------------------------
%%%%%%%%%%%%%%%%%%%%%%%%%%%%%%%%%%%%%%%

\section{Malcharacteristic subgroups of free groups}
\label{sec:MalcharF2}
%In Lemma \ref{lem:malcharTriange} we use Lemma \ref{lem:reductiontolift} to reduce the proof that $M=\<x, y\>$ is malcharacteristic in $T_i=\<a, b; a^i, b^i, (ab)^i\>$ to proving that the lift $\widetilde{M}$ of $M$ to $F(a, b)$ is malcharacteristic.
In this section we find concrete examples of malcharacteristic subgroups of the free group $F(a, b)$. To find such subgroups we use fibre products of Stallings' graphs to obtain an algorithm which decides whether or not the malcharacteristic property holds for certain subgroups of $F(a, b)$.

The examples we obtain are used to find similar examples in triangle groups.
Indeed, at the end of Section \ref{sec:Malchar} we gave two specific words $x$ and $y$ and stated that they generate a malcharacteristic subgroup $M=\langle x, y\rangle$ of $T_{i, j, k}$. Let $\widetilde{M}=\langle x, y\rangle$ be the subgroup of $F(a, b)$ generated by these words $x$ and $y$.
% given in Section \ref{sec:Malchar} (so generated by the same words which generated the subgroup $M$ of $T_{i,j,k}$).
It follows from Lemma \ref{lem:malcharFree} that $\widetilde{M}$ is malcharacteristic in $F(a, b)$. We later use the fact that $\widetilde{M}$ is malcharacteristic in $F(a, b)$ to prove that $M$ is malcharacteristic in $T_{i,j,k}$.

\p{Maps of graphs}
The principal device applied in this section is that of maps of graphs; see Stallings' paper for more details as well as for the relevant definitions and terminology \cite{stallings1983topology}.
%A \emph{graph} is a $1$-dimensional CW-complex. A map between graphs is \emph{combinatorial} if $0$-cells are mapped to $0$-cells, and open $1$-cells are mapped homeomorphically to open $1$-cells. In this section each map $\eta:\Gamma_1\rightarrow\Gamma_2$ between graphs will have the property that, after we possibly subdivide $\Gamma_1$, the map $\eta$ is combinatorial.

%\p{Folding}
The key tool when studying maps of graphs is Stallings' folding algorithm \cite[Section 3.2]{stallings1983topology}.
The algorithm begins with a map of graphs $\Gamma_1\rightarrow \Gamma$, and factors this map as a composition $\Gamma_1\rightarrow\Gamma_2\rightarrow\cdots\rightarrow\Gamma_k\rightarrow\Gamma$ such that $\Gamma_k\rightarrow\Gamma$ is a locally injective map of graphs (and therefore is $\pi_1$-injective \cite[Proposition 5.3]{stallings1983topology}), and such that for $1\leq j < k$ the map $\Gamma_j\rightarrow\Gamma_{j+1}$ is a folding map. The graph $\Gamma_k$ is uniquely determined, although in general the individual foldings maps are not unique. A map of graphs $\Gamma_1\rightarrow\Gamma$ is \emph{folded} if it cannot be factored using the above algorithm.

Let $F(\mathbf{x})$ be a free group with basis $\mathbf{x}$. The basis $\mathbf{x}$ corresponds naturally to a bouquet of circles $\Gamma_\mathbf{x}$, which are directed and labeled by the letters $x_i\in \mathbf{x}$. Any word in $F(\mathbf{x})$ determines a closed path in $\Gamma_\mathbf{x}$. Let $\mathbf{s}=\{ W_1, W_2, \ldots, W_n\}$ be a set of words in $F(\mathbf{x})$. Then this set corresponds naturally to a bouquet $\Gamma_{\mathbf{s}}$ of $n$ circles, with circles corresponding to the $W_i\in \mathbf{s}$. There exists a natural map $\Gamma_{\mathbf{s}}\rightarrow \Gamma_\mathbf{x}$ which sends the $i$-th loop of $\Gamma_{\mathbf{s}}$ to the closed path of $\Gamma_{\mathbf{x}}$ corresponding to $W_i$. The \emph{Stallings' graph of $\mathbf{s}$} is the graph $\Gamma_{\mathbf{s}}$ equipped with the natural labeling induced by the map $\Gamma_{\mathbf{s}}\rightarrow\Gamma_\mathbf{x}$. A Stallings' graph $\Gamma_{\mathbf{s}}$ is \emph{folded} if the corresponding map $\Gamma_{\mathbf{s}}\rightarrow\Gamma_\mathbf{x}$ is folded. We use $\widehat{\Gamma}_{\mathbf{s}}$ to denote the \emph{folded Stallings' graph of $\mathbf{s}$}, so the unique folded graph, equipped with the natural map, such that the map $\Gamma_{\mathbf{s}}\rightarrow\Gamma_\mathbf{x}$ factors as $\Gamma_{\mathbf{s}}\rightarrow\widehat{\Gamma}_{\mathbf{s}}\rightarrow\Gamma$. If $C=\langle \mathbf{s}\rangle$ and the graphs $\Gamma_{\mathbf{s}}$ and $\widehat{\Gamma}_{\mathbf{s}}$ have the same rank (that is, $\pi_1(\Gamma_{\mathbf{s}})\cong \pi_1(\widehat{\Gamma}_{\mathbf{s}})$) then the induced map $\pi_1(\widehat{\Gamma}_{\mathbf{s}})\hookrightarrow\pi_1(\Gamma_{\mathbf{x}})$ corresponds to the map $C\hookrightarrow F(\mathbf{x})$.

\p{Fibre products}
The proofs in this section apply fibre products in the category of graphs \cite[Theorem 5.5]{stallings1983topology}.
If $C=\langle \mathbf{s}\rangle$ is a subgroup of $F(\mathbf{x})$ then $C$ is malnormal in $F(\mathbf{x})$ if and only if the non-diagonal components of the fibre-product $\widehat{\Gamma}_{\mathbf{s}}\otimes\widehat{\Gamma}_{\mathbf{s}}$ form a forest. Fibre products are computable for finitely generated subgroups of free groups, and hence malnormality is decidable for such subgroups of free groups.\footnote{For a different proof of decidability see \cite{baumslag1999malnormality}.} Also, if $C=\langle \mathbf{s}\rangle$ and $D=\langle \mathbf{t}\rangle$ are subgroups of $F(\mathbf{x})$ then $D\cap C^g=1$ for all $g\in F(\mathbf{x})$ if and only if the fibre-product $\widehat{\Gamma}_{\mathbf{t}}\otimes\widehat{\Gamma}_{\mathbf{s}}$ is a forest.
In particular, if $\delta\in\aut(F(\mathbf{x}))$ then by computing $\widehat{\Gamma}_{\delta(\mathbf{s})}\otimes\widehat{\Gamma}_{\mathbf{s}}$ it is decidable whether or not $\delta(C)\cap C^g=1$ for all $g\in F(\mathbf{x})$.

%A map of graphs $\Gamma_1\rightarrow\Gamma$ is \emph{simply folded} if every possible factoring, even after subdivision, of the map as $\Gamma_1\rightarrow\Gamma_2\rightarrow\Gamma$ where $f:\Gamma_1\rightarrow\Gamma_2$ is a folding map is such that there exists some closed interval $I$ in $\Gamma_1$ which is not homemorphic to a circle in $\Gamma_1$ but its image $f(I)$ in $\Gamma_2$ is homeomorphic to a circle. A Stallings' graph $\Gamma_S$ is \emph{simply folded} if the corresponding map $\Gamma_S\rightarrow\Gamma_\mathbf{x}$ is simply folded. We use $\overline{\Gamma}_S$ to denote a \emph{simply folded Stallings' graph of $S$}, so a simply folded graph such that the map $\Gamma_S\rightarrow\Gamma_\mathbf{x}$ factors as $\Gamma_S\rightarrow\overline{\Gamma}_S\rightarrow\Gamma$. We note that this definition is designed to exclude foldings which wind one edge around another. For example, the Stallings' graph $\Gamma_{\{ab, b\}}$ is simply folded, but $\widehat{\Gamma}_{\{ab, b\}}=\Gamma_{\{a, b\}}$.

\p{Length preserving automorphisms}
An automorphism $\alpha\in\aut(F(a, b))$ is \emph{length-preserving} if $|\alpha(a)|=1=|\alpha(b)|$.
We begin with the following lemma, which follows immediately from a result of Cohen--Metzler--Zimmermann \cite[Statement 3.9]{cohen1981does}.
For words $\mathbf{s}\in F(a, b)$ we write $\mathbf{s}^{{+}}$ to mean the subsemigroup of $F(a, b)$ generated by the words $\mathbf{s}$.%
\footnote{
%We could have used the ``Kleene star'' $\mathbf{s}^{\ast}$ (submonoid) rather than the ``Kleene plus'' $\mathbf{s}^{{+}}$ (subsemigroup) here, but we reserve the notation $\mathbf{s}^{\ast}$ for the ``symmetrised closure'' (see Section \ref{sec:MalnormAspher}).
We use the ``Kleene plus'' $\mathbf{s}^{{+}}$ (subsemigroup) rather than the ``Kleene star'' $\mathbf{s}^{\ast}$ (submonoid) as we reserve the notation $\mathbf{s}^{\ast}$ for the symmetrised closure (see Section \ref{sec:MalnormAspher}).
}
%\footnote{The Kleene star $\mathbf{s}^{\ast}$ is a more common construction than the Kleene plus $\mathbf{s}^{{+}}$, where $\mathbf{s}^{\ast}$ is $\mathbf{s}^{{+}}$ along with the empty word. We use $\mathbf{s}^{{+}}$ because the notation $\mathbf{s}^{\ast}$ clashes with the notation used in Section \ref{sec:MalnormAspher} for the symmetrised closure of a set of words (and we only need the subsemigroup and not the submonoid).}
%By $\{U(a, b), V(a, b)\}^{{+}}$ we mean the subsemigroup of $F(a, b)$ generated by the words $U(a, b)$ and $V(a, b)$.
%We call a word $W$ \emph{fully cyclically reduced} if either $|W|=1$ or $W$ is both freely and cyclically reduced and begins and ends with different letters. Note that a cyclically reduced word $W^{\prime}$ is always a cyclic shift of a fully cyclically reduced word. We use $\epsilon, \epsilon^{\prime}$, etc. to mean an integer of absolute value $1$.

\begin{lemma}
\label{lem:PrimitiveForm}
For all $\delta\in \aut(F(a, b))$ there exists an inner automorphism $\gamma\in\inn(F(a, b))$, length-preserving automorphisms $\alpha_1, \alpha_2$, and an integer $m\in\mathbb{Z}$ such that one of the following occurs:
\begin{enumerate}
\item\label{item:PrimitiveForm1} $\delta\gamma$ is length-preserving; or%$\alpha\psi_1\gamma_u(a)=a$, $\alpha\psi_1\gamma_u(b)=b$; or
\item\label{item:PrimitiveForm2} $\alpha_2\delta\alpha_1\gamma(a)=ab^m$, $\alpha_2\delta\alpha_1\gamma(b)=b$ with $m\neq0$; or
\item\label{item:PrimitiveForm3} $\delta\alpha_1\gamma(a),\delta\alpha_1\gamma(b)\in\{ab^m,ab^{m+1}\}^{{+}}$.
\end{enumerate}
%Then there exists some $\gamma_v\in\inn(F(a, b)), \epsilon=\pm1$ such that no positive word over $\gamma_v\alpha(a)$, $\gamma_v\alpha(b^{\epsilon})$ contains $a^3$ or $a^{-3}$ or no positive word over contains $b^3$ or $b^{-3}$.
\end{lemma}
The length-preserving automorphism $\alpha_1$ corresponds to trivial changes of notation in the images of $a$ and $b$, for example replacing $\{ab^m,ab^{m+1}\}^{{+}}$ with $\{ba^m,ba^{m+1}\}^{{+}}$. The length-preserving automorphism $\alpha_2$ in Case (\ref{item:PrimitiveForm2}) allows for the images to be swapped, for example if $\delta$ is the map $\delta: a\mapsto b$, $b\mapsto ab^m$ then we take $\alpha_2:a\mapsto b$, $b\mapsto a$ and take $\alpha_1$ and $\gamma$ to be trivial.

%\begin{proof}
%Suppose $\alpha\in F(a, b)$ is such that there does not exist a length-preserving $\alpha\in\aut(F(a, b))$, an inner automorphism $\gamma_u$ and an integer $k\in\mathbb{Z}$ such that $\gamma_u\alpha\alpha(a)=ab^k$, $\gamma_u\alpha\alpha(b)=b$. Then there exists some $\gamma_v\in\inn(F(a, b)), \epsilon=\pm1, n\in\mathbb{Z}$ such that $\gamma_v\alpha(a)$ and $\gamma_v\alpha(b^{\epsilon})$ lie in the free subsemigroup of $F(a, b)$ generated by $ab^n$ and $ab^{n+1}$ \cite[Statement 3.9]{cohen1981does}, and hence every positive word over $\gamma_v\alpha(a)$ and $\gamma_v\alpha(b^{\epsilon})$ is contained in this subsemigroup. As no element of this subsemigroup has a $a^{\pm3}$ as a subword, and as the length-preserving automorphism $\alpha$ only introduces trivial changes in notation, the result holds.

%Suppose $\alpha\in F(a, b)$ is such that there exists a length-preserving $\alpha\in\aut(F(a, b))$, an inner automorphism $\gamma_v$ and a non-zero integer $k\in\mathbb{Z}$ such that $\gamma_v\alpha\alpha(a)=ab^k$, $\gamma_u\alpha\alpha(b)=b$. If $k>0$ then no positive word over $\gamma_v\alpha\alpha(a)$, $\gamma_v\alpha\alpha(b)$ contains $a^{\pm3}$, while if $k<0$ then no positive word over $\gamma_v\alpha\alpha(a)$, $\gamma_v\alpha\alpha(b^{-1})$ contains $a^{\pm3}$. Therefore, as the length-preserving automorphism only $\alpha$ introduces trivial changes in notation, the result holds.
%\end{proof}

\p{Determining the malcharacteristic property}
Lemma \ref{lem:malcharlem} now gives an algorithm which decides if a given subgroup from a certain class $\mathcal{C}$ of subgroups of $F(a, b)$ is malcharacteristic or not.
Note that it is decidable if a finite set of words $\mathbf{s}$ satisfies the conditions of Lemma \ref{lem:malcharlem} (hence, membership of the class $\mathcal{C}$ is decidable).
%By a positive word in $F(a, b)$ we mean a word in $\{a, b\}^{\ast}$, so a word not containing $a^{-1}$ or $b^{-1}$.
Note also that the subsemigroup $\{a^2, a^3, b^2, b^3\}^{{+}}$ of $F(a, b)$ consists precisely of the freely reduced words $W$ in $F(a, b)$ which do not contain $a^{-1}$ or $b^{-1}$ and such that every instance of $a$ in $W$ is part of an $a^2$-term and every instance of $b$ in $W$ is part of a $b^2$-term.
Therefore, in Lemma \ref{lem:malcharlem} every instance of $a$ in any circuit in the folded Stallings' graph $\widehat{\Gamma}_{\mathbf{s}}$ is part of an $a^2$-term, and similarly for $b$.

\begin{lemma}
\label{lem:malcharlem}
Let $\mathbf{s}$ be a finite set of pairwise non-equal words contained in the subsemigroup $\{a^2, a^3, b^2, b^3\}^{{+}}$ of the free group $F(a, b)$, and write $C:=\langle \mathbf{s}\rangle$. Suppose that every circuit in the folded Stalling's graph $\widehat{\Gamma}_{\mathbf{s}}$ contains some $a^3$-term and some $b^3$-term. Then $C$ is malcharacteristic in $F(a, b)$ if and only if $C$ is malnormal in $F(a, b)$ and $\alpha(C)\cap C^g=1$ for all non-trivial length-preserving automorphism $\alpha\in\aut(F(a, b))$ and all $g\in F(a, b)$. In particular, it is decidable whether $C$ is malcharacteristic in $F(a, b)$ or not.
%$\alpha\in\aut(F(a, b))$ such that $|\alpha(a)|=1=|\alpha(b)|$ the fibre product $\Gamma_C\otimes_{\Gamma_{F(x, y)}}\Gamma_{\alpha(C)}$ is a forest.
\end{lemma}

\begin{proof}
%Suppose $C$ is malcharacteristic in $F(a, b)$. Then it is clear that $C$ is malnormal in $F(a, b)$ and for all non-trivial length-preserving $\alpha\in\aut(F(a, b))$ and all $g\in G$, $\alpha(C)\cap C^g=1$.
%
First, suppose that $C$ is malnormal in $F(a, b)$ and that $\alpha(C)\cap C^g=1$ for all length-preserving automorphisms $\alpha\in\aut(F(a, b))$ and all $g\in F(a, b)$. We show that for all non-inner $\delta\in\aut(F(a, b))$ the fibre product of the folded Stallings' graphs $\widehat{\Gamma}_{\delta(\mathbf{s})}\otimes\widehat{\Gamma}_{\mathbf{s}}$ is a forest; as $C$ is malnormal this implies that $C$ is malcharacteristic in $F(a, b)$. To show that this condition holds for all such $\delta\in\aut(F(a, b))$ we consider the three cases of Lemma \ref{lem:PrimitiveForm}.
The result holds by assumption if $\delta\in\aut(F(a, b))$ is such that Case (\ref{item:PrimitiveForm1}) of Lemma \ref{lem:PrimitiveForm} holds.

%because $\Gamma_{\delta\gamma(S)}=\Gamma_{\delta(S)}$, as inner automorphisms correspond to changing base points in the Stallings' graph, we have that $\Gamma_{\delta(S)}\otimes\Gamma_{S}$ is a forest, and so $\delta(C)\cap C^g=1$ for all $g\in F(a, b)$ as required.

Now, suppose $\delta\in \aut(F(a, b))$ is such that Case (\ref{item:PrimitiveForm2}) of Lemma \ref{lem:PrimitiveForm} holds, so $\alpha_2\delta\alpha_1\gamma(a)=ab^m$, $\alpha_2\delta\alpha_1\gamma(b)=b$ for some $m\in\mathbb{Z}\setminus\{0\}$, $\gamma\in\inn(F(a, b))$ and $\alpha_1, \alpha_2\in\aut(F(a, b))$ length-preserving. Then
\begin{align*}
\alpha_2\delta\alpha_1\gamma(\mathbf{s})&\subseteq\{(ab^m)^2, (ab^m)^3, b^2, b^3\}^{{+}}\\&\subseteq\{ab^ma, (ab^m)^2a, b\}^{{+}}=:\mathbf{s}_1.
\end{align*}
However, no element of $\mathbf{s}_1$ contains an $a^3$-term. Hence no circuit in the folded Stallings' graph $\widehat{\Gamma}_{\alpha_2\delta\alpha_1\gamma(\mathbf{s})}$ contains an $a^3$-term, and as $\gamma$ is inner the same holds for $\Gamma_{\alpha_2\delta\alpha_1(\mathbf{s})}$. As $\alpha_1, \alpha_2$ are length-preserving, we further have that either no circuit in $\widehat{\Gamma}_{\delta(\mathbf{s})}$ contains an $a^3$-term or no circuit in $\widehat{\Gamma}_{\delta(\mathbf{s})}$ contains a $b^3$-term. Therefore, the fibre product $\widehat{\Gamma}_{\delta(\mathbf{s})}\otimes\Gamma_{\mathbf{s}}$ is a forest, as required.

Finally, suppose $\delta\in \aut(F(a, b))$ is such that Case (\ref{item:PrimitiveForm3}) of Lemma \ref{lem:PrimitiveForm} holds, so $\delta\alpha_1\gamma(a), \delta\alpha_1\gamma(b)\in\{ab^m, ab^{m+1}\}^{{+}}$ for some $m\in\mathbb{Z}$, $\gamma\in\inn(F(a, b))$ and $\alpha_1\in\aut(F(a, b))$ length-preserving. Then there exist positive words $U_a$ and $U_b$ such that:
%as no word in $\mathbf{s}$ contains an $a^{-1}$ or $b^{-1}$,
\begin{align*}
\delta\alpha_1\gamma(\mathbf{s})
&=\{U_a(ab^m, ab^{m+1})^2, U_a(ab^m, ab^{m+1})^3, U_b(ab^m, ab^{m+1})^2, U_b(ab^m, ab^{m+1})^3\}\\
&\subseteq\{ab^m, ab^{m+1}\}^{{+}}=:\mathbf{s}_2.
\end{align*}
If $m\neq0, -1$ then no element of $\mathbf{s}_2$ contains an $a^3$-term. If $m=0$ or $m=-1$ then $\mathbf{s}_2=\{a, ab^{\epsilon}\}^{{+}}$, $\epsilon=\pm1$, and so no element of $\mathbf{s}_2$ contains a $b^3$-term. Hence, no circuit in the folded Stallings' graph $\widehat{\Gamma}_{\delta\alpha_1\gamma(\mathbf{s})}$ contains an $a^3$-term or no circuit in $\widehat{\Gamma}_{\delta\alpha_1\gamma(\mathbf{s})}$ contains a $b^3$-term. Therefore, as $\gamma$ is inner and as $\alpha_1$ is length-preserving, either no circuit in $\widehat{\Gamma}_{\delta(\mathbf{s})}$ contains an $a^3$-term or no circuit in $\widehat{\Gamma}_{\delta(\mathbf{s})}$ contains a $b^3$-term. Hence, the fibre product $\widehat{\Gamma}_{\delta(\mathbf{s})}\otimes\Gamma_{\mathbf{s}}$ is a forest, as required.

%, and hence no circuit in the folded Stallings' graph $\widehat{\Gamma}_{\delta\alpha_1\gamma(\mathbf{s})}$ contains an $a^3$-term

The opposite direction is trivial. For the decidability statement, recall that we can use fibre products of Stallings' graphs to determine if $C$ is malnormal in $F(a, b)$ or not, and also to determine for each non-trivial length-preserving automorphism $\alpha$ whether $\alpha(C)\cap C^g=1$ for all $g\in F(a, b)$ or not. The decidability statement then follows because there are only finitely many length-preserving automorphisms $\alpha\in\aut(F(a, b))$.
%, and so we can take the finitely many translates of the set $\mathbf{s}$ under the length-preserving automorphisms and then compute the fibre products $\widehat{\Gamma}_{\alpha(\mathbf{s})}\otimes\widehat{\Gamma}_{\mathbf{s}}$.
%
\end{proof}

\p{Examples of malcharacteristic subgroups}
Lemma \ref{lem:malcharFree} now gives an example of a malcharacteristic subgroup $L$ of $F(a, b)$.

\begin{lemma}
\label{lem:malcharFree}
Let $L$ be the subgroup of $F(a, b)$ which is generated by the following elements, with $\rho\gg1$.
\begin{align*}
\omega_{x}&:=a^{3}b^{3}a^3b^4\ldots a^3b^{\rho+2}\\
\omega_{y}&:=a^3b^{\rho+3}a^3\ldots a^{3}b^{2\rho+2}
\end{align*}
Then $L$ is a non-cyclic malcharacteristic subgroup of $F(a, b)$.
\end{lemma}

\begin{proof}
First, note that $\mathbf{s}=\left\{\omega_{x}, \omega_{y}\right\}\subseteq\{a^2, a^3, b^2, b^3\}^{{+}}$, and that every circuit in the folded Stalling's graph $\widehat{\Gamma}_{\mathbf{s}}$ contains some $a^3$-term and some $b^3$-term. Therefore, we may apply Lemma \ref{lem:malcharlem}. Now, the words $\omega_{x}$ and $\omega_{y}$ satisfy Wise's $c(5)$ small cancellation condition for $\rho \gg1$, and so the subgroup $L$ is non-cyclic and malnormal in $F(a, b)$ \cite[Theorems~2.11~\&~2.14]{wise2001residual}.

%We shall now prove that $\widetilde{M}$ is malcharacteristic. We first prove that $\widetilde{M}_{\{x, y\}}$ is malcharacteristic.
Suppose that $\alpha\in\aut(F(a, b))$ is non-trivial and length-preserving, and that $\alpha(L)\cap L^g\neq1$ for some $g\in F(a, b)$. Then the fibre product $\Gamma_{\alpha(\mathbf{s})}\otimes\Gamma_{\mathbf{s}}$ is not a forest. If we pick a circuit in this fibre product $\Gamma_{\alpha(\mathbf{s})}\otimes\Gamma_{\mathbf{s}}$ then this circuit contains a subpath with label $U=(a^3b^pa^3b^{p+1}a^3b^{p+2}a^3b^{p+3}a^3b^{p+4}a^3)^{\epsilon}$, where $p>3$ and $\epsilon=\pm1$, as every circuit in $\Gamma_{\mathbf{s}}$ contains a subpath with a label of this form. By traversing our chosen circuit in the opposite direction if necessary, we may assume $\epsilon=1$.

The set $\alpha(\mathbf{s})$ is $C^{\prime}(\lambda)$ small cancellation, with $\lambda$ arbitrarily small (as $\rho\gg 1$). Therefore, $U$ is either a subword of an element of $\alpha(\mathbf{s})^{\pm1}$ (so of $\alpha\left(\omega_{x}\right)^{\pm1}$ or of $\alpha\left(\omega_{y}\right)^{\pm1}$), or a subword of the product of two elements of $\alpha(\mathbf{s})^{\pm1}$. Hence a subword $a^3b^qa^3b^{q+1}a^3$ of $U$ is a subword of either $\alpha\left(\omega_{x}\right)^{\pm1}$ or of $\alpha\left(\omega_{y}\right)^{\pm1}$, $q>3$. It is then clear that $\alpha$ is trivial (one way to see this is to note that $\alpha^{-1}$ is length-preserving and then to apply each non-trivial length-preserving automorphism to $a^3b^qa^3b^{q+1}a^3$ and to $a^3b^qa^3b^{q+1}a^3$; the result is never a subword of $\omega_{x}$ nor of $\omega_{y}$). This is a contradiction, so $\alpha(L)\cap L^g=1$ for all $g\in F(a, b)$ as required.
\end{proof}

%Recall that $\widetilde{M}_{\{x, y\}}$ is the subgroup of $F(a, b)$ generated by the words $x$ and $y$ given in Section \ref{sec:Malchar}.
The subgroup $\widetilde{M}=\langle x, y\rangle$ of $F(a, b)$ mentioned at the start of this section is in the same $\aut(F(a, b))$-orbit as $L$. Hence, by Lemma \ref{lem:malcharFree}, $\widetilde{M}$ is malcharacteristic in $F(a, b)$. In Lemma \ref{lem:malcharTriange} we use this fact to prove that the subgroup $M=\langle x, y\rangle$ of $T_{i, j, k}$ is malcharacteristic in $T_{i, j, k}$.

%Note that if we write $\widetilde{M}_{\{x, y\}}$ to be the lift to $F(a, b)$ of the subgroup $M$ of $T_{i, j, k}$ defined at the end of Section \ref{sec:Malchar} then $\widetilde{M}_{\{x, y\}}$ and $\widetilde{M}$ are in the same automorphic orbit of $F(a, b)$. Therefore, by Lemma \ref{lem:malcharFree}, $\widetilde{M}_{\{x, y\}}$ is malcharacteristic in $F(a, b)$.

Lemmas \ref{lem:malcharArbRank} and \ref{lem:malcharFree} combine to prove the following result.

\begin{proposition}
\label{lem:malcharFreeArbRank}
The free group $F(a, b)$ contains malcharacteristic subgroups $L_n$ of arbitrary rank $n$ (possibly countably infinite).
\end{proposition}

Indeed, the subgroup $L_n$ is generated by $n$ finite subwords satisfying $C^{\prime}(1/6)$ of the following infinite word:
%\cite[Lemma 6]{BumaginWise2005}.
\[
\omega_{x}\omega_{y}\left(\omega_{x}\omega_{y}^2\right)\omega_{x}\omega_{y}\left(\omega_{x}\omega_{y}^2\right)^2\omega_{x}\omega_{y}\left(\omega_{x}\omega_{y}^2\right)^3\cdots.
\]

%Note that if the collection of words $W_1, \ldots, W_n$ satisfy Wise's $c(4)$ small cancellation condition on maps of graphs then $C$ is malnormal \cite{wise2001residual}. This condition has the same flavour as the statement of the above lemma.

%%%%%%%%%%%%%%%%%%%%%%%%%%%%%%%%%%%%%%%
%--------------------------------Malnormal subgroups of small cancellation groups--------------------------------
%%%%%%%%%%%%%%%%%%%%%%%%%%%%%%%%%%%%%%%
\section{Malnormal subgroups from small cancellation presentations}
\label{sec:MalnormAspher}
The main results of this section are Theorems \ref{thm:FreeMalnormMetric6} and \ref{thm:FreeMalnormMetric4T4}. We also prove Theorem \ref{thm:FreeMalnormMetric}, which is a special case of these two theorems. Theorem \ref{thm:FreeMalnormMetric4T4} is applied in Section \ref{sec:MalcharTriangle} to prove that the subgroup $M$ of $T_{i, j, k}$, $i, j, k\geq6$, as stated at the end of Section \ref{sec:Malchar}, is a free and malcharacteristic subgroup of $T_{i, j, k}$; Theorem \ref{thm:intro1} and Theorem \ref{thm:GeneralTriangle1} then follow by Theorem \ref{thm:mainconstruction}.

The results of this section lift conjugation in groups $\langle \mathbf{x}; \mathbf{r}\rangle$ to conjugation in the ambient free group $F(\mathbf{x})$.
%These theorems lift conditions on conjugation of subgroups and intersections of subgroups in small cancellation groups $\langle \mathbf{x}; \mathbf{r}\rangle$ to the analogous situation in the ambient free group $F(\mathbf{x})$.
Our principal tool here is small cancellation theory (see \cite[Section V]{L-S} for the basic definitions and results), and our key innovation is to prove results about the subgroup $\langle \mathbf{s}\rangle_G$ of the group $G=\langle \mathbf{x}; \mathbf{r}\rangle$ by studying the presentation $\langle \mathbf{x}; \mathbf{r}, \mathbf{s}\rangle$.
Theorems \ref{thm:FreeMalnormMetric6} and \ref{thm:FreeMalnormMetric4T4} give a new method for recognising, for $L$ and $M$ subgroups of a small cancellation group $G$, if $L^g\cap M= 1$ for all $g\in G$.
%We also prove Lemma \ref{lem:FreeSbgps}, which gives easy examples of free subgroups of small cancellation groups.

Theorem \ref{thm:FreeMalnormMetric6} is not applied in this paper. It is included for the sake of completeness, and because we believe that Theorem \ref{thm:FreeMalnormMetric} and Theorems \ref{thm:FreeMalnormMetric6} and \ref{thm:FreeMalnormMetric4T4} are of general interest. These results can be viewed as a generalisation of results of Wise, who applied small cancellation conditions on graphs to give a method of recognising malnormal subgroups of free groups \cite{wise2001residual}. We conclude this section by noting that our proofs can be adapted to Wise's graphical small cancellation setting.
%We refer the reader to Lyndon--Schupp's book for the relevant definitions and results regarding classical small cancellation theory \cite[Section V]{L-S}.

%A subgroup $\langle \mathbf{s}\rangle$ of a group given by a small cancellation presentation $\langle \mathbf{x};\mathbf{r}\rangle$ is called \emph{relatively small cancellation} if the presentation $\langle \mathbf{x}; \mathbf{r}, \mathbf{s}\rangle$ is also a small cancellation presentation (the precise small cancellation condition depends on context).
%Loosely speaking, Theorems \ref{thm:FreeMalnormMetric6} and \ref{thm:FreeMalnormMetric4T4} prove that no element of a relatively small cancellation subgroup is non-freely conjugate to a Dehn reduced word.
%In this section we give a brief introduction to small cancellation theory. We then state and prove Theorem \ref{thm:Conjugacy}, and prove Theorem \ref{thm:FreeMalnorm}.

%\input{Smallcancellation.tex}

%%%%%%%%%%%%%%%%%%%%%%%%%%%%%%%%%%%%%%%%%%%%%%%%%%%%%%%%%%%%%%%%
%------------------------------------------------------------------------------------------------------------------------------------------------------------------------
%------------------------Free subgroups of small cancellation groups-------------------------------------------------------------------------------
%------------------------------------------------------------------------------------------------------------------------------------------------------------------------
%%%%%%%%%%%%%%%%%%%%%%%%%%%%%%%%%%%%%%%%%%%%%%%%%%%%%%%%%%%%%%%%

\subsection{Free subgroups of small cancellation groups}
We start with Lemma \ref{lem:diagramD_U}, which describes how to view elements of the subgroup $\langle\mathbf{s}\rangle$ of $\langle\mathbf{x}; \mathbf{r}\rangle$ as diagrams over the presentation $\langle \mathbf{x}; \mathbf{r}, \mathbf{s}\rangle$. We then apply this view to prove Lemma \ref{lem:FreeSbgps}, below, which gives a new method for recognising free subgroups of small cancellation groups.

If $\mathbf{r}$ is a subset of $F(\mathbf{x})$ then we write $\mathbf{r}^{\ast}$ for the symmetrised closure of $\mathbf{r}$, that is, $\mathbf{r}^{\ast}$ is the set of all cyclic shifts, inverses, and cyclic shifts of inverses of elements of $\mathbf{r}$.
%\footnote{In Section \ref{sec:MalcharF2} we used the similar symbol ${+}$ (and not $\ast$) to mean a different operation on a set of words; we do not use the symbol ${+}$ or its associated operation again in this paper.}
By an $\mathbf{r}^{\ast}$-diagram we mean a diagram over the presentation $\langle \mathbf{x};\mathbf{r}\rangle$. Note that an $\mathbf{r}^{\ast}$-diagram is also an $(\mathbf{r}\cup \mathbf{s})^{\ast}$-diagram. By an $\mathbf{r}^{\ast}$-disk diagram we mean a diagram $\Delta$ where $\mathbb{R}\setminus\Delta$ has exactly one component. By an \emph{$\mathbf{r}^{\ast}$-region $\mathcal{R}$} we mean an $\mathbf{r}^{\ast}$-diagram which corresponds to a single relator $R\in\mathbf{r}^{\ast}$. We call $\mathcal{R}$ a \emph{region} if the set of relators is understood.
By an $\mathbf{r}^{\ast}$-piece we mean a piece relative to the set $\mathbf{r}^{\ast}$.
%
%Given a $C(3)$ small cancellation presentation $\langle \mathbf{x}; \mathbf{r}, \mathbf{s}\rangle$, we begin by viewing elements of $\langle \mathbf{s}\rangle$ as $\mathbf{s}^{\ast}$-diagrams, and hence as $(\mathbf{r}\cup \mathbf{s})^{\ast}$-diagrams.
%
A word $W\in F(\mathbf{x})$ is \emph{cyclically reduced} if every cyclic shift of $W$ is freely reduced.

\begin{lemma}
\label{lem:diagramD_U}
Let $\langle \mathbf{x}; \mathbf{r}, \mathbf{s}\rangle$ be a $C(m)$ presentation, $m\geq3$. For every word $U=S_1^{\epsilon_1}\cdots S_{n}^{\epsilon_n}$, $S_i\in\mathbf{s}$, there exists an $\mathbf{s}^{\ast}$-disk diagram $\mathcal{D}_U$ with boundary labels precisely the cyclically reduced cyclic shifts of $U$, and which has $n$ $\mathbf{s}^{\ast}$-regions $\mathcal{S}_i$, $1\leq i\leq n$.
%all of which intersect with the boundary $\partial \mathcal{D}_U$ of $\mathcal{D}_U$. Moreover, each path $e_i=\mathcal{S}_i\cap\partial\mathcal{D}_U$
For each such region $\mathcal{S}_i$ the intersection of $\mathcal{S}_i$ with the boundary $\partial \mathcal{D}_U$ is connected and contains an edge, so $p_i:=\mathcal{S}_i\cap\partial \mathcal{D}_U$ is a path. Each such path $p_i$ is the product of at least $m-2$ $(\mathbf{r}\cup\mathbf{s})^{\ast}$-pieces.
\end{lemma}

\begin{proof}
First form the bouquet of $n$ petals labeled clockwise with the words $S_1, \ldots, S_n$ respectively (different petals may have the same label).
This defines, in the obvious way, an $\mathbf{s}^{\ast}$-disk diagram $\Delta_U$ with $n$ regions $\mathcal{S}_i$ such that $S_i$ is a label of $\mathcal{S}_i$. Form the $\mathbf{s}^{\ast}$-disk diagram $\mathcal{D}_U$ by taking $\Delta_U$ and folding small-cancellation pieces in adjacent petals together, that is, perform certain Stallings' folds on the edges of $\Delta_U$.
%This diagram $\mathcal{D}_U$ resembles a disk with a (possibly internal) marked point $u$ and internal edges going from the marked point to boundary $\partial\mathcal{D}_U$, as in Figure \ref{\mathbf{x}\mathbf{x}\mathbf{x}}.
%If $n>1$ then the marked point $u$ is a vertex of degree at least $n$ in $\mathcal{D}_U$ and all other vertices have degree $2$ or $3$.
%Note that, viewing $U$ as a word over $F(\mathbf{x})$, a cyclic shift $V$ of $U$ is a boundary label for $\mathcal{D}_U$.
%We write $e_i$ for the path on the boundary of $\mathcal{D}_U$ corresponding to $S_i^{\epsilon_i}$.
Applying the $C(m)$ condition, $p_i:=\mathcal{S}_i\cap\partial\mathcal{D}_U$ is connected and contains at least $m-2$ $(\mathbf{r}\cup\mathbf{s})^{\ast}$-pieces.
Clearly some, and hence every cyclically reduced cyclic shift of $U$ is a label of $\partial \mathcal{D}_U$, and the result follows.
\end{proof}

Suppose that two diagrams $\Delta_1$ and $\Delta_2$ have a paths on their boundaries with a common label, so there exist paths $I_1\subset\partial\Delta_1$ and $I_2\subset\partial\Delta_2$ such $I_1$ and $I_2$ have the same label. Then we may form a new diagram $\Delta=\Delta_1\cup\Delta_2$ where we equate the paths $I_1$ and $I_2$. We call the image $I\subset\Delta$ of the paths $I_1$ and $I_2$ the \emph{attaching path of $\Delta$ over $\Delta_1$ and $\Delta_2$}, or just the \emph{attaching path} if $\Delta$, $\Delta_1$ and $\Delta_2$ are understood.

If $I$ is a path in a diagram then we use $\iota(I)$ and $\tau(I)$ to denote, respectively, the initial and terminal vertex of $I$. If $I$ and $J$ are paths with $\tau(I)=\iota(J)$ then we write $IJ$ for the concatenation of the paths $I$ and $J$. A path $J^{\iota}$ is an \emph{initial subpath of the path $J$} if there exists a path $K$ with $\tau(J^{\iota})=\iota(K)$ and $J^{\iota}K=I$. Similarly, a path $J^{\tau}$ is a \emph{terminal subpath of the path $J$} if there exists a path $K$ with $\tau(K)=\iota(J^{\tau})$ and $KJ^{\tau}=I$.
If $I$ is a path in the diagram $\Delta$ and $(J_1, \ldots, J_n)$ is a sequence of paths in $\Delta$ with $\tau(J_i)=\iota(J_{i+1})$, $1\leq i<n$, then the path $I$ \emph{straddles the sequence of paths $(J_1, \ldots, J_n)$} if $J_1^{\tau}J_2\cdots J_{n-1}J_n^{\iota}$ is a subpath of $I$, where $J_1^{\tau}$ is a non-empty terminal subpath of $J_1$ and $J_n^{\iota}$ is a non-empty initial subpath of $J_n$.

\begin{lemma}
\label{lem:connectingpathI}
Let $\langle \mathbf{x}; \mathbf{r}, \mathbf{s}\rangle$ be a $C(4)$ presentation and let $U\in F(\mathbf{x})$. Suppose that $\Delta$ is an $(\mathbf{r}\cup \mathbf{s})^{\ast}$-diagram containing the subdiagram $\mathcal{D}_U$, formed as in Lemma \ref{lem:diagramD_U}. If $\mathcal{R}$ is an $\mathbf{r}^{\ast}$-region of $\Delta$ which attaches to $\mathcal{D}_U$ over an attaching path $I$ then $I$ cannot straddle a sequence of three paths $(p_i, p_{i+1\pmod n}, p_{i+2\pmod n})$ of $\mathcal{D}_U$.
\end{lemma}

\begin{proof}
Suppose $I$ straddles the sequence $(p_i, p_{i+1\pmod n}, p_{i+2\pmod n})$. Then $p_{i+1\pmod n}$ consists of a single $(\mathbf{r}\cup\mathbf{s})^{\ast}$-piece as $\mathcal{R}$ corresponds to a single relator $R\in\mathbf{r}^{\ast}$. Now, by taking $m=4$ in Lemma \ref{lem:diagramD_U} we see that $p_{i+1\pmod n}$ cannot be written as the product of fewer than two $(\mathbf{r}\cup\mathbf{s})^{\ast}$-pieces, a contradiction.
\end{proof}

For a word $W\in F(\mathbf{x})$ and a rational number $c$ we write $W>c\mathbf{r}$ to mean that there is a relator $R\in\mathbf{r}^{\ast}$ such that $W$ is a subword of $R$ with $|W|>c|R|$. A word $U\in F(\mathbf{x})$ is called \emph{$\mathbf{r}^{\ast}$-reduced} if $U$ is freely reduced and there is no subword $W$ of $U$ with $W>\frac12\mathbf{r}$. We call $U$ \emph{cyclically $\mathbf{r}^{\ast}$-reduced} if every free reduction of every cyclic shift of $U$ is non-empty and $\mathbf{r}^{\ast}$-reduced. Note that a non-cyclically reduced word may be cyclically $\mathbf{r}^{\ast}$-reduced. We call $W$ \emph{(cyclically) Dehn reduced}, rather than (cyclically) $\mathbf{r}^{\ast}$-reduced, if the underlying set of relators $\mathbf{r}$ is understood.

Note that for $\mathbf{s}\subset F(\mathbf{x})$, if a word $U\in \langle \mathbf{s}\rangle$ is freely/cyclically reduced as a word over $\mathbf{s}^{\pm1}$ then it is not necessarily freely/cyclically reduced over $F(\mathbf{x})$. We therefore have to be careful what we mean by ``free reduction''.
If $\mathbf{s}\subset F(\mathbf{x})$ we shall write $U_{\mathbf{s}}(\mathbf{s})$, $V_{\mathbf{s}}(\mathbf{s})$, and so on, to mean words over the set $\mathbf{s}^{\pm1}$ which are freely reduced over $\mathbf{s}$, and we shall write $U_{\mathbf{x}}(\mathbf{s})$, $V_{\mathbf{x}}(\mathbf{s})$, and so on, to mean words over the set $\mathbf{s}^{\pm1}$ which are freely reduced over $\mathbf{x}$.

If $G$ is given by the generating set $\mathbf{x}$ and $\mathbf{s}\subset F(\mathbf{x})$ then we use $\langle \mathbf{s}\rangle_G$ to denote the subgroup of $G$ generated by the words $\mathbf{s}$.

\begin{lemma}
\label{lem:FreeSbgps}
Let $G$ be given by a $C^{\prime}(1/6)$ or $C^{\prime}(1/4)-T(4)$ presentation $\langle \mathbf{x};\mathbf{r}\rangle$.
Suppose that $\langle \mathbf{x};\mathbf{r}, \mathbf{s}\rangle$ is a $C^{\prime}(1/4)$ presentation. Then every word $U_{\mathbf{s}}(\mathbf{s})$ over $\mathbf{s}$ which is freely reduced over $\mathbf{s}$ is cyclically $\mathbf{r}^{\ast}$-reduced, and hence $\langle \mathbf{s}\rangle_G$ is a free subgroup of $\langle \mathbf{x};\mathbf{r}\rangle$ with basis $\mathbf{s}$.
\end{lemma}

\begin{proof}
%To see that $\langle \mathbf{s}\rangle$ is a free subgroup of $\langle \mathbf{x};\mathbf{r}\rangle$ with basis $\mathbf{s}$,
Suppose that there exists a non-empty word $U:=U_{\mathbf{s}}(\mathbf{s})=S_1^{\epsilon_1}\ldots S_n^{\epsilon_n}$, $S_i\in \mathbf{s}$, which is cyclically reduced as a word over $\mathbf{s}^{\pm1}$ but is not cyclically $\mathbf{r}^{\ast}$-reduced. We obtain a contradiction.

Consider the (non-empty) diagram $\mathcal{D}_U$ given by Lemma \ref{lem:diagramD_U}.
%Fix some label $V$ of $\partial \mathcal{D}_U$, and note that $V$ is equal to the identity in $\langle \mathbf{x}; \mathbf{r}\rangle$. Therefore,
%Given a relator $R\in\mathbf{r}$, we can attempt to attach the corresponding $2$-cell $\mathcal{R}$ to the boundary of the diagram $\mathcal{D}_U$ over an attaching path $I$ with label both a subword of a label of $\partial\mathcal{R}$ and of a label of $\partial\mathcal{D}_U$ to obtain a disk diagram in $\langle \mathbf{x}; \mathbf{r}, \mathbf{s}\rangle$.
%; we say that $V\in\partial\mathcal{R}\cap\partial\mathcal{D}_U$.
As $U$ is not cyclically $\mathbf{r}^{\ast}$-reduced the small cancellation conditions on the presentation $\langle\mathbf{x}; \mathbf{r}\rangle$ mean that we may attach an $\mathbf{r}^{\ast}$-region $\mathcal{R}$ to the boundary $\partial\mathcal{D}_U$ of the diagram $\mathcal{D}_U$ over an attaching path $I$ such that $|I|>\frac12|\partial\mathcal{R}|$ \cite[Theorems V.4.5~\&~V.4.6]{L-S}.

%Firstly, suppose that $\langle \mathbf{x};\mathbf{r}, \mathbf{s}\rangle$ is a $C^{\prime}(1/6)$ small cancellation presentation. Then $I$ contains at most two $\{\mathbf{r}, \mathbf{s}\}$-pieces, one from $e_i$ and one from $e_{i+1\pmod n}$, for some $1\leq i\leq n$. Hence, the path $I$ contains less than $3/6$, and so less than half, of the boundary of $\mathcal{R}$, a contradiction.

By Lemma \ref{lem:connectingpathI}, the attaching path $I$ straddles at most two paths $p_i$ and $p_{i+1\pmod n}$ of $\partial\mathcal{D}_U$, and hence the path $I$ contains at most two $(\mathbf{r}\cup \mathbf{s})^{\ast}$-pieces. Therefore, as $\langle \mathbf{x}; \mathbf{r}, \mathbf{s}\rangle$ satisfies $C'(1/4)$ we have that $|I|<\frac24|\partial\mathcal{R}|=\frac12|\partial\mathcal{R}|$, a contradiction.
\end{proof}

%end comment

%%%%%%%%%%%%%%%%%%%%%%%%%%%%%%%%%%%%%%%%%%%%%%%%%%%%%%%%%%%%%%%%
%------------------------------------------------------------------------------------------------------------------------------------------------------------------------
%------------------------Conjugacy in small cancellation groups-------------------------------------------------------------------------------
%------------------------------------------------------------------------------------------------------------------------------------------------------------------------
%%%%%%%%%%%%%%%%%%%%%%%%%%%%%%%%%%%%%%%%%%%%%%%%%%%%%%%%%%%%%%%%

\subsection{Conjugacy in small cancellation groups}
We wish to lift conditions on conjugation and intersections of subgroups in $C'(1/6)$ and $C'(1/4)-T(4)$ small cancellation groups $\langle \mathbf{x}; \mathbf{r}\rangle$ to the analogous situation in the ambient free group $F(\mathbf{x})$. We do this in, respectively, Theorems \ref{thm:FreeMalnormMetric6} and \ref{thm:FreeMalnormMetric4T4} below. First we use annular diagrams to describe conjugation in the relevant small cancellation groups.
%These theorems lift conditions on conjugation of subgroups and intersecting subgroups in small cancellation groups to the analogous situation in the ambient free group.

%This classical theorem is summarised in Figure \ref{\mathbf{x}\mathbf{x}\mathbf{x}}.

%%%%%%%%%%%%%%%%%%%%%%%%%%%%%%%%%%%%%%%%%%%%%%%%%%%%%%%%%%%%%%%%
%------------------------------------------------------------------------------------------------------------------------------------------------------------------------
%------------------------Conjugacy problem for small cancellation groups-------------------------------------------------------------------------------
%------------------------------------------------------------------------------------------------------------------------------------------------------------------------
%%%%%%%%%%%%%%%%%%%%%%%%%%%%%%%%%%%%%%%%%%%%%%%%%%%%%%%%%%%%%%%%

\p{The conjugacy problem} Groups admitting $C'(1/6)$ or $C'(1/4)-T(4)$ presentations have soluble conjugacy problem, and we state here certain aspects of the classical proof of this result which we use in this section. (For more details of this classical proof see \cite[Section V]{L-S}.)
%. The proof of this result, as given in Lyndon and Schupp's book \cite[Section V]{L-S}, involves a classification of annular diagrams over such small cancellation presentations. This classification is a used heavilly in the current section, and hence we state the necessary results, and accompanying figures, here.

%The proof of this result, as given in Lyndon and Schupp's book \cite[Section V]{L-S}, involves a classification of annular diagrams over such small cancellation presentations. This classification is a used heavilly in the current section, and hence we state the necessary results, and accompanying figures, here.

An \emph{annular $\mathbf{r}^{\ast}$-diagram} $\mathcal{A}$ is an $\mathbf{r}^{\ast}$-diagram such that $\mathbb{R}^2\setminus\mathcal{A}$ has exactly two components. Such a diagram has an \emph{interior boundary} $\partial_i\mathcal{A}$ and an \emph{exterior boundary} $\partial_e\mathcal{A}$. Annular diagrams encode conjugation:
Let $U$ be a label of $\partial_i\mathcal{A}$ read in a clockwise direction starting from a point $a\in\partial_i\mathcal{A}$, let $V$ be a label of $\partial_e\mathcal{A}$ read in a clockwise direction from a point $b\in\partial_e\mathcal{A}$, and let $W$ be the label of a path $I$ in $\mathcal{A}$ which starts at $a$ and ends at $b$ (so $\iota(I)=a$ and $\tau(I)=b$). When $\mathcal{A}$ is split along this path $I$
%in the sense of Figure \ref{fig:annulardiagrams}
we obtain a disc diagram $\mathcal{D}$ with boundary label $UW^{-1}V^{-1}W$. Hence, two words $U$ and $V$ denote conjugate elements of the group defined by the presentation $\langle\mathbf{x};\mathbf{r}\rangle$ if and only if there exists some annular $\mathbf{r}^{\ast}$-diagram $\mathcal{A}$ such that $U$ is a label for $\partial_i\mathcal{A}$ and $V$ is a label for $\partial_e\mathcal{A}$, where labels are read in the same direction.
For $C^{\prime}(1/6)$ and $C^{\prime}(1/4)-T(4)$ small cancellation presentations there are extremely strong structural theorems for annular diagrams; it is these structural theorems we state below and apply in this section. There are two cases: either there exists some region $\mathcal{R}$ of $\mathcal{A}$ which intersects both the internal and external boundaries $\partial_i\mathcal{A}$ and $\partial_e\mathcal{A}$ of $\mathcal{A}$, or there exists no such region.
%In particular, if $\mathcal{A}$ is an annular $\mathbf{r}^{\ast}$-diagram where $\langle \mathbf{x}; \mathbf{r}\rangle$ is a $C^{\prime}(1/6)$ or $C^{\prime}(1/4)-T(4)$ small cancellation presentation then either every region of $\mathcal{A}$ has edges on both the interior boundary and the exterior boundary of $\mathcal{A}$, or no region has this property. In the case that no region has this property, every region has connected intersection with one of $\partial_i\mathcal{A}$ of $\partial_e\mathcal{A}$, and each such intersection contains at least one edge (see \cite[Figure V.5.2]{L-S}).

Firstly, we give the structure of $\mathcal{A}$ when there exists no region which intersects both the internal and external boundaries of $\mathcal{A}$. This theorem is illustrated in Figure \ref{fig:annulardiagrams}. An $\mathbf{r}^{\ast}$-diagram $\Delta$ is \emph{reduced} if it does not contain a subdiagram $\Delta'$ which consists of precisely two regions and such that the label on $\partial\Delta'$ reduces to the trivial element of the free group $F(\mathbf{x})$. If $\mathcal{R}$ is a region of an annular diagram $\mathcal{A}$, then by an \emph{interior vertex} of $\mathcal{R}$ we mean a vertex which is not contained in the boundary of $\mathcal{A}$, and by an \emph{interior piece} of $\mathcal{R}$ we mean a connected subpath $I$ of $\partial \mathcal{R}$ such that, when viewed as a subpath of $\mathcal{A}$, its initial and terminal vertices $\iota(I)$ and $\tau(I)$ have degree greater than two while every other vertex in $I$ has degree precisely two, and such that no edges of $I$ are contained in the boundary of $\mathcal{A}$.

\begin{theorem}[Lyndon--Schupp, Theorem V.5.3]
\label{thm:AnnStructureThick}
Let $G$ be given by a presentation $\mathcal{P}=\langle\mathbf{x}; \mathbf{r}\rangle$ which satisfies either
\begin{enumerate}
\item\label{case:sixth} $C^{\prime}(1/6)$, or
\item\label{case:fourthT4} $C^{\prime}(1/4)-T(4)$.
\end{enumerate}
Assume the following three hypotheses.
\begin{enumerate}[(A)]
\item $\mathcal{A}$ is a reduced annular $\mathbf{r}^{\ast}$-diagram.
\item Each label of $\partial_i\mathcal{A}$ is cyclically $\mathbf{r}^{\ast}$-reduced, and each label of $\partial_e\mathcal{A}$ is cyclically $\mathbf{r}^{\ast}$-reduced.
\item $\mathcal{A}$ does not contain a region $\mathcal{R}$ such that both $\mathcal{R}\cap\partial_i\mathcal{A}$ and $\mathcal{R}\cap\partial_e\mathcal{A}$ contain an edge.
\end{enumerate}
Let $(q, p)$ be $(3, 6)$ or $(4, 4)$ in Cases (\ref{case:sixth}) and (\ref{case:fourthT4}) respectively. Then $\mathcal{A}$ satisfies all of the following conditions:
\begin{enumerate}[(i)]
\item For each region $\mathcal{R}$, either $\mathcal{R}\cap\partial_i\mathcal{A}$ or $\mathcal{R}\cap\partial_e\mathcal{A}$ contain an edge.
\item Each region $\mathcal{R}$ contains precisely $p/q+2$ interior pieces.
\item Each interior vertex of $\mathcal{A}$ has degree precisely $q$.
\end{enumerate}
\end{theorem}

%\begin{minipage}{0.49\textwidth}
\begin{figure}
\centering
\begin{tikzpicture}
%set up decorations on lines
\begin{scope}[thick, decoration={
%markings,mark=at position 53/100 with {\arrow[line width=1pt]{stealth}}
%snake%, amplitude=.4mm,segment length=2mm,post length=1mm
}
] 
%Draw left-hand annulus C'(1.6)
\centerarc[thick](-3,0)(0:360:0.5)%(120:420:0.5)
\centerarc[thick](-3,0)(0:360:2)%(120:420:2)
\centerarc[thick](-3,0)(0:360:1.25)
%lines on LH annulus
%outer lines
\draw[postaction={decorate}]($(-3, 0)+({(1.25)*cos(20)},{(1.25)*sin(20)})$) to ($(-3, 0)+({(2)*cos(20)},{(2)*sin(20)})$);
\draw[postaction={decorate}] ($(-3, 0)+({(1.25)*cos(120)},{(1.25)*sin(120)})$) to ($(-3, 0)+({(2)*cos(120)},{(2)*sin(120)})$);
\draw[postaction={decorate}]($(-3, 0)+({(1.25)*cos(220)},{(1.25)*sin(220)})$) to ($(-3, 0)+({(2)*cos(220)},{(2)*sin(220)})$);
\draw[postaction={decorate}]($(-3, 0)+({(1.25)*cos(320)},{(1.25)*sin(320)})$) to ($(-3, 0)+({(2)*cos(320)},{(2)*sin(320)})$);
\draw[postaction={decorate}]($(-3, 0)+({(1.25)*cos(420)},{(1.25)*sin(420)})$) to ($(-3, 0)+({(2)*cos(420)},{(2)*sin(420)})$);
%inner lines
\draw[postaction={decorate}]($(-3, 0)+({(0.5)*cos(-10)},{(0.5)*sin(-10)})$) to ($(-3, 0)+({(1.25)*cos(-10)},{(1.25)*sin(-10)})$);
\draw[postaction={decorate}]($(-3, 0)+({(0.5)*cos(90)},{(0.5)*sin(90)})$) to ($(-3, 0)+({(1.25)*cos(90)},{(1.25)*sin(90)})$);
\draw[postaction={decorate}]($(-3, 0)+({(0.5)*cos(180)},{(0.5)*sin(180)})$) to ($(-3, 0)+({(1.25)*cos(180)},{(1.25)*sin(180)})$);
\draw[postaction={decorate}]($(-3, 0)+({(0.5)*cos(280)},{(0.5)*sin(280)})$) to ($(-3, 0)+({(1.25)*cos(280)},{(1.25)*sin(280)})$);
\draw[postaction={decorate}]($(-3, 0)+({(0.5)*cos(400)},{(0.5)*sin(400)})$) to ($(-3, 0)+({(1.25)*cos(400)},{(1.25)*sin(400)})$);
%draw right-hand annulus C'(1/4)-T(4)
\centerarc[thick](3,0)(0:360:0.5)
\centerarc[thick](3,0)(0:360:2)
\centerarc[thick](3,0)(0:360:1.25)
%lines on RH annulus
\draw[postaction={decorate}] ($(3, 0)+({(0.5)*cos(0)},{(0.5)*sin(0)})$) to ($(3, 0)+({(2)*cos(0)},{(2)*sin(0)})$);
\draw[postaction={decorate}] ($(3, 0)+({(0.5)*cos(90)},{(0.5)*sin(90)})$) to ($(3, 0)+({(2)*cos(90)},{(2)*sin(90)})$);
%\draw[postaction={decorate}] ($(3, 0)+({(0.5)*cos(130)},{(0.5)*sin(130)})$) to ($(3, 0)+({(2)*cos(130)},{(2)*sin(130)})$);
\draw[postaction={decorate}] ($(3, 0)+({(0.5)*cos(160)},{(0.5)*sin(160)})$) to ($(3, 0)+({(2)*cos(160)},{(2)*sin(160)})$);
%\draw[postaction={decorate}] ($(3, 0)+({(0.5)*cos(200)},{(0.5)*sin(200)})$) to ($(3, 0)+({(2)*cos(200)},{(2)*sin(200)})$);
%\draw[postaction={decorate}] ($(3, 0)+({(0.5)*cos(240)},{(0.5)*sin(240)})$) to ($(3, 0)+({(2)*cos(240)},{(2)*sin(240)})$);
%\draw[postaction={decorate}] ($(3, 0)+({(0.5)*cos(300)},{(0.5)*sin(300)})$) to ($(3, 0)+({(2)*cos(300)},{(2)*sin(300)})$);
\draw[postaction={decorate}] ($(3, 0)+({(0.5)*cos(290)},{(0.5)*sin(290)})$) to ($(3, 0)+({(2)*cos(290)},{(2)*sin(290)})$);
%labels (no used currently)
\node[below=1pt] at ($(-3, 0)+(0,-2)$) {$C'(1/6)$};
\node[below=1pt] at ($(3, 0)+(0,-2)$) {$C'(1/4)-T(4)$};
%\node[right=1pt] at ($(3, 0)+({(1.1)*cos(420)},{(1.1)*sin(120)})$) {$p$};
%\node[left=1pt] at ($(-3, 0)+({(1.1)*cos(90)},{(1.1)*sin(90)})$) {$p$};
\end{scope}
\end{tikzpicture}
\caption{Annular $C'(1/6)$ and $C'(1/4)-T(4)$ diagrams $\mathcal{A}$ with no region intersecting both the internal and external boundaries.}\label{fig:annulardiagrams}
\end{figure}
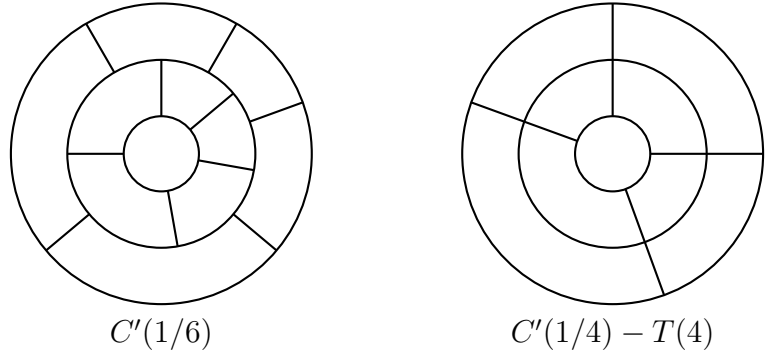
%\end{minipage}

Secondly, we give the structure of $\mathcal{A}$ when there exists some region which intersects both the internal and external boundaries of $\mathcal{A}$. This theorem is illustrated in Figure \ref{fig:AnnularDiagramForm}.
%If $A$ is an annular diagram described by Proposition \ref{thm:LS4.5} then it is split into \emph{islands}, which are reduced subdiagrams of $A$ whose boundary is not self-intersecting and of the form $\sigma\eta$ with $\sigma\subset\partial I$ and $\eta\subset\partial E$, and \emph{bridges}, which are non-trivial paths in $\partial I\cap\partial E$. An example of an annular diagram described by Proposition \ref{thm:LS4.5} is given by Figure \ref{fig:AnnularDiagramForm}.

\begin{theorem}[Lyndon--Schupp, Theorem V.5.5]
\label{thm:AnnStructureThin}
Let $G$ be given by a presentation $\mathcal{P}=\langle\mathbf{x}; \mathbf{r}\rangle$ which satisfies either $C^{\prime}(1/6)$ or $C^{\prime}(1/4)-T(4)$.
Assume the following three hypotheses.
\begin{enumerate}[(A)]
\item $\mathcal{A}$ is a reduced annular $\mathbf{r}^{\ast}$-diagram.
\item Each label of $\partial_i\mathcal{A}$ is cyclically $\mathbf{r}^{\ast}$-reduced, and each label of $\partial_e\mathcal{A}$ is cyclically $\mathbf{r}^{\ast}$-reduced.
\item $\mathcal{A}$ contains a region $\mathcal{S}$ such that both $\mathcal{S}\cap\partial_i\mathcal{A}$ and $\mathcal{S}\cap\partial_e\mathcal{A}$ contain an edge.
\end{enumerate}
Then every region $\mathcal{R}$ of $\mathcal{A}$ has edges on both $\partial_i\mathcal{A}$ and $\partial_e\mathcal{A}$, and $\mathcal{R}$ has at most two internal pieces.
\end{theorem}

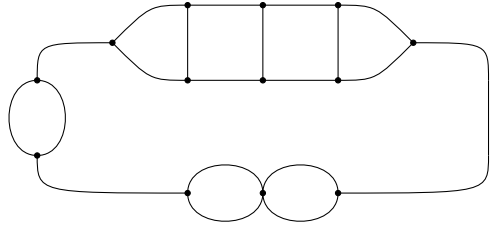
\begin{figure}[h]
\centering
\begin{tikzpicture}
%\draw[name=a] (0,0) ellipse (20pt and 10pt);
%left hand diagram
\filldraw [black] (-2, 0.5) circle (1pt);
\filldraw [black] (-2, -0.5) circle (1pt);
%top diagram
\filldraw [black] (-1, 1) circle (1pt);
\filldraw [black] (0, 1.5) circle (1pt);
\filldraw [black] (0, 0.5) circle (1pt);
\filldraw [black] (1, 1.5) circle (1pt);
\filldraw [black] (1, 0.5) circle (1pt);
\filldraw [black] (2, 1.5) circle (1pt);
\filldraw [black] (2, 0.5) circle (1pt);
\filldraw [black] (3, 1) circle (1pt);
%bottom diagram
\filldraw [black] (0, -1) circle (1pt);
\filldraw [black] (1, -1) circle (1pt);
\filldraw [black] (2, -1) circle (1pt);
%\filldraw [black] (2.5, -1) circle (1pt);
%lines (left moving clockwise)
%left diagram
\draw (-2, -0.5) .. controls (-2.5, -0.5) and (-2.5, 0.5) .. (-2, 0.5);
\draw (-2, -0.5) .. controls (-1.5, -0.5) and (-1.5, 0.5) .. (-2, 0.5);
%bridge
\draw (-2, 0.5) .. controls (-2, 1) .. (-1, 1);
%top diagram
%%top line
\draw (-1, 1) .. controls (-0.5, 1.5) .. (0, 1.5);
\draw (0, 1.5) -- (1, 1.5);
\draw (1, 1.5) -- (2, 1.5);
%\draw (1, 1.5) -- (1.25, 1.5);
%\draw[dotted] (1.25, 1.5) -- (1.75, 1.5);
%\draw (1.75, 1.5) -- (2, 1.5);
\draw (2, 1.5) .. controls (2.5, 1.5) .. (3, 1);
%%bottom line
\draw (-1, 1) .. controls (-0.5, 0.5) .. (0, 0.5);
\draw (0, 0.5) -- (1, 0.5);
\draw (1, 0.5) -- (2, 0.5);
%\draw (1, 0.5) -- (1.25, 0.5);
%\draw[dotted] (1.25, 0.5) -- (1.75, 0.5);
%\draw (1.75, 0.5) -- (2, 0.5);
\draw (2, 0.5) .. controls (2.5, 0.5) .. (3, 1);
%%vertical lines
\draw (0, 1.5) -- (0, 0.5);
\draw (1, 1.5) -- (1, 0.5);
\draw(2, 1.5) -- (2, 0.5);
%bridge
%\draw (3, 1) .. controls (3.5, 1) and (3.75, -1) .. (2, -1);
\draw (3, 1) .. controls (4, 1) .. (4, 0.5);
\draw (4, 0.5) -- (4, -0.5);
\draw (4, -0.5) .. controls (4, -1) .. (2, -1);
%bottom diagrams
\draw (2, -1) .. controls (2, -0.5) and (1, -0.5) .. (1, -1);
\draw (2, -1) .. controls (2, -1.5) and (1, -1.5) .. (1, -1);
\draw (1, -1) .. controls (1, -0.5) and (0, -0.5) .. (0, -1);
\draw (1, -1) .. controls (1, -1.5) and (0, -1.5) .. (0, -1);
%bridge
\draw (0, -1) .. controls (-2, -1) .. (-2, -0.5);
\end{tikzpicture}
\caption{An annular diagram $\mathcal{A}$ with a region intersecting both the internal and external boundaries of $\mathcal{A}$}\label{fig:AnnularDiagramForm}
\end{figure}

%%%%%%%%%%%%%%%%%%%%%%%%%%%%%%%%%%%%%%%%%%%%%%%%%%%%%%%%%%%%%%%%
%------------------------------------------------------------------------------------------------------------------------------------------------------------------------
%------------------------Main small cancellation results/arguments-------------------------------------------------------------------------------
%------------------------------------------------------------------------------------------------------------------------------------------------------------------------
%%%%%%%%%%%%%%%%%%%%%%%%%%%%%%%%%%%%%%%%%%%%%%%%%%%%%%%%%%%%%%%%

\p{Subgroup intersection}
We now apply the above structural results for annular diagrams to prove Theorems \ref{thm:FreeMalnormMetric6} and \ref{thm:FreeMalnormMetric4T4}. These theorems lift conditions on conjugation and intersections of subgroups in, respectively, $C'(1/6)$ and $C'(1/4)-T(4)$ small cancellation groups $\langle \mathbf{x}; \mathbf{r}\rangle$ to the analogous situation in the ambient free group $F(\mathbf{x})$.
%If $I$ is a path in a graph then we use $\iota(I)$ to denote the initial vertex of $I$ and $\tau(I)$ to denote the terminal vertex of $I$.

\begin{theorem}
\label{thm:FreeMalnormMetric6}
Suppose that $\langle \mathbf{x};\mathbf{r}, \mathbf{s}\rangle$ is a $C^{\prime}(1/6)$ small cancellation presentation. Let $G$ be the group given by the sub-presentation $\langle \mathbf{x};\mathbf{r}\rangle$, and let $\varphi: F(\mathbf{x})\rightarrow \langle \mathbf{x}; \mathbf{r}\rangle$ be the natural map.

Suppose that $\mathbf{t}\subset F(\mathbf{x})$ is a set of words such that every word $V_{\mathbf{t}}(\mathbf{t})$ over $\mathbf{t}$ which is freely reduced over $\mathbf{t}$ is cyclically $\mathbf{r}^{\ast}$-reduced.
%
%\marginpar{``every product in $\langle \mathbf{t}\rangle$ is Dehn reduced'' or \emph{cyclically} Dehn reduced?}
%
%Then $U(\mathbf{s})^{W(\mathbf{x})}=_{F(\mathbf{x})}V(\mathbf{t})$ if and only if $U(\mathbf{s})^{W(\mathbf{x})}=_{G}V(\mathbf{t})$.
For all $U\in\langle\mathbf{s}\rangle$ and $V\in\langle\mathbf{t}\rangle$, if there exists $g\in G$ such that $g^{-1}\varphi(U)g=_G\varphi(V)$ then there exists $W\in F(\mathbf{x})$ such that $\varphi(W)=g$ and $W^{-1}UW=_{F(\mathbf{x})} V$.
In particular, there exists $g\in G$ such that $\langle \mathbf{s}\rangle_G^g\cap \langle \mathbf{t}\rangle_G\neq_G1$ if and only if there exists $W\in F(\mathbf{x})$ such that $\langle \mathbf{s}\rangle^{W}\cap \langle \mathbf{t}\rangle\neq_{F(\mathbf{x})}1$.
\end{theorem}

\begin{proof}
First note that every word $U_{\mathbf{s}}(\mathbf{s})$ over $\mathbf{s}$ which is freely reduced over $\mathbf{s}$ is cyclically $\mathbf{r}^{\ast}$-reduced by Lemma \ref{lem:FreeSbgps}.
Therefore, suppose that $\mathcal{A}$ is an annular $\mathbf{r}^{\ast}$-diagram such that its interior boundary $\partial_i\mathcal{A}$ has a label which is a cyclic shift and free reduction (both in $F(\mathbf{x})$) of some word $U:=U_{\mathbf{s}}(\mathbf{s})=S_1^{\epsilon_1}\cdots S_n^{\epsilon_n}$, $S_i\in\mathbf{s}$, and its exterior boundary $\partial_e\mathcal{A}$ has a label which is a cyclic shift and free reduction (both in $F(\mathbf{x})$) of some word $V_{\mathbf{t}}(\mathbf{t})$. It is sufficient to prove that $\mathcal{A}$ contains no regions.

Consider the disk diagram $\mathcal{D}_U$ from Lemma \ref{lem:diagramD_U}.
Each of the paths $p_k$ on the boundary of $\mathcal{D}_U$ contains at least five $(\mathbf{r}\cup \mathbf{s})^{\ast}$-pieces.
Note that we can form an $(\mathbf{r}\cup \mathbf{s})^{\ast}$-disk diagram $\Delta$ by attaching the boundary $\partial\mathcal{D}_U$ to the internal boundary $\partial_i\mathcal{A}$ of $\mathcal{A}$.
Suppose that $\mathcal{A}$ contains at least one region; we find a contradiction.

As $\mathcal{A}$ contains a region, the boundary of $\mathcal{D}_U$ connects to a region $\mathcal{R}$ of $\mathcal{A}$ over an attaching path $I$ such that $I$ contains an edge.

First, suppose that $\mathcal{R}\cap \partial_e\mathcal{A}$ is empty.
By Theorem \ref{thm:AnnStructureThick}.(ii) we have that $|I|>\frac26|\partial\mathcal{R}|$. Hence, $I$ cannot be written as the product of fewer than three $(\mathbf{r}\cup \mathbf{s})^{\ast}$-pieces. Therefore, the path $I$ straddles a sequence of three paths $(p_{i}, p_{i+1\pmod n}, p_{i+2\pmod n})$, a contradiction by Lemma \ref{lem:connectingpathI}.

Next, suppose that $\mathcal{R}\cap\partial_e\mathcal{A}$ is non-empty. By Theorem \ref{thm:AnnStructureThin}, every region of $\mathcal{A}$ has an edge on $\partial_i\mathcal{A}$ and an edge on $\partial_e\mathcal{A}$.
Additionally, suppose that $\mathcal{R}$ does not share an edge with any other region of $\mathcal{A}$. Then (using the fact that every label of $\partial_i\mathcal{A}$ and of $\partial_e\mathcal{A}$ is cyclically $\mathbf{r}^{\ast}$-reduced) the boundary of $\mathcal{R}$ is split equally between $\partial_i\mathcal{A}$ and $\partial_e\mathcal{A}$, and so $|I|=\frac12|\partial\mathcal{R}|$. Hence $I$ cannot be written as the product of fewer than four $(\mathbf{r}\cup \mathbf{s})^{\ast}$-pieces. Therefore, the path $I$ straddles a sequence of four paths $(p_i,\ldots, p_{i+3\pmod n})$, a contradiction by Lemma \ref{lem:connectingpathI}.
On the other hand, suppose that $\mathcal{R}$ shares an edge with some region $\mathcal{S}$ of $\mathcal{A}$.
Now, as with the path $I$, the path $J:=\mathcal{S}\cap\partial_i\mathcal{A}$ contains an edge. We may assume that the terminal vertex of $I$ is the initial vertex of $J$, so $\tau(I)=\iota(J)$. The minimal possible length of $I$ corresponds to when $|\partial_e\mathcal{A}\cap\mathcal{R}|=\frac12|\partial\mathcal{R}|$ and $\mathcal{R}$ connects over a piece to an additional region $\mathcal{T}$ of $\mathcal{A}$. Hence, $|I|>(1-\frac16-\frac16-\frac12)|\partial\mathcal{R}|=\frac16|\partial\mathcal{R}|$. Hence, the attaching path $I$ cannot be written as the product of fewer than two $(\mathbf{r}\cup\mathbf{s})^{\ast}$-pieces. Similarly, $|J|>\frac16|\partial\mathcal{S}|$ so the attaching path $J$ cannot be written as the product of fewer than two $(\mathbf{r}\cup\mathbf{s})^{\ast}$-pieces. We therefore have that $I$ straddles two paths $p_{i}$ and $p_{i+1\pmod n}$, with $\iota(I)\in p_i$ and $\tau(I)\in p_{i+1\pmod n}$, and that $J$ straddles two paths $p_{j}$ and $p_{j+1\pmod n}$, with $\iota(J)\in p_j$ and $\tau(J)\in p_{j+1\pmod n}$. 
Recall that each path $p_k$ is the product of at least five $(\mathbf{r}\cup\mathbf{s})^{\ast}$-pieces. As $\tau(I)=\iota(J)$ and as the path $p_{i+1\pmod n}$ does not consist of a single $(\mathbf{r}\cup\mathbf{s})^{\ast}$-piece we have that $p_{i+1\pmod n}=p_j$. However, this means that $p_j$ decomposes into only two pieces, a contradiction.
\end{proof}

We now prove the analogous result for $C^{\prime}(1/4)-T(4)$ presentations. Note that if $\langle \mathbf{x}; \mathbf{r}, \mathbf{s}\rangle$ is a $C^{\prime}(1/4)-T(4)$ presentation then the sub-presentation $\langle \mathbf{x}; \mathbf{r}\rangle$ is also a $C^{\prime}(1/4)-T(4)$ presentation.

\begin{theorem}
\label{thm:FreeMalnormMetric4T4}
Suppose that $\langle \mathbf{x};\mathbf{r}, \mathbf{s}\rangle$ is a $C^{\prime}(1/4)-T(4)$ small cancellation presentation. Let $G$ be the group given by the sub-presentation $\langle \mathbf{x};\mathbf{r}\rangle$, and let $\varphi: F(\mathbf{x})\rightarrow \langle \mathbf{x}; \mathbf{r}\rangle$ be the natural map.

Suppose that $\mathbf{t}\subset F(\mathbf{x})$ is a set of words such that every word $V_{\mathbf{t}}(\mathbf{t})$ over $\mathbf{t}$ which is freely reduced over $\mathbf{t}$ is cyclically $\mathbf{r}^{\ast}$-reduced.
%%
%%\marginpar{``every product in $\langle \mathbf{t}\rangle$ is Dehn reduced'' or \emph{cyclically} Dehn reduced?}
%%
%Then $U(\mathbf{s})^{W(\mathbf{x})}=_{F(\mathbf{x})}V(\mathbf{t})$ if and only if $U(\mathbf{s})^{W(\mathbf{x})}=_{G}V(\mathbf{t})$.
For all $U\in\langle\mathbf{s}\rangle$ and $V\in\langle\mathbf{t}\rangle$, if there exists $g\in G$ such that $g^{-1}\varphi(U)g=_G\varphi(V)$ then there exists $W\in F(\mathbf{x})$ such that $\varphi(W)=g$ and $W^{-1}UW=_{F(\mathbf{x})} V$.
In particular, there exists $g\in G$ such that $\langle \mathbf{s}\rangle_G^g\cap \langle \mathbf{t}\rangle_G\neq_G1$ if and only if there exists $W\in F(\mathbf{x})$ such that $\langle \mathbf{s}\rangle^{W}\cap \langle \mathbf{t}\rangle\neq_{F(\mathbf{x})}1$.
\end{theorem}

\begin{proof}
First note that every word $U_{\mathbf{s}}(\mathbf{s})$ over $\mathbf{s}$ which is freely reduced over $\mathbf{s}$ is cyclically $\mathbf{r}^{\ast}$-reduced by Lemma \ref{lem:FreeSbgps}.
Therefore, suppose that $\mathcal{A}$ is an annular $\mathbf{r}^{\ast}$-diagram such that its interior boundary $\partial_i\mathcal{A}$ has a label which is a cyclic shift and free reduction (both in $F(\mathbf{x})$) of some word $U:=U_{\mathbf{s}}(\mathbf{s})=S_1^{\epsilon_1}\cdots S_n^{\epsilon_n}$, $S_i\in\mathbf{s}$, and its exterior boundary $\partial_e\mathcal{A}$ has a label which is a cyclic shift and free reduction (both in $F(\mathbf{x})$) of some word $V_{\mathbf{t}}(\mathbf{t})$. It is sufficient to prove that $\mathcal{A}$ contains no regions.

Consider the disk diagram $\mathcal{D}_U$ from Lemma \ref{lem:diagramD_U}.
Each of the paths $p_k$ on the boundary of $\mathcal{D}_U$ contains at least three $(\mathbf{r}\cup \mathbf{s})^{\ast}$-pieces.
Note that we can form an $(\mathbf{r}\cup \mathbf{s})^{\ast}$-disk diagram $\Delta$ by attaching the boundary $\partial\mathcal{D}_U$ to the internal boundary $\partial_i\mathcal{A}$ of $\mathcal{A}$.
Suppose that $\mathcal{A}$ contains at least one region; we find a contradiction.

As $\mathcal{A}$ contains a region, the boundary of $\mathcal{D}_U$ connects to a region $\mathcal{R}$ of $\mathcal{A}$ over an attaching path $I$ such that $I$ contains an edge.

First, suppose that $\mathcal{R}\cap \partial_e\mathcal{A}$ is empty.
By Theorem \ref{thm:AnnStructureThick}.(ii) we have that $|I|>\frac14|\partial\mathcal{R}|$.
Hence, $I$ cannot be written as the product of fewer than two $(\mathbf{r}\cup \mathbf{s})^{\ast}$-pieces. Therefore, the path $I$ straddles two paths $p_{i}$, $p_{i+1\pmod n}$. Then the vertex $v=p_{i}\cap p_{i+1\pmod n}$ is an internal vertex of a $C^{\prime}(1/4)-T(4)$ disc diagram which has degree three, a contradiction.

Next, suppose that $\mathcal{R}\cap\partial_e\mathcal{A}$ is non-empty. By Theorem \ref{thm:AnnStructureThin}, every region of $\mathcal{A}$ has an edge on $\partial_i\mathcal{A}$ and an edge on $\partial_e\mathcal{A}$.
%Additionally, suppose that $\mathcal{R}$ does not share a $1$-cell with any other $2$-cell of $\mathcal{A}$. Then the attaching path $I$ consists of precisely $\frac12\partial\mathcal{R}$, and hence cannot be written as the product of fewer than three $\{\mathbf{r}, \mathbf{s}\}$-pieces. Therefore, the path $I$ straddles three paths $e_i,e_{i+1\pmod n}, e_{i+2\pmod n}$, a contradiction by Lemma \ref{lem:connectingpathI}.
Additionally, suppose that $\mathcal{R}$ does not share an edge with any other region of $\mathcal{A}$. Then (using the fact that every label of $\partial_i\mathcal{A}$ and of $\partial_e\mathcal{A}$ is cyclically $\mathbf{r}^{\ast}$-reduced) the boundary of $\mathcal{R}$ is split equally between $\partial_i\mathcal{A}$ and $\partial_e\mathcal{A}$, and so $|I|=\frac12|\partial\mathcal{R}|$. Hence $I$ cannot be written as the product of fewer than three $(\mathbf{r}\cup \mathbf{s})^{\ast}$-pieces. Therefore, the path $I$ straddles a seqence of three paths $(p_i, p_{i+1\pmod n}, p_{i+2\pmod n})$, a contradiction by Lemma \ref{lem:connectingpathI}.
On the other hand, suppose that $\mathcal{R}$ shares an edge with some region $\mathcal{S}$ of $\mathcal{A}$.
Now, as with the path $I$, the path $J:=\mathcal{S}\cap\partial_i\mathcal{A}$ contains an edge.
We may assume that the terminal vertex of $I$ is the initial vertex of $J$, so $\tau(I)=\iota(J)$, and note that this vertex is an internal vertex of the disc diagram.
Therefore, the vertex $\tau(I)=\iota(J)$ has degree at least $4$, by the $T(4)$ condition. Hence, there is a path $p_i$ of $\mathcal{D}_U$ such that either $p_i$ is a subpath of $I$ or $I$ is a subpath of $p_i$. Recall that each path $p_k$ is the product of at least three $(\mathbf{r}\cup\mathbf{s})^{\ast}$-pieces. If $p_i$ is a subpath of $I$ then $p_i$ consists of a single $(\mathbf{r}\cup\mathbf{s})^{\ast}$-piece, a contradiction. If $I$ is a subpath of $p_i$ then $|I|<\frac14|\partial\mathcal{R}|$. Now, $|\partial_e\mathcal{A}\cap\mathcal{R}|\leq\frac12|\partial\mathcal{R}|$, and $|\mathcal{R}\cap\mathcal{S}|<\frac14|\partial\mathcal{R}|$, hence there exists an additional region $\mathcal{T}$ of $\mathcal{A}$ which connects over a piece to $\mathcal{R}$, and where $K:=\mathcal{T}\cap\mathcal{D}_U$ is a non-empty path. We may assume that $\tau(K)\in I$, and so $\tau(K)=\iota(I)$. The vertex $\tau(K)=\iota(I)$ is an internal vertex of the disk diagram $\Delta$, and hence has degree four by the $T(4)$ condition. Thus $I=p_i$, but then $p_i$ consists of a single $(\mathbf{r}\cup\mathbf{s})^{\ast}$-piece, a contradiction.
%The attaching path $I$ cannot be written as the product of fewer than two pieces. Similarly, the attaching path $J$ cannot be written as the product of fewer than two pieces. We therefore have that $I$ straddles two paths $e_{i}$ and $e_{i+1\pmod n}$, with $\iota(I)\in e_i$ and $\tau(I)\in e_{i+1\pmod n}$, and that $J$ straddles two paths $e_{j}$ and $e_{j+1\pmod n}$, with $\iota(J)\in e_j$ and $\tau(J)\in e_{j+1\pmod n}$. As $\tau(I)=\iota(J)$, and as the path $e_{i+1\pmod n}$ does not consist of a single piece we have that $e_{i+1\pmod n}=e_j$. However, this means that $e_j$ decomposes into only two pieces, a contradiction.
%\marginpar{This proof should be easy, as each $e_i$ has three pieces. Consider a vertex $v$ in the piece decomposition of $e_i$. Then $v\in\partial_i\mathcal{A}$ and $v$ has degree four in $\mathcal{A}$. This is not possible by the structural theorem for annular $C^{\prime}(1/4)-T(4)$ diagrams.}
%
\end{proof}

\p{Proof of Theorem \ref{thm:FreeMalnormMetric}}
We now prove Theorem \ref{thm:FreeMalnormMetric}.

\begin{proof}[Proof of Theorem \ref{thm:FreeMalnormMetric}]
As the set $\mathbf{s}$ contains no proper powers and satisfies the $C(5)$ small cancellation condition, the set $\mathbf{s}$ also satisfies Wise's $c(5)$ small cancellation condition \cite[Definition 2.3]{wise2001residual}. Hence, $\langle \mathbf{s}\rangle$ is a malnormal in $F(\mathbf{x})$ with minimal basis $\mathbf{s}$ \cite[Theorems 2.11 \& 2.14]{wise2001residual}.
The result then follows by taking $\mathbf{t}:=\mathbf{s}$ in Theorem \ref{thm:FreeMalnormMetric6} or in Theorem \ref{thm:FreeMalnormMetric4T4}, as appropriate.
\end{proof}

\p{\boldmath{$C(m)-T(4)$} does not work}
The $C^{\prime}(1/6)$ and $C^{\prime}(1/4)-T(4)$ conditions are needed in Theorems \ref{thm:FreeMalnormMetric6} and \ref{thm:FreeMalnormMetric4T4} because they are metric conditions, in the sense that pieces have bounded length. The following example demonstrates that the conclusions of these two theorems do not hold in general if the set $\mathbf{r}\cup\mathbf{s}$ satisfies the non-metric condition $C(m)-T(4)$ for any $m$, even if $\mathbf{r}$ satisfies $C^{\prime}(\lambda)$ for arbitrarily small $\lambda$: Let $\mathbf{x}=\{a, b, x_1, y_1, \ldots, x_{2p}, y_{2p}\}$, and take
\begin{align*}
R&=[x_1, y_1]\cdots[x_{2p}, y_{2p}],~\textnormal{and}\\
S&=[x_1, y_1]\cdots[x_p, y_p]abab^2\cdots ab^{q}.
\end{align*}
Then $\langle \mathbf{x}; R\rangle$ has the $C^{\prime}\left(1/(2p-1)\right)$ small cancellation condition while $\langle \mathbf{x}; R, S\rangle$ has the $C(m)-T(4)$ condition, where $m$ is arbitrarily large and dependent on $p$ and $q$ (note that $\langle\mathbf{x}; R, S\rangle$ has $T(4)$ as $R$ and $S$ are positive words). Note that $R$ has a piece $[x_1, y_1]\cdots[x_p, y_p]$ which has length $\frac12|R|$. Then the words $S$ and $T=abab^2\cdots ab^{q}\left([x_{p+1}, y_{p+1}]\cdots[x_{2p}, y_{2p}]\right)^{-1}$ are both $\mathbf{r}^{\ast}$-reduced, as are all their powers, and neither of $S$ nor $T$ is a proper power of any element of $\mathbf{x}$. Moreover, $S$ and $T$ are conjugate in $\langle \mathbf{x}; \mathbf{r}\rangle$ but are not freely conjugate. Hence, for all $W(\mathbf{x})\in F(\mathbf{s})$ we have that $\langle\mathbf{s}\rangle\cap\langle\mathbf{t}\rangle^{W(\mathbf{x})}=_{F(X)}1$, but there exists $g\in G$ such that $\langle\mathbf{s}\rangle_G\cap\langle\mathbf{t}\rangle_G^g\neq_G1$.

\begin{comment}%begin comment
\p{What to do with this example?!?}
We record a counter-example which demonstrate that weakening the graphical $c^{\prime}(1/8)$ condition on $\langle \mathbf{x}; \mathbf{r},\mathbf{s}\rangle$ to the graphical $c^{\prime}(1/6)$ condition does not imply the conclusions of Theorem \ref{thm:ConjugacyMetricStar}. Take $\mathbf{x}=\{a, b, c, d, x, y, z\}$, $\mathbf{r}=\{bxcdc^{-1}z^{-1}b, cydba^{-1}x^{-1}c, dzacb^{-1}y^{-1}d\}$. Then $\langle \mathbf{x}; \mathbf{r}\rangle$ is a $C^{\prime}(1/6)$ presentation (all pieces have length $1$). Note that every element of the subgroup $\langle a, b, c, d\rangle$ is cyclically $\mathbf{r}^{\ast}$-reduced. Take $\mathbf{s}=\{abc^2d^2ba^{-1}\}$. Then $\langle \mathbf{x}; \mathbf{r}, \mathbf{s}\rangle$ satisfies the graphical $c^{\prime}(1/6)$ condition, but the words $abc^2d^2ba^{-1}$ and $acb^{-1}dba^{-1}cdc^{-1}$ are not freely conjugate but are conjugate in $\langle \mathbf{x}; \mathbf{r}, \mathbf{s}\rangle$.
\end{comment}%end comment

%\p{Generalising Theorem \ref{thm:FreeMalnormMetric}}
%Therefore, the conclusions of Theorems \ref{thm:FreeMalnormMetric6} and \ref{thm:FreeMalnormMetric4T4} do not extend to the non-metric small cancellation conditions. However, these conclusions are stronger than the conclusion of Theorem \ref{thm:FreeMalnormMetric}. We prove in a later paper that Theorem \ref{thm:FreeMalnormMetric} and Lemma \ref{lem:FreeSbgps} are intricately related and they both hold when $\langle \mathbf{x}; \mathbf{r}, \mathbf{s}\rangle$ is $C(6)$ (in fact, when a much more general condition holds). We do not prove this more general result here as the methods do not extend to the proofs of Theorems \ref{thm:FreeMalnormMetric6}~and \ref{thm:FreeMalnormMetric4T4}, and we use the full power of these two theorems in this paper.

\p{Wise's graphical small cancellation theory}
In the classical small cancellation theory, which underlies Theorems \ref{thm:FreeMalnormMetric6} and \ref{thm:FreeMalnormMetric4T4}, we ensure that certain conditions hold for all cyclic shifts of every $R\in(\mathbf{r}\cup\mathbf{s})^{\pm1}$ (the symmetrised closure $(\mathbf{r}\cup\mathbf{s})^{\ast}$). This is more powerful than is necessary for Theorems \ref{thm:FreeMalnormMetric6} and \ref{thm:FreeMalnormMetric4T4}. In particular, we do not need all cyclic shifts of the elements of $\mathbf{s}^{\pm1}$ because the diagram $\mathcal{D}_U$, given by Lemma \ref{lem:diagramD_U}, is constructed only by elements of $\mathbf{s}^{\pm1}$.

We therefore introduce the following conditions, which sit between the classical small cancellation conditions and Wise's graphical small cancellation conditions \cite{wise2001residual} (see also Gromov \cite{Gromov2003random}). We state below two theorems which use these new conditions to generalise Theorems \ref{thm:FreeMalnormMetric6} and \ref{thm:FreeMalnormMetric4T4}. We do not prove these analogous theorems as the proofs are essentially identical to the proofs of Theorems \ref{thm:FreeMalnormMetric6} and \ref{thm:FreeMalnormMetric4T4}. This is because Wise's small cancellation conditions involve the use of a ``distinguished vertex'' for the boundary of each region $\mathcal{R}$, but these new conditions allow for this vertex to be placed arbitrarily on the boundary of $\mathcal{R}$. Therefore, in the proofs of Theorems \ref{thm:FreeMalnormMetric6} and \ref{thm:FreeMalnormMetric4T4} we can first ``place'' the distinguished vertex of the region $\mathcal{R}$ where we want it to be in $\partial\mathcal{R}$ (for example, we can choose that the distinguished vertex will be an internal vertex of the annular diagram $\mathcal{A}$ of degree at least three, or will be the initial vertex $\iota(I)$ of the attaching path $I$).

The new conditions are as follows; they can be roughly though of as Wise's graphical small cancellation conditions involving the symmetrised closure of a subset of the relators. Let $\mathbf{r}=\{R_1, \ldots, R_n\}$. We say that the presentation $\langle \mathbf{x}; \mathbf{r}, \mathbf{s}\rangle$ has the graphical $c^{\prime}_{\mathbf{r}^{\ast}}(\lambda)$ condition if every set $\{R_1^{W_1}, \ldots, R_n^{W_n}\}\cup\mathbf{s}$, where each $R_i^{W_i}$ is cyclically reduced, has the graphical $c^{\prime}(\lambda)$ condition. The $c_{\mathbf{r}^{\ast}}(m)$ and $t_{\mathbf{r}^{\ast}}(n)$ conditions are defined analogously.

\begin{theorem}
\label{thm:ConjugacyNon-Metric6}
Suppose that $G$ is given by the $C^{\prime}(1/6)$ small cancellation presentation $\langle \mathbf{x};\mathbf{r}\rangle$, and let $\varphi: F(\mathbf{x})\rightarrow \langle \mathbf{x}; \mathbf{r}\rangle$ be the natural map. Suppose that $\langle \mathbf{x};\mathbf{r}, \mathbf{s}\rangle$ is a graphical $c^{\prime}_{\mathbf{r}^{\ast}}(1/6)$ small cancellation presentation.

Suppose that $\mathbf{t}\subset F(\mathbf{x})$ is a set of words such that every word $V_{\mathbf{t}}(\mathbf{t})$ over $\mathbf{t}$ which is freely reduced over $\mathbf{t}$ is cyclically $\mathbf{r}^{\ast}$-reduced.
%
%\marginpar{``every product in $\langle \mathbf{t}\rangle$ is Dehn reduced'' or \emph{cyclically} Dehn reduced?}
%
For all $U\in\langle\mathbf{s}\rangle$ and $V\in\langle\mathbf{t}\rangle$, if there exists $g\in G$ such that $g^{-1}\varphi(U)g=_G\varphi(V)$ then there exists $W\in F(\mathbf{x})$ such that $\varphi(W)=g$ and $W^{-1}UW=_{F(\mathbf{x})} V$.
In particular, there exists $g\in G$ such that $\langle \mathbf{s}\rangle_G^g\cap \langle \mathbf{t}\rangle_G\neq_G1$ if and only if there exists $W\in F(\mathbf{x})$ such that $\langle \mathbf{s}\rangle^{W}\cap \langle \mathbf{t}\rangle\neq_{F(\mathbf{x})}1$.
\end{theorem}

\begin{theorem}
\label{thm:ConjugacyNon-Metric4T4}
Suppose that $G$ is given by the $C^{\prime}(1/4)-T(4)$ small cancellation presentation $\langle \mathbf{x};\mathbf{r}\rangle$, and let $\varphi: F(\mathbf{x})\rightarrow \langle \mathbf{x}; \mathbf{r}\rangle$ be the natural map. Suppose that $\langle \mathbf{x};\mathbf{r}, \mathbf{s}\rangle$ is a graphical $c^{\prime}_{\mathbf{r}^{\ast}}(1/4)-t_{\mathbf{r}^{\ast}}(4)$ small cancellation presentation.

Suppose that $\mathbf{t}\subset F(\mathbf{x})$ is a set of words such that every word $V_{\mathbf{t}}(\mathbf{t})$ over $\mathbf{t}$ which is freely reduced over $\mathbf{t}$ is cyclically $\mathbf{r}^{\ast}$-reduced.
%
%\marginpar{``every product in $\langle \mathbf{t}\rangle$ is Dehn reduced'' or \emph{cyclically} Dehn reduced?}
%
For all $U\in\langle\mathbf{s}\rangle$ and $V\in\langle\mathbf{t}\rangle$, if there exists $g\in G$ such that $g^{-1}\varphi(U)g=_G\varphi(V)$ then there exists $W\in F(\mathbf{x})$ such that $\varphi(W)=g$ and $W^{-1}UW=_{F(\mathbf{x})} V$.
In particular, there exists $g\in G$ such that $\langle \mathbf{s}\rangle_G^g\cap \langle \mathbf{t}\rangle_G\neq_G1$ if and only if there exists $W\in F(\mathbf{x})$ such that $\langle \mathbf{s}\rangle^{W}\cap \langle \mathbf{t}\rangle\neq_{F(\mathbf{x})}1$.
\end{theorem}

\section{Malcharacteristic free subgroups of triangle groups}
\label{sec:MalcharTriangle}
In this section we obtain examples of malcharacteristic subgroups of triangle groups. In particular, Lemma \ref{lem:malcharTriange} proves that the subgroup $M$ of $T_{i, j, k}$, $i, j, k\geq6$, stated in Section \ref{sec:Malchar} is malcharacteristic and free of rank two. At the end of the section we apply this to prove Theorem \ref{thm:intro1} and Theorem \ref{thm:GeneralTriangle1}.

\p{\boldmath{$M$} is free and malnormal}
Recall from Section \ref{sec:Malchar} that the subgroup $M$ of $T_{i, j, k}$ is defined as $M:=\langle x, y\rangle$, with $x$ and $y$ as follows with $\rho\gg\max(i,j,k)$.
\begin{align*}
x&:=(ab^{-1})^{3}(a^2b^{-1})^{3}(ab^{-1})^3(a^2b^{-1})^4\ldots (ab^{-1})^3(a^2b^{-1})^{\rho+2}\\
y&:=(ab^{-1})^3(a^2b^{-1})^{\rho+3}(ab^{-1})^3(a^2b^{-1})^{\rho+4}\ldots (ab^{-1})^{3}(a^2b^{-1})^{2\rho+2}
\end{align*}
We now prove that this subgroup $M$ of $T_{i, j, k}$ is free of rank two and malnormal.

\begin{lemma}\label{lem:malnormal}
The subgroup $M=\langle x, y\rangle$ of $T_{i, j, k}=\langle a, b; a^i, b^j, (ab)^k\rangle$, $i, j, k\geq6$, is a malnormal subgroup of $T_{i,j ,k}$ and is free of rank two.
\end{lemma}

\begin{proof}
The set $\{a^i, b^j, (ab)^k, x, y\}$ has the $C^{\prime}(1/4)-T(4)$ small cancellation condition. The result then follows from Theorem \ref{thm:FreeMalnormMetric}.
\end{proof}

\p{The automorphisms of \boldmath{$T_{i, j, k}$}}
After Lemma \ref{lem:malnormal}, in order to prove that $M$ is malcharacteristic in $T_{i, j, k}$ it is sufficient to prove that if $\delta\in\aut(T_{i, j, k})$ is such that $\delta(M)\cap M\neq1$ then $\delta\in\inn(T_{i, j, k})$. We therefore need to understand the outer automorphism groups of triangle groups; we do this in Lemmas \ref{lem:ExtendingZieschang} and \ref{lem:transversal}. The proof of Lemma \ref{lem:ExtendingZieschang} follows closely certain proofs in a paper of Zieschang \cite{zieschang1976triangle} (but Lemma \ref{lem:ExtendingZieschang} does not follow from Zieschang's paper).
%Note that $\displaystyle\left(\begin{array}{ccc}i&j&k\\t_1&t_2&t_3\end{array}\right)$ is a permutation if and only if the groups $T_{i, j, k}$ and $T_{t_1, t_2, t_3}$ are exactly the same, up to permutation of generators.
%Recall that a word $W$ is Dehn reduced if it is $\mathbf{r}^{\ast}$-reduced where the set of relators $\mathbf{r}$ is understood.

\begin{lemma}
\label{lem:ExtendingZieschang}
Suppose that $T_{i, j, k}$, with $i, j, k\geq6$, and $T_{t_1, t_2, t_3}$ define the same presentation, up to a permutation of the generators. Write $G=\langle a, b; a^i, b^j, (ab)^k\rangle$ and $H=\langle x_1, x_2, x_3; x_1^{t_1}, x_2^{t_2}, x_3^{t_3}, x_1x_2x_3\rangle$. If $f: G\rightarrow H$ is an isomorphism then there exists an inner automorphism $\gamma\in\inn(H)$, integers $p, q\in\{1, 2, 3\}$ with $p\neq q$, and $\epsilon=\pm1$ such that $f\gamma(a)=x_{p}^{\epsilon}$ and $f\gamma(b)=x_q^{\epsilon}$.
\end{lemma}

\begin{proof}
We may assume $f(a)=U_1x_p^{\xi_p}U_1^{-1}$ and $f(b)=U_2x_q^{\xi_q}U_2^{-1}$, with $p, q\in\{1, 2, 3\}$ not necessarily distinct, and with $|\xi_p|\leq\frac12t_p$ and $|\xi_q|\leq\frac12t_q$ \cite[Satz IV.12]{Zieschang1970planar} (see also \cite[Theorem 4.3.2]{Fine1999algebraic}). Therefore, there exists an inner automorphism $\gamma\in\inn(H)$ such that $f\gamma(a)=x_p^{\xi_p}$ and $f\gamma(b)=Ux_q^{\xi_q}U^{-1}$, where $U$ does not begin with $x_{p}^{\pm1}$ and does not end with $x_q^{\pm1}$.
We view $H$ using the $C^{\prime}(1/4)-T(4)$ presentation $\mathcal{P}_H=\langle x_p, x_q; x^{t_p}, x^{t_q}, (x_px_q)^{t_r}\rangle$. Let $\mathbf{s}=\{x^{t_p}, x^{t_q}, (x_px_q)^{t_r}\}$. Assume that the word $U$ is Dehn reduced (that is, $U$ is $\mathbf{s}^{\ast}$-reduced).

Suppose that $U$ is not the empty word. Now, the word $x_p^{\xi_p}Ux_q^{\xi_q}U^{-1}$ has finite order in $H$, and so there exist some $n>1$ such that $(x_p^{\xi_p}Ux_q^{\xi_q}U^{-1})^n$ contains a subword $W$ of some $R\in\mathbf{s}^{\ast}$ with $|W|>\frac23|R|$ \cite[Theorem V.4.4.ii]{L-S}. Then $W$ must contain letters from $U$ or $U^{-1}$ and at least one letter from $x_p^{\xi_p}$ or $x_q^{\xi_q}$. Therefore, $R$ is conjugate to $(x_px_q)^{\pm t_r}$. Since $t_q>2$, the word $x_q^{-\epsilon_q}x_p^{\xi_p}x_q^{\epsilon_q}$ is not contained in $W$, and hence if $W$ contains a part of $x_p^{\xi_p}$ then $W$ contains a single letter from $x_p^{\xi_p}$, and this is either at the start or the end of $W$. Similarly, since $t_p>2$ the word $x_p^{-\epsilon_p}x_q^{\xi_q}x_p^{\epsilon_p}$ is not contained in $W$ and so the analogous statement holds for $W$ containing a part of $x_q^{\xi_q}$.

We therefore have two cases: either $W$ contains $x_{s_1}^{\pm 1}$ from $x_{s_1}^{\xi_{s_1}}$ and the other letters are from one of $U$ or $U^{-1}$, or $W$ begins with $x_{s_1}^{\pm1}$ from $x_{s_1}^{\xi_{s_1}}$, ends with $x_{s_2}^{\pm1}$ from $x_{s_2}^{\xi_{s_2}}$, $s_1\neq s_2$, and all other letters are from precisely one of $U$ or $U^{-1}$. In the first case, $U$ contains a part $W^{\prime}$ from the defining relation $(x_px_q)^{t_r}$ with $|W^{\prime}|=|W|-1>\frac23\cdot2t_r-1>t_r=\frac12|R|$ (as $t_r>3$). Hence, $U$ is not Dehn reduced, a contradiction. In the second case, $W$ is again part of the relation $(x_px_q)^{\pm t_r}$, therefore $|U|=|W|-2>\frac23\cdot2t_r-2\geq t_r$ (as $t_r\geq 6$), and so again $U$ is not Dehn reduced, a contradiction. Therefore, $U$ is the empty word, and we have $f\gamma(a)=x_p^{\xi_p}$ and $f\gamma(b)=x_q^{\xi_q}$.
%\marginpar{Really should have $p=1$, $q=2$...}
%, the word $u$ is empty \cite[Proof of Proposition 3.1]{zieschang1976triangle} (in this citation, Zieschang assumes that $p=q$; his proof is easily altered to make no use of this hypothesis). Hence, $f\gamma(a)=x_p^{\xi_p}$ and $f\gamma(b)=x_q^{\xi_q}$. If $p=q$ then $H$ is cyclic, a contradiction.

We now prove that there exists some $\epsilon=\pm1$ such that $\xi_p=\epsilon=\xi_q$, which proves the result.
Suppose otherwise.
Then the longest possible subword of a relator contained in the word $(x_p^{\xi_p}x_q^{\xi_q})^{t_r}$ is either $x_p^{\epsilon_p}x_q^{\epsilon_q}x_p^{\epsilon_p}$ or $x_q^{\epsilon_q}x_p^{\epsilon_p}x_q^{\epsilon_q}$, so the word $(x_p^{\xi_p}x_q^{\xi_q})^{t_r}$ is Dehn reduced. As this word is non-empty and Dehn reduced it is non-trivial, a contradiction.
%Then the word $x_p^{\xi_p}x_q^{\xi_q}$ has finite order in $H$, and so some power $(x_p^{\xi_p}x_q^{\xi_q})^d$ is not Dehn reduced in $\mathcal{P}_H$. As $|\xi_p|\leq\frac12t_p$ and $|\xi_q|\leq\frac12t_q$, some power $(x_p^{\xi_p}x_q^{\xi_q})^d$ contains $(x_px_qx_px_q)^{\epsilon2}$. However, as there does not exist $\epsilon=\pm1$ with $\xi_p=\epsilon=\xi_q$, for all $d>0$ the longest possible subword of a cyclic shift of $(x_px_q)^{\pm t_r}$ contained in $(x_p^{\xi_p}x_q^{\xi_q})^d$ is either $x_p^{\epsilon_p}x_q^{\epsilon_q}x_p^{\epsilon_p}$ or $x_q^{\epsilon_q}x_p^{\epsilon_p}x_q^{\epsilon_q}$, a contradiction.
%\marginpar{$C^{\prime}-T(4)$, or something else?}
\end{proof}

Write $\Psi$ for the set consisting of the following twelve maps. Note that each map defines an automorphism of $F(a, b)$.
\begin{align*}
\psi_{(1, \epsilon)}: \;a&\mapsto a^{\epsilon} &\psi_{(2, \epsilon)}:\; a &\mapsto a^{\epsilon}&\psi_{(3, \epsilon)}:\;a &\mapsto (ab)^{\epsilon}\\
b&\mapsto b^{\epsilon} &b&\mapsto (ab)^{-\epsilon} &b &\mapsto b^{-\epsilon}
\\\\
\psi_{(4, \epsilon)}: \;a&\mapsto b^{\epsilon} &\psi_{(5, \epsilon)}:\; a &\mapsto b^{\epsilon}&\psi_{(6, \epsilon)}:\;a &\mapsto (ab)^{\epsilon}\\
b&\mapsto a^{\epsilon} &b&\mapsto (ab)^{-\epsilon} &b &\mapsto a^{-\epsilon}
\end{align*}

\begin{lemma}
\label{lem:transversal}
Let $i, j, k\geq6$. Then:
\begin{enumerate}
\item If $i, j, k$ are pairwise non-equal then $\Psi_{i, j, k}:=\{\psi_{(1, 1)}, \psi_{(1, -1)}\}$ is a transversal for $\out(T_{i, j, k})$.
\item If $i=j\neq k$ then $\Psi_{i, i, k}:=\{\psi_{(1, 1)}, \psi_{(1, -1)}, \psi_{(4, 1)}, \psi_{(4, -1)}\}$ is a transversal for $\out(T_{i, j, k})$.
\item If $i=j=k$ then $\Psi$ is a transversal for $\out(T_{i, i, i})$.
\end{enumerate}
\end{lemma}

\begin{proof}
%Let $T$ denote one of $T_{i, j, k}, T_{i, i, k},T_{i, i, i}$.
By Lemma \ref{lem:ExtendingZieschang}, if $\delta\in\aut(T_{i, j, k})$ then there exists some $\gamma\in\inn(T_{i, j, k})$ such that $\delta\gamma\in\Psi$. The result follows by checking which elements of $\Psi$ define automorphisms of $T_{i, j, k}$ for each of the three cases.
\end{proof}

\p{The images of \boldmath{$x$} and \boldmath{$y$} under automorphisms}
Let $\mathbf{s}:=\{x, y\}\subset F(a, b)$, with $x$ and $y$ words as defined at the start of this section.
%By combining Theorem \ref{thm:FreeMalnormMetric4T4} with Lemma \ref{lem:malcharFree}, and as $M$ is malnormal by Lemma \ref{lem:malnormal}, in order to prove that if $\delta(M)\cap M\neq 1$ then $\delta\in\inn(T_{i, j, k})$ it is sufficient
Our final step before proving that $M$ is malcharacteristic in $T_{i, j, k}$ is to prove that for each $\psi\in\Psi$, every word $W_{\psi(\mathbf{s})}(\psi(\mathbf{s}))$ over $\psi(\mathbf{s})$ which is freely reduced over $\psi(\mathbf{s})$ is cyclically Dehn reduced in $T_{i, j, k}$. This allows us to apply Theorem \ref{thm:FreeMalnormMetric4T4} to $\psi(\mathbf{s})$. Lemmas \ref{lem:automorphicorbit1}--\ref{lem:automorphicorbit3} now analyse these words $W_{\psi(\mathbf{s})}(\psi(\mathbf{s}))$.

%%FOLLOWING PARAGRAPH WAS IN ORIGINAL SUBMITTED VERSION!!!%%%
%Every map $\psi\in\Psi$ extends naturally to the whole group $T_{i, j, k}$, even if $\psi\not\in\aut(T_{i, j, k})$. That there are maps $\psi\in \Psi$ which are not necessarily automorphisms of the groups involved does not affect the proofs of Lemmas \ref{lem:automorphicorbit1}--\ref{lem:automorphicorbit3}.
%%END PARAGRAPH%%

%As we observed in Lemma \ref{lem:transversal}, a map $\psi\in\Psi$ is not necessarily an automorphism of $T_{i, j, k}$. This fact does not affect the proofs of Lemmas \ref{lem:automorphicorbit1}--\ref{lem:automorphicorbit3}.
%, even if it is not an automorphism of $T_{i, j, k}$.
%In the following proofs we consider a word $W$ which is a word over $x^{\pm1}$ and $y^{\pm1}$, and hence also over $a^{\pm1}$ and $b^{\pm1}$. Note that $W$ may be freely reduced over $x^{\pm1}$ and $y^{\pm1}$ but not over $a^{\pm1}$ and $b^{\pm1}$. We therefore refer to ``free reduction over $F(a, b)$'' to mean free reduction of the word $W$ when viewed in $F(a, b)$. After this free reduction, the resulting word may not be a word over $x^{\pm1}$ and $y^{\pm1}$ any more.
%If $U$ and $V$ are freely reduced words in $F(a, b)$ then we say that a word $W(U, V)$ over $U$ and $V$ is freely reduced if no free reduction occurs when forming $W(U, V)$.

%\marginpar{In Lemmas \ref{lem:automorphicorbit1} and \ref{lem:automorphicorbit2}:\\ no ``...'',\\  and use ``subsemigroup''?}

\begin{lemma}
\label{lem:automorphicorbit1}
Let $T_{i, j, k}=\<a, b; a^i, b^j, (ab)^k\>$, with $i, j, k\geq6$. Suppose $\psi$ is a map contained in the set $\Psi$, and we shall write $A:=\psi(a)\psi(b)^{-1}$, $B:=\psi(a)^2\psi(b)^{-1}$. Then for all $p, q\geq 0$ the word $B^pA^3B^q$ is freely reduced over $F(a, b)$ and does not contain $a^{\pm4}$, $b^{\pm4}$, $((ab)^3a)^{\pm1}$ or $(b(ab)^3)^{\pm1}$, and so is Dehn reduced.
\end{lemma}

\begin{proof}
The proof is by inspection of the appropriate words. Indeed the following words are all freely reduced over $F(a, b)$ for $p, q\geq 3$ with the longest subwords of relators occurring being $a^3\leq a^i$, $b\leq b^j$ and $babab\leq (ab)^k$. Each word represents $B^pA^3B^{q}$ for the indicated $\psi\in\Psi$.
\begin{align*}
&\psi_{(1, 1)}:
&(a^2b^{-1})^{p}(ab^{-1})^3(a^2b^{-1})^q\\
&\psi_{(2, 1)}:
&(a^3b)^{p}(a^2b)^3(a^3b)^q\\
&\psi_{(3, 1)}:
&(abab^2)^{p}(ab^2)^3(abab^2)^q\\
\end{align*}
Now, if the result holds for $\psi_{(l, 1)}$, $l\in\{1, 2, 3\}$ then it also holds for $\psi_{(l+3, 1)}$ (as $\psi_{(l, 1)}$ and $\psi_{(l+3, 1)}$ differ just by a switch of $a$ and $b$, and possibly by an inner automorphism). Further, if the result holds for $\psi_{(l, 1)}$ then it also holds for $\psi_{(l, -1)}$. Hence, the proof is complete.
\end{proof}

\begin{lemma}
\label{lem:automorphicorbit2}
Let $T_{i, j, k}=\<a, b; a^i, b^j, (ab)^k\>$, with $i, j, k\geq6$. Suppose $\psi$ is a map contained in the set $\Psi$, and we shall write $A:=\psi(a)\psi(b)^{-1}$, $B:=\psi(a)^2\psi(b)^{-1}$. Then for all $p,q\geq 3$, $\epsilon_0=\pm1$, after free reduction over $F(a, b)$ the word $B^{\epsilon_0 p}A^{3}B^{-\epsilon_0 q}$
has the form $B^{\epsilon_0 (p-1)}CB^{-\epsilon_0 (q-1)}$ for some word $C\in F(a, b)$. Moreover, the resulting word $B^{\epsilon_0 (p-1)}CB^{-\epsilon_0 (q-1)}$ does not contain $a^{\pm4}$, $b^{\pm4}$, $(ab)^{\pm3}$ or $(ba)^{\pm3}$ and so is Dehn reduced.
\end{lemma}

\begin{proof}\predisplaypenalty=9900
As with Lemma \ref{lem:automorphicorbit1}, the proof is by inspection of the appropriate words. Indeed, after free reduction over $F(a, b)$ the following words are all Dehn reduced for $p,q\geq3$, $\epsilon_0=\pm1$, with the longest subwords of relators occurring being $a^{\pm3}\leq a^{\pm i}$, $b^{\pm2}\leq b^{\pm j}$ and $(babab)^{\pm1}\leq (ab)^{\pm k}$. In each line the left-hand word represents $B^{\epsilon_0p}A^{3}B^{-\epsilon_0q}$ for the indicated map $\psi\in\Psi$ and the right-hand word represents the freely reduced form of the word. Each map has two cases: $\epsilon_0=1$ and $\epsilon_0=-1$.
%The left-hand column represents the words when $p>0$ and the right-hand column represents the words when $p<0$.
\begin{align*}
&\psi_{(1, 1)}:%, \epsilon_0=1:
&(a^2b^{-1})^{p}(ab^{-1})^3(a^2b^{-1})^{-q}&=(a^2b^{-1})^{p}(ab^{-1})^2a^{-1}(ba^{-2})^{q-1}\\
&%\psi_{(1, 1)}, \epsilon_0=-1:
&(a^2b^{-1})^{-p}(ab^{-1})^3(a^2b^{-1})^{q}&=(ba^{-2})^{p-1}ba^{-1}b^{-1}(ab^{-1})^2(a^2b^{-1})^{q}\\
&\psi_{(2, 1)}:%, \epsilon_0=1:
&(a^3b)^{p}(a^2b)^3(a^3b)^{-q}&=(a^3b)^{p}(a^2b)^2a^{-1}(b^{-1}a^{-3})^{q-1}\\
&%\psi_{(2, 1)}, \epsilon_0=-1:
&(a^3b)^{-p}(a^2b)^3(a^3b)^{q}&=(b^{-1}a^{-3})^{p-1}b^{-1}a^{-1}b(a^2b)^2(a^3b)^{q}\\
&\psi_{(3, 1)}:%, \epsilon_0=1:
&(abab^2)^{p}(ab^2)^3(abab^2)^{-q}&=(abab^2)^{p}ab^2aba^{-1}(b^{-2}a^{-1}b^{-1}a^{-1})^{q-1}\\
&%\psi_{(3, 1)}, \epsilon_0=-1:
&(abab^2)^{-p}(ab^2)^3(abab^2)^{q}&=(b^{-2}a^{-1}b^{-1}a^{-1})^{p-1}b^{-2}a^{-1}b(ab^2)^2(abab^2)^{q}
\end{align*}
Now, if the result holds for $\psi_{(l, 1)}$, $l\in\{1, 2, 3\}$ then it also holds for $\psi_{(l+3, 1)}$ (as $\psi_{(l, 1)}$ and $\psi_{(l+3, 1)}$ differ just by a switch of $a$ and $b$, and possibly by an inner automorphism). Further, if the result holds for $\psi_{(l, 1)}$ then it also holds for $\psi_{(l, -1)}$. Hence, the proof is complete.
\end{proof}

We now combine Lemmas \ref{lem:automorphicorbit1} and \ref{lem:automorphicorbit2} as follows.
%Recall that the notation $W_{\mathbf{x}}(\mathbf{s})$ means that we are freely reducing a word in the ambiant free group $F(\mathbf{x})$. Therefore,
Note that the word $W_{\{a, b\}}(\psi(\mathbf{s}))$ corresponds to the image of the word $W_{\mathbf{s}}(\mathbf{s})$ under the map $\psi$ after applying free reduction in $F(a, b)$.
%Then the image of the word $W_{\{x, y\}}(x, y)$ under the map $\psi$ is, after free reduction in $F(a, b)$, Dehn reduced.

\begin{lemma}
\label{lem:automorphicorbit3}
Let $T_{i, j, k}=\<a, b; a^i, b^j, (ab)^k\>$, with $i, j, k\geq6$, and let $\mathbf{s}:=\{x, y\}$. If $\psi$ is a map contained in the set $\Psi$ and $W_{\mathbf{s}}(\mathbf{s})$ is a word over $\mathbf{s}$ which is freely reduced over $\mathbf{s}$ then the word $W_{\{a, b\}}(\psi(\mathbf{s}))$ is Dehn reduced.
%Then the image of the word $W_{\{x, y\}}(x, y)$ under the map $\psi$ is, after free reduction in $F(a, b)$, Dehn reduced.
%, and we shall write $A:=\psi(a)\psi(b)^{-1}$, $B:=\psi(a)^2\psi(b)^{-1}$. Then for all $p, q\geq 3$ the word $B^pA^3B^q$ is freely reduced and does not contain $a^{\pm4}$, $b^{\pm4}$, $((ab)^3a)^{\pm1}$ or $(b(ab)^3)^{\pm1}$, and so is Dehn reduced.
\end{lemma}

\begin{proof}
Write $A:=\psi(a)\psi(b)^{-1}$ and $B:=\psi(a)^2\psi(b)^{-1}$. Write $X_{p_0, q_0}:=B^{p_0}A^3B^{q_0}$, $Y_{p_1, q_1}:=B^{p_1}A^3B^{-q_1}$ and $Z_{p_2, q_2}:=B^{-p_2}A^3B^{q_2}$, where $p_0, q_0\geq0$ and where $p_1, p_2, q_1, q_2\geq3$, and write
\[
\mathbf{s}_0:=\{X_{p_0, q_0}\mid p_0, q_0\geq0\}\cup \{Y_{p_1, q_1}, Z_{p_2, q_2}\mid p_i, q_i\geq3\}.
\]
Now, $W_{\psi(\mathbf{s})}(\psi(\mathbf{s}))$ is a product of the words $(A^3B^3A^3B^4\cdots A^3B^{\rho+2})^{\pm1}$ and $(A^3B^{\rho+3}A^3\cdots A^3B^{2\rho+2})^{\pm1}$. Therefore, there exists a word $U_{\mathbf{s}_0}(\mathbf{s}_0)$ such that $U_{\mathbf{s}_0}(\mathbf{s}_0)\equiv W_{\psi(\mathbf{s})}(\psi(\mathbf{s}))$,
%where no $Y_{p_1, q_1}^{\pm1}$ is adjacent to a $Z_{p_2, q_2}^{\pm1}$ in $U_{\mathbf{s}_0}(\mathbf{s}_0)$ and where $U_{\mathbf{s}_0}(\mathbf{s}_0)$ does not contain as a subword any of the following: $(X_{p_0, q_0}^{-1}Y_{p_1, q_1})^{\pm1}$, $(Y_{p_1, q_1}X_{p_0, q_0})^{\pm1}$, $(X_{p_0, q_0}Z_{p_2, q_2})^{\pm1}$, $(Z_{p_2, q_2}X_{p_0, q_0}^{-1})^{\pm1}$. Hence,
and where every subword of length two of $U_{\mathbf{s}_0}(\mathbf{s}_0)$ has the form either $(X_{p_0, q_0}Y_{p_1, q_1})^{\pm1}$, $(Y_{p_1, q_1}X_{p_0, q_0}^{-1})^{\pm1}$, $(X_{p_0, q_0}^{-1}Z_{p_2, q_2})^{\pm1}$, $(Z_{p_2, q_2}X_{p_0, q_0})^{\pm1}$, or $(X_{p_0, q_0}X_{p_0^{\prime}, q_0^{\prime}})^{\pm1}$. Note that $U_{\{a, b\}}(\mathbf{s}_0)\equiv W_{\{a, b\}}(\psi(\mathbf{s}))$.

%Consider $\psi(W)$ as a product of the words $A$ and $B$. This product can be partitioned into subwords of the form $B^{\pm p}$, $(B^pA^3B^q)^{\pm1}$, $(B^{p}A^3B^{-q})^{\pm1}$ and $(B^{-p}A^3B^{q})^{\pm1}$, $p, q\geq3$.
By Lemma \ref{lem:automorphicorbit1}, the words $X_{p_0, q_0}$ are freely reduced over $F(a, b)$ and are Dehn reduced.
By Lemma \ref{lem:automorphicorbit2}, after free reduction in $F(a, b)$ all words $Y_{p_1, q_1}$ and $Z_{p_2, q_2}$ are Dehn reduced and of the form respectively $B^{p_1-1}C_1B^{-(q_1-1)}$ and $B^{-(p_2-1)}C_2B^{q_2-1}$ for some fixed words $C_1, C_2\in F(a, b)$.
Then, applying Lemma \ref{lem:automorphicorbit1} again, the following words are freely reduced over $F(a, b)$ and Dehn reduced:
\[
(X_{p_0, q_0}B^{p_1-1}C_1B^{-(q_1-1)})^{\pm1}, (B^{p_1-1}C_1B^{-(q_1-1)}X_{p_0, q_0}^{-1})^{\pm1},
\]
\[
(X_{p_0, q_0}^{-1}B^{-(p_2-1)}C_2B^{q_2-1})^{\pm1}, (B^{-(p_2-1)}C_2B^{q_2-1}X_{p_0, q_0})^{\pm1}, (X_{p_0, q_0}X_{p_0^{\prime}, q_0^{\prime}})^{\pm1}
\]
These words represent the length-two subwords of $U_{\mathbf{s}_0}(\mathbf{s}_0)$ after free reduction in $F(a, b)$, and hence the word $U_{\{a, b\}}(\mathbf{s}_0)$ is Dehn reduced. As $U_{\{a, b\}}(\mathbf{s}_0)\equiv W_{\{a, b\}}(\psi(\mathbf{s}))$, the words $W_{\{a, b\}}(\psi(\mathbf{s}))$ is Dehn reduced as required.
\end{proof}

%%%%%%%%%%%%%%%------------------------%%%%%%%%%%%%%%%
%%%%%%%%%%%%%%%Proof of Theorem%%%%%%%%%%%%%%%
%%%%%%%%%%%%%%%------------------------%%%%%%%%%%%%%%%

\p{Examples of malcharacteristic subgroups}
We are finally ready to prove that the subgroup $M$ of $T_{i,j,k}$, $i, j, k\geq6$, is malcharacteristic. The words $x$ and $y$ in Lemma \ref{lem:malcharTriange} are the same words defined at the end of Section \ref{sec:Malchar} and at the start of Section \ref{sec:MalcharTriangle}.

\begin{lemma}
\label{lem:malcharTriange}
Let $M$ be the subgroup of $T_{i, j, k}=\langle a, b; a^i, b^j, (ab)^k\rangle$, $i, j, k\geq6$, which is generated by the following elements, with $\rho\gg i, j ,k$.
\begin{align*}
x&:=(ab^{-1})^{3}(a^2b^{-1})^{3}(ab^{-1})^3(a^2b^{-1})^4\ldots (ab^{-1})^3(a^2b^{-1})^{\rho+2}\\
y&:=(ab^{-1})^3(a^2b^{-1})^{\rho+3}(ab^{-1})^3(a^2b^{-1})^{\rho+4}\ldots (ab^{-1})^{3}(a^2b^{-1})^{2\rho+2}
\end{align*}
Then $M$ is a malcharacteristic subgroup of $T_{i, j, k}$ and is free of rank two.
\end{lemma}

\begin{proof}
Write $\widetilde{M}:=\langle x, y\rangle$ for the subgroup of $F(a, b)$ generated by the words $x$ and $y$. Then $\widetilde{M}$ is in the same $\aut(F(a, b))$-orbit as the subgroup $L$ defined in Lemma \ref{lem:malcharFree}. Therefore, by Lemma \ref{lem:malcharFree}, the subgroup $\widetilde{M}$ is malcharacteristic in $F(a, b)$. We use this to prove that $M$ is malcharacteristic in $T_{i, j, k}$.

Now, by Lemma \ref{lem:malnormal}, it is sufficient to prove that if $\delta\in\aut(T_{i, j, k})$ is such that $\delta(M)\cap M\neq 1$ then $\delta\in\inn(T_{i, j, k})$. So, suppose that there exists an automorphism $\delta\in\aut(T_{i, j, k})$ and two words $U$ and $V$ such that $U(x, y)=V(\delta(x), \delta(y))$, and we wish to prove that $\delta$ is inner. By Lemma \ref{lem:transversal} there exist $\psi\in\Psi$ and $\gamma_g\in\inn(T_{i, j, k})$, $g\in T_{i, j, k}$, such that $\delta=\psi\gamma_g$, so $\psi\gamma_g(M)\cap M\neq1$. Now, after free reduction over $F(a, b)$ every word in $\psi(M)$ is Dehn reduced, by Lemma \ref{lem:automorphicorbit3}.
It follows that every word in $\psi(M)$ is cyclically Dehn reduced (as every cyclic shift of a word $U\in\psi(M)$ is a subword of $U^2$, and $U^2\in\psi(M)$ is Dehn reduced).
Hence, we may apply Theorem \ref{thm:FreeMalnormMetric4T4}. Therefore, there exists a word $W\in F(a, b)$ such that $\widetilde{\psi}\gamma_W\left(\widetilde{M}\right)\cap \widetilde{M}\neq_{F(a, b)}1$, where $\widetilde{\psi}$ is the automorphism of $F(a, b)$ defined using the same words as $\psi$. Now $\widetilde{\psi}\not\in\inn(F(a, b))$ so $\widetilde{M}$ is not malcharacteristic in $F(a, b)$, a contradiction.
%
%Lemma \ref{lem:reductiontolift}, that for $\phi$ is an arbitrary automorphism of $T_i$, if the words $U(x, y)$ and $V(x\phi, y\phi)$ are conjugate in $T_i$ then they are freely conjugate in $F(a, b)$. Then, because every automorphism of $T_i$ lifts to an automorphism of the ambient free group $F(a, b)$, in order to prove that $M$ is malcharacteristic in $T_i$ it is sufficient to prove that the lift $\widetilde{M}$ of $M$ to $F(a, b)$ is malcharacteristic in $F(a, b)$, and we do this in Lemma \ref{lem:malcharFree}.
\end{proof}

If any of $i$, $j$ or $k$ in $T_{i, j, k}$ are less than $6$ then it is not clear that the subgroup $M$ is malcharacteristic. The precise issues are found in Lemma \ref{lem:ExtendingZieschang} and in Lemmas \ref{lem:automorphicorbit1} and \ref{lem:automorphicorbit2}, whose proofs each require $i, j, k\geq6$. Because of the small cancellation arguments we apply, it is likely that Lemma \ref{lem:malcharTriange}, and hence Theorem \ref{thm:intro1} and Theorem \ref{thm:GeneralTriangle1}, can be extended to $i, j, k>6$ or even to $i, j, k>4$. To develop our results in this direction, first Lemma \ref{lem:ExtendingZieschang} would need to be extended appropriately. To extend Lemmas \ref{lem:automorphicorbit1} and \ref{lem:automorphicorbit2} it is likely that new words $x, y\in F(a, b)$ would need to be found.
%If Lemma \ref{lem:ExtendingZieschang} could be extended to hold for $i, j, k>6$, or even $i, j, k>4$, then in order to extend Lemma \ref{lem:malcharTriange}, and hence Theorem \ref{thm:intro1}, to these value of $i, j, k$ then we would require new words $x$ and $y$ which would allow Lemmas \ref{lem:automorphicorbit1} and \ref{lem:automorphicorbit2} to be extended.
%In particular, the arguments underlying Lemma \ref{lem:malcharTriange} reply on the $C^{\prime}(1/4)-T(4)$ small cancellation condition (and not the stronger $C^{\prime}(1/8)$ condition suggested by the requirement of $i, j, k>8$). Lemma \ref{lem:transversal} suggests we probably require $i, j, k\geq6$. Therefore, choosing different words $x$ and $y$ may allow for Lemma \ref{lem:malcharTriange} to be extended to the groups $T_{i,j ,k}$ with $i, j, k\geq6$, and hence for Theorem \ref{thm:intro1} to be extended to the groups $T_i$, $i\geq6$.

%The stated subgroup $M$ of $T_{i, j, k}$ is free of rank two. Therefore, by Lemma \ref{lem:malcharArbRank}, we have the following result.

Lemmas \ref{lem:malcharArbRank} and \ref{lem:malcharTriange} combine to prove the following result.

\begin{proposition}
\label{lem:malcharTriangleArbRank}
A triangle group $T_{i, j, k}$, $i, j, k\geq6$, contains malcharacteristic subgroups $M_n$ of arbitrary rank $n$ (possibly countably infinite).
\end{proposition}

Indeed, the subgroup $M_n$ is generated by $n$ finite subwords satisfying $C^{\prime}(1/6)$ of the following infinite word:
\[
{x}{y}({x}{y}^2){x}{y}({x}{y}^2)^2{x}{y}({x}{y}^2)^3\cdots.
\]

\p{Proofs of Theorems \ref{thm:intro1}, \ref{thm:intro2}, \ref{thm:intro3} and \ref{thm:intro4}}
We first prove Theorem \ref{thm:intro1}.

\begin{proof}[Proof of Theorem \ref{thm:intro1}]
By Lemma \ref{lem:malcharTriange}, the group $T_i$ contains a malcharacteristic subgroup which is free of rank two. As $T_i$ has Serre's property FA \cite[Example 6.3.5]{trees} and as $\phi: a\mapsto b$, $b\mapsto b^{-1}a^{-1}$ defines a non-inner automorphism of $T_i$ which has order three, the result follows from Theorem \ref{thm:mainconstruction}.
\end{proof}

We now prove Theorem \ref{thm:intro2}.

\begin{proof}[Proof of Theorem \ref{thm:intro2}]
%Note that every object $P/N_k\in\mathbf{P}$ corresponds to a quotient presentation $\mathcal{P}_k$ of $\mathcal{P}$ in an obvious way, and that this correspondence is an isomorphism of categories. Suppose we have $P/N_j\rightarrow P/N_k$.
%If there exists a morphism $\mathcal{P}_j\rightarrow\mathcal{P}_k$ then there exists a presentation $\mathcal{P}_j^{\prime}$ such that $\mathcal{P}_j\sim\mathcal{P}_j^{\prime}$ and such that $\mathcal{P}_k$ is a quotient presentation of $\mathcal{P}_j^{\prime}$.
%Then $T_{\mathcal{P}_j^{\prime}}\twoheadrightarrow T_{\mathcal{P}_k}$ by Theorem \ref{thm:mainconstruction}.\ref{point:mainconstruction3}, and the result follows as $T_{\mathcal{P}_j^{\prime}}=T_{\mathcal{P}_j}$ by Theorem \ref{thm:mainconstruction}.\ref{point:mainconstruction2}.
%If there exists a morphism $\mathcal{P}_j\rightarrow\mathcal{P}_k$ then $\mathcal{P}_k$ is a quotient presentation of $\mathcal{P}_j$. Therefore, $T_{\mathcal{P}_j}\twoheadrightarrow T_{\mathcal{P}_k}$ by Theorem \ref{thm:mainconstruction}.\ref{point:mainconstruction3}.
Theorem \ref{thm:intro2} follows immediately from Theorem \ref{thm:FunctorialProperties}.
\end{proof}

We now prove Theorem \ref{thm:intro3}.

\begin{proof}[Proof of Theorem \ref{thm:intro3}]
Theorem \ref{thm:intro3} follows immediately from Theorem \ref{thm:ResidualProperties}.
\end{proof}

We now prove Theorem \ref{thm:intro4}.

\begin{proof}[Proof of Theorem \ref{thm:intro4}]
Theorem \ref{thm:intro4} follows immediately from Theorem \ref{thm:FreeProduct}.
\end{proof}

\p{The triangle groups \boldmath{$T_{i, j, k}$}}
Theorems \ref{thm:GeneralTriangle1}--\ref{thm:GeneralTriangle4} prove analogous results to Theorems \ref{thm:intro1}, \ref{thm:intro2}, \ref{thm:intro3} and \ref{thm:intro4} where the base group $H$ is a fixed triangle group $T_{i, j, k}$ rather than an equilateral triangle group $T_i$.

\begin{theorem}
\label{thm:GeneralTriangle1}
Fix a (non-equilateral) triangle group $T_{i, j, k}$ with $i, j, k\geq6$.
Every countable group $Q$ can be realised as an index-one or-two subgroup of the outer automorphism group of an automorphism-induced HNN-extension $T_{\mathcal{P}}^{i, j, k}$ of $T_{i, j, k}$. Moreover, $\aut(T_{\mathcal{P}}^{i, j, k})$ has an index-one or-two subgroup which splits as $T_{\mathcal{P}}^{i, j, k}\rtimes Q$.
%For every countable group presentation $\mathcal{Q}$ with at least one generator there exists an automorphism-induced HNN-extension $H_{\mathcal{Q}}$ of $H$ such that $\out(H_{\mathcal{Q}})\cong Q$ and $\aut(H_{\mathcal{Q}})\cong H_{\mathcal{Q}}\rtimes Q$, where $Q:=\pi_1(\mathcal{Q})$.
\end{theorem}

\begin{proof}
By Lemma \ref{lem:malcharTriange}, the group $T_{i, j, k}$ contains a malcharacteristic subgroup which is free of rank two. As $T_{i, j, k}$ has Serre's property FA \cite[Example 6.3.5]{trees} and as $\phi: a\mapsto a^{-1}$, $b\mapsto b^{-1}$ defines a non-inner automorphism of $T_{i, j, k}$, the result follows from Theorem \ref{thm:mainconstruction}.
\end{proof}

The proofs of Theorems \ref{thm:GeneralTriangle2}--\ref{thm:GeneralTriangle4} are identical to the proofs of Theorems \ref{thm:intro2}, \ref{thm:intro3} and \ref{thm:intro4}, and hence are omitted.

\begin{theorem}
\label{thm:GeneralTriangle2}
%Fix a presentation $\mathcal{P}$ of a group $P$. Let $\mathcal{P}_k$ be a quotient presentation of $\mathcal{P}$ and let $N_k$ denote the kernel of the induced quotient map $\pi_1(\mathcal{P})\twoheadrightarrow\pi_1(\mathcal{P}_k)$. Then the map $\mathbf{P}\rightarrow\operatorname{Grp}$ given by $P/N_k\mapsto T_{\mathcal{P}_k}^{i, j, k}$ is a functor.
The map defined by $\mathcal{P}\mapsto T^{i, j, k}_{\mathcal{P}}$ is a functor from the category of countable group presentations $\operatorname{Pres}$ to the category of groups $\operatorname{Grp}$.
\end{theorem}

\begin{theorem}
\label{thm:GeneralTriangle3}
Let $\mathbb{T}^{i, j, k}_{\operatorname{Fin}}$ be the class of groups $T^{i, j, k}_{\mathcal{P}}$ where $\pi_1(\mathcal{P})$ is finite. If $\pi_1(\mathcal{P})$ is residually finite then $T^{i, j, k}_{\mathcal{P}}$ is residually-$\mathbb{T}^{i, j, k}_{\operatorname{Fin}}$.
\end{theorem}

\begin{theorem}
\label{thm:GeneralTriangle4}
The subgroups $M_n$, $n\in\mathbb{N}$, in the proof of Theorem \ref{thm:GeneralTriangle1} may be chosen in such a way that for every presentation $\mathcal{P}$ with finite generating set and for every countable group presentation $\mathcal{Q}$ there exists a surjection $T_{\mathcal{P}}\twoheadrightarrow T_{\mathcal{P}\ast\mathcal{Q}}$.
\end{theorem}

\section{Residual finiteness}
\label{sec:RF}
Bumagin--Wise asked if every countable group $Q$ can be realised as the outer automorphism group of a finitely generated, residually finite group $G_Q$ \cite[Problem 1]{BumaginWise2005}. We now prove Corollary \ref{corol:intro3} (which is a corollary of Theorem \ref{thm:intro3}). Corollary \ref{corol:intro3} answers this question of Bumagin--Wise for all finitely generated, residually finite groups $Q$ by taking $G_Q:=T_{\mathcal{P}}$ for $\mathcal{P}$ a presentation of $Q$ with finite generating set.
%Recall that before Theorem \ref{thm:ResidualProperties} we pointed out that if ${\mathcal{P}}=\langle \mathbf{x}; \mathbf{r}\rangle$ has infinite generating set, so $|\mathbf{x}|=\infty$, then the associated subgroups of the quotient groups $H_{\mathcal{P}_g}$ are not finitely generated.

%The reason that Corollary \ref{corol:intro3} does not extend to all residually finite groups, including those which cannot be finitely generated, is as follows: Let $\mathbb{T}_{\operatorname{Fin}}^{\infty}$ because if $\pi_1(\mathcal{P})$ cannot be finitely generated then the group $T_{\mathcal{P}}$ is actually resi

\begin{proof}[Proof of Corollary \ref{corol:intro3}]
We write $h:=T_i$. By Theorem \ref{thm:intro3} it is sufficient to prove that if $\mathcal{P}=\langle \mathbf{x}; \mathbf{r}\rangle$ has finite generating set $\mathbf{x}$ and $\pi_1(\mathcal{P})$ is finite then $T_{\mathcal{P}}$ is residually finite. Under these conditions the associated subgroup $K_{\mathcal{P}}$ of the group $T_{\mathcal{P}}=H\ast_{(K_{\mathcal{P}}, \phi)}$ has finite index in the subgroup $M_{|\mathbf{x}|}$. As $M_{|\mathbf{x}|}$ is a free group of finite rank ${|\mathbf{x}|}$, the subgroup $K_{\mathcal{P}}$ is finitely generated.

Now, triangle groups are LERF \cite{Scott1978Subgroups}. Hence, $K_{\mathcal{P}}$ is separable in the triangle group $H$. Therefore, the automorphism-induced HNN-extension $T_{\mathcal{P}}$ is residually finite \cite[Lemma 4.4]{BaumslagTretkoff}, as required.
%
%\marginpar{Trick is: If $x\in K_{\mathcal{P}}N$ for all $N\leq_fH$, then there exists $M_n\leq M_{\ifty}$ of rank $n$ and $K_n:=K_{\mathcal{P}}\cap M_n$ such that $x\in K_nM_n$...hmm...issue is that $k$ is not necessarily fixed...}
\end{proof}

In order to extend Corollary \ref{corol:intro3} to all countable residually finite groups it would be necessary to prove that if $\mathcal{P}=\langle \mathbf{x}; \mathbf{r}\rangle$ is a presentation of a finite group where $|\mathbf{x}|=\infty$ then the group $T_{\mathcal{P}}$ is residually finite. (In Corollary \ref{corol:intro3} we explicitly use the fact that $\mathbf{x}$ is a finite set.)
%Note that if the presentation $\mathcal{P}$ of a finite group $Q$ has infinite generating set then the subgroup ${K}_{\mathcal{P}}$ of the triangle group $H$ is not finitely generated, and hence not necessarily separable in $H$.
We therefore have the following corollary of Theorem \ref{thm:intro3}.
%Here, the group $G_Q$ is a group $T_{\mathcal{Q}}$ for an arbitrary presentation $\mathcal{Q}$ of $Q$.

\begin{corollary}
\label{corol:residualFiniteness}
%Suppose that if $\mathcal{P}$ is a countable presentation of a finite group then the group $T_{\mathcal{P}}$ constructed by Theorem \ref{thm:intro1} is residually finite.
Let $\mathbb{T}_{\operatorname{Fin}}^{\infty}$ be the class of groups $T_{\mathcal{P}}$ where $\pi_1(\mathcal{P})$ is finite but where $\mathcal{P}$ has infinite generating set.
Suppose that every element of $\mathbb{T}_{\operatorname{Fin}}^{\infty}$ is residually finite.
Then for every countable, residually finite group $Q$ there exists
%a presentation $\mathcal{Q}$ of $Q$ such that the group $H_{\mathcal{Q}}$ constructed by Theorem \ref{thm:intro1} is residually finite. Therefore, for every countable residually finite group $Q$ there exists
a finitely generated, residually finite group $G_Q$ such that $\out(G_Q)\cong Q$.
\end{corollary}

Note that Corollary \ref{corol:residualFiniteness} represents the best application of Theorem \ref{thm:intro1} in this direction, in the sense that if $T_{\mathcal{P}}$ is residually finite then $Q:=\pi_1(\mathcal{P})$ is residually finite. This is because $Q$ embeds in $\aut(T_{\mathcal{P}})$, so if $Q$ is not residually finite then $T_{\mathcal{P}}$ is not residually finite \cite{Baumslag1963automorphisms}.

\section{Questions}
\label{sec:Qns}

We now pose certain questions which arose in the writing of this paper.

\p{Injectivity of the functor}
The functor of Theorem \ref{thm:intro2} should be injective on both objects and morphisms. Injectivity of the functor is equivalent to a positive answer to the following question, where $\sim$ is the equivalence relation on presentations defined in Section \ref{sec:Construction}.
\begin{question}
Is it true that $T_{\mathcal{P}}\cong T_{\mathcal{Q}}$ if and only if $\mathcal{P}\sim\mathcal{Q}$?
\end{question}
It is unlikely that the functor of Theorem \ref{thm:FunctorialProperties} is injective for all choices of groups $H$ and automorphisms $\phi\in\aut(H)$.
\begin{question}
For which base groups $H$ and automorphisms $\phi\in\aut(H)$ does it hold that in the construction of Theorem \ref{thm:mainconstruction}, $H_{\mathcal{P}}\cong H_{\mathcal{Q}}$ if and only if $\mathcal{P}\sim\mathcal{Q}$?
\end{question}

\p{Residual finiteness}
Question \ref{Q:RF} leads naturally on from Corollary \ref{corol:intro3} and Corollary \ref{corol:residualFiniteness}.
\begin{question}\label{Q:RF}
Suppose $\mathcal{P}=\langle \mathbf{x}; \mathbf{r}\rangle$ is a presentation of a finite group such that the set $\mathbf{x}$ is infinite. Is the group $T_{\mathcal{P}}$ from Theorem \ref{thm:intro1} residually finite?
\end{question}

\p{Free groups}
%In Section \ref{sec:MalcharF2} we proved that $F_2$, the free group of rank two, contains a malcharacteristic subgroup which is free of rank two; it then follows from Lemma \ref{lem:malcharArbRank} that $F_2$ contains malcharacteristic subgroups of arbitrary rank.
In Section \ref{sec:MalcharF2} we proved that $F_2$, the free group of rank two, contains a malcharacteristic subgroup $L$ with $L\cong F_2$; it then follows from Lemma \ref{lem:malcharArbRank} that $F_2$ contains malcharacteristic subgroups of arbitrary rank.
\begin{question}
Do the free groups $F_n$, $n>2$, each contain a malcharacteristic subgroup which is free of rank two?
\end{question}

Malnormality is decidable in free groups. In Lemma \ref{lem:malcharlem} we gave conditions relating to the malcharacteristic property being decidable in $F_2$.
\begin{question}
Is it decidable if a subgroup $M$ of a free group $F_n$, $n\geq2$, is malcharacteristic?
\end{question}

Malnormality is a generic property of subgroups of free groups: using the ``few relator'' model for random groups we see that a random finite set of words $\mathbf{s}\subset F_n$ contains no proper powers and satisfies the classical small cancellation conditions \cite[Proposition 2]{Ollivier2005invitation}. It then follows that $\mathbf{s}$ is malnormal in $F_n$ with free basis $\mathbf{s}$ \cite[Theorems~2.11~\&~2.14]{wise2001residual}.
\begin{question}
\label{Qn:generic}
Does a generic set of words $\mathbf{s}\subset F_n$, $n\geq2$, generate, with overwhelming probability, a malcharacteristic subgroup of $F_n$?
\end{question}
Note that Question \ref{Qn:generic} admits a positive answer if we assume $|\mathbf{s}|=1$ \cite[Theorem A]{Kapovich2006Generic}.

\p{Other groups}
As was mentioned in the introduction, we expect Theorem \ref{thm:intro1} to hold for many other classes of groups and not just equilateral triangle groups. Now, Thompson's group $V$ has Serre's property FA \cite{farley2011proof} and has outer automorphisms of any given order \cite{bleak2016further}. Therefore, a positive answer to the following question would allow us to apply Theorem \ref{thm:mainconstruction} to $V$, obtaining a result analogous to Theorem \ref{thm:intro1}.
\begin{question}
Does Thompson's group $V$ contain a malcharacteristic subgroup which is free of rank two?
\end{question}
Thompson's group $T$ also has Serre's property FA \cite{farley2011proof}, but $\out(T)$ is cyclic of order two \cite{Brin1996chameleon}. Therefore, a positive answer to the following question would allow us to apply Theorem \ref{thm:mainconstruction} to $T$, obtaining a result analogous to Theorem \ref{thm:GeneralTriangle1}.
\begin{question}
Does Thompson's group $T$ contain a malcharacteristic subgroup which is free of rank two?
\end{question}
Note that we ask about $T$ and $V$ and not Thompson's group $F$ because it is well-known that $F$ does not contain a non-abelian free group.
%Also note the Brin-Thompson groups $nV$ have Serre's property FA \cite{kato2015higher} but nothing is known about their outer automorphism groups.

\bibliographystyle{amsalpha}
\bibliography{BibTexBibliography}
\end{document}